\documentclass[10pt]{amsart}
\usepackage{amscd, amssymb}
\usepackage{pb-diagram}
\usepackage{enumerate}
\usepackage[all]{xypic}
\usepackage{multirow}

\theoremstyle{plain}
\newtheorem{Thm}[equation]{Theorem}

\newtheorem{Prop}[equation]{Proposition}
\newtheorem{Lem}[equation]{Lemma}

\newtheorem{Rmk}[equation]{Remark}
\newtheorem{Conj}[equation]{Conjecture}
 
\numberwithin{equation}{section}

\newcommand{\OO}{\operatorname{O}}

\newcommand{\U}{\operatorname{U}}

\newcommand{\Mp}{\operatorname{Mp}}

\newcommand{\Sw}{\mathcal{S}}
\newcommand{\Aut}{\mathcal{A}}

\newcommand{\Ik}{\operatorname{Ik}}

 \newcommand{\im}{\operatorname{Im\,}}

\newcommand{\Val}{\operatorname{Val}}

\newcommand{\Sp}{\operatorname{Sp}}
\newcommand{\Ind}{\operatorname{Ind}}

\newcommand{\Hom}{\operatorname{Hom}}

\newcommand{\GL}{\operatorname{GL}}

\newcommand{\C}{\mathbb C}
\newcommand{\A}{\mathbb{A}}
\newcommand{\Q}{\mathbb{Q}}

\newcommand{\Z}{\mathbb{Z}}
\newcommand{\R}{\mathbb{R}}

\newcommand{\bm}{\begin{multline*}}
\newcommand{\tu}{\end  {multline*}}

\newcommand{\disc}{{\rm disc}}

\newcommand{\smalltwomatrix[5]}{
  \left( \begin{smallmatrix}
            {#2} & {#3}\\
            {#4} & {#5}
         \end{smallmatrix}\right)}

\newcommand{\leftup}[2]{{^{#1}}\mspace{-2.5mu}#2}

\title[Siegel-Weil and Rallis]{The Regularized Siegel-Weil Formula \\
(The Second Term Identity)\\ and the Rallis Inner Product Formula}
\author{Wee Teck Gan, Yannan Qiu and Shuichiro Takeda}
\address{W. T. Gan and Y. Qiu: National University of Singapore,
21 Lower Kent Ridge Road,
Singapore, 119077 }
\email{matgwt@nus.edu.sg}

\email{matqy@nus.edu.sg}

\address{S. Takeda: Mathematics Department, University of Missouri, Columbia, 202
  Math Sciences Building, Columbia, MO, 65211}
\email{takedas@missouri.edu}

\dedicatory{
In memory of a pioneer \\
Steve Rallis (1942-2012)}
\begin{document}
\maketitle
 
\begin{abstract}
In this paper, we establish the second term identity of the
Siegel-Weil formula in full generality, and derive the Rallis inner
product formula for global theta lifts for any dual pair. As a
corollary, we resolve the non-vanishing problem of global theta
lifts initiated by Steve Rallis.
\end{abstract}

\section{\bf Introduction}

About 30 years ago, Steve Rallis (\cite{R1,R2,R3}) initiated a program to
understand the cuspidality and non-vanishing of global theta
liftings. We take a few moments to describe Rallis' program and his
fundamental contributions to it.
\vskip 5pt

\subsection{\bf Theta correspondence.}
Let $F$ be a number field with the ring of adeles $\A$, and let $E$ be
either $F$ or a quadratic extension of $F$.  With $\epsilon = \pm1$,
let $U_n$ be an  $n$-dimensional $-\epsilon$-Hermitian space   over
$E$, and let $V_r$ be an
$m$-dimensional $\epsilon$-Hermitian space of Witt index $r$. 
Then one has an associated  reductive dual pair $G(U_n) \times
H(V_r)$, where $G(U_n)$ is the isometry group of $U_n$ (or a covering
thereof).    The group $G(U_n)(\A) \times H(V_r)(\A)$ has a Weil
representation $\omega$ (depending on some other auxiliary data), and
one has an automorphic realization 
\[  
\theta:  \omega \longrightarrow \{\text{Functions on $[G(U_n)] \times
  [H(V_r)]$} \} 
\]
where we have written $[G(U_n)]$ for $G(U_n)(F) \backslash G(U_n)(\A)$.
If $\pi$ is a cuspidal automorphic representation of $G(U_n)(\A)$,
then the global theta lift $\Theta(\pi)$  of $\pi$ to $H(V_r)$ is the
automorphic representation of $H(V_r)$ spanned by the functions
\[  
\theta(\phi, f)(h) = \int_{G(U_{n})(F) \backslash G(U_{n})(\A)}
\theta(\phi)(g,h) \cdot \overline{f(g)} \, dg 
\]
where $f \in \pi$, $\phi \in \omega$ and $dg$ is the Tamagawa measure.  
\vskip 5pt

The main problem in the theory of theta correspondence  is to
investigate the cuspidality and non-vanishing of $\Theta(\pi)$. The
cuspidality issue was quickly handled by Rallis in \cite{R2} where he
discovered the so-called tower property. Thus the remaining issue is
the question of nonvanishing. 
\vskip 5pt

\subsection{\bf Rallis' program.} 
 The goal of Rallis' program is to obtain a local-global criterion for
 the non-vanishing of the global theta lifting $\Theta(\pi)$, with a
 prototype statement: 
\begin{itemize}
\item {\it The global theta lifting $\Theta(\pi)$ to $H(V_{r})(\A)$ of
    a cuspidal representation $\pi$ on $G(U_{n})(\A)$  is nonzero}
\end{itemize}
if and only if
 \begin{itemize}
\item {\it the local theta liftings to $H(V_{r})(F_v)$ of $\pi_v$ on
    $G(U_{n})(F_v)$  are nonzero for all places $v$ of $F$}; 
\item  {\it the standard L-function $L(s,\pi)$  of $\pi$ is
    nonvanishing or has a pole at a distinguished point $s_0$.}
\end{itemize}
 Besides drawing from his earlier work (\cite{RS1, RS2, RS3}) with
 G. Schiffmann, Rallis was motivated by the then-recently-appeared
 results of Waldspurger \cite{W} who proved such a result for the theta
 correspondence between $\OO_3$ and $\Mp_2$. Moreover, Rallis was able
 to carry out this program for the theta correspondence between
 $\OO_n$ and $\Mp_2$ in \cite{R2, R3}. 

\vskip 10pt

\subsection{\bf The Rallis Inner Product Formula.} 
The cornerstone of Rallis' program is the so-called {\it Rallis inner
  product formula}. Namely, one may determine the non-vanishing of
$\Theta(\pi)$ by computing the Petersson inner product
$\langle \theta(\phi, f), \theta(\phi,f) \rangle$. The Rallis inner
product formula relates this inner product to the $L$-values of $\pi$.  
\vskip 5pt

The mechanism for the Rallis inner product formula relies on the
following see-saw diagram of dual pairs:
\[
    \xymatrix{
    G(W_n)\ar@{-}[d]_{i}\ar@{-}[dr]&H(V_r)\times H(V_r)\ar@{-}[d]\\
    G(U_n)\times G(U^-_n)\ar@{-}[ur]&H(V_r)^{\Delta},
    }
\]
where $U^{-}_n$ denotes the $-\epsilon$-Hermitian space obtained from
$U_n$ by multiplying the form by $-1$, so that $G(U^-_n) = G(U_n)$,
and $W_n$ (to be read ``doubled-U") denotes the space $U_n + ( U^-_n)$
.  The resulting see-saw identity reads:  
\begin{align}  \label{A:inner}
 & \langle \theta(\phi_1, f_1), \theta(\phi_2,f_2) \rangle  \notag \\
 = & \int_{[H(V_r)]} \left( \int_{ [G(U_n)]} \theta(\phi_1)(g_1,h)
   \cdot \overline{f_1(g_1)} \, dg_1 \right) \cdot 
\left(  \int_{ [G(U_n)]} 
\overline{ \theta(\phi_2)(g_2,h)} \cdot f_2(g_2) \, dg_2 \right) \, dh \notag \\
= & \int_{ [G(U_n) \times  G(U_n)]} \left( \int_{ [H(V_r)]}
  \theta(\phi_1)(g_1, h) \cdot \overline{ \theta(\phi_2) (g_2,h)} \,
  dh \right) \cdot  
  \overline{f_1(g_1)} \cdot f_2(g_2) \, dg_1 \, dg_2 
\end{align}
where in the last equality, we have formally exchanged the integrals. 
\vskip 5pt

In order to justify the above exchange and to relate the last expression above to L-functions, 
Rallis was led to develop, in collaboration with S. Kudla and
Piatetski-Shapiro, several key ingredients. We recall these key
ingredients in turn and some recent developments concerning them.
\vskip 10pt

\subsection{\bf  A regularized Siegel-Weil formula.} 
The Siegel-Weil formula was discovered by Siegel in the context of
classical modular forms and then cast in the representation theoretic
language and considerably extended in an influential paper of Weil
\cite{We}. It identifies the global theta lift of the trivial
representation of $H(V_r)$ to
$G(W_n)$  (which is the inner integral in (\ref{A:inner}))  as an
Eisenstein series, at least when some convergence conditions are
satisfied. In a series of 3 papers \cite{KR1, KR2, KR5}, Kudla and Rallis
greatly extended the theory of the Siegel-Weil formula to situations
where these convergence conditions are not satisfied. Their work
culminates in a {\em regularized Siegel-Weil formula}, and they
established what is now known as the {\em first term identity}, at
least when $G(W_n)$ is symplectic and $H(V_r)$ orthogonal. Their work
was subsequently refined and extended to other dual pairs  by others
(\cite{Ik, I1, I2, I3, Y2, Y3, Mo, JS}), especially in the work of Ikeda, Ichino and Yamana.
\vskip 5pt

In particular, the first term identity in the so-called first term
range (see below) is now completely established. Thus, one has a
re-interpretation of the inner integral in (\ref{A:inner}) as a
special value or residue of a Siegel  Eisenstein series.

\vskip 10pt

\subsection{\bf  The theory of the doubling zeta integral.}  
The regularized Siegel-Weil formula led Piatetski-Shapiro and Rallis \cite{PS-R87}
to consider  the doubling zeta integral, which is the outer integral
in (\ref{A:inner}), with the inner integral replaced by a
Siegel-Eisenstein series on $G(W_n)$.
This  is a family of zeta integrals giving rise, \`a la Tate's thesis
and Godement-Jacquet \cite{GJ}, to a theory of the standard L-function
$L(s,\pi)$ for classical groups. In a paper of Lapid-Rallis \cite{LR}, the
local theory of the doubling zeta integral was worked out in full
detail. In particular, the local standard $\gamma$-factor (twisted by
$\GL_1$) was precisely defined, and characterized by a list of
properties (the Ten Commandments). Moreover, the local L-factor and
$\epsilon$-factor were defined from the $\gamma$-factor following a
procedure of Shahidi. These local L-factors and $\epsilon$-factors are
the ``right" ones, in the sense that they are compatible with those on
Galois side under the local Langlands correspondence. 
\vskip 5pt

In a recent paper \cite{Y4}, Yamana showed that the local L-factors
defined by Lapid-Rallis \cite{LR} are precisely the GCD's for the local
zeta integrals associated to the family of ``good" sections. This
implies that the analytic properties of the local zeta integrals are
precisely controlled by the standard L-factors. With this, the theory
of the doubling zeta integral is essentially complete.
\vskip 5pt

Globally, the study of the doubling zeta integral leads to a precise
understanding of the analytic properties of the standard L-functions
of classical groups, such as the possible location of poles of these
L-functions (\cite{KR3, T2, Y4}).

\vskip 10pt

\subsection{\bf  Local theta correspondence.} In the course of
establishing the Siegel-Weil formula and in the application of the
theory of the doubling zeta integrals, 
Kudla and Rallis resolved many of the local problems  in the theory of the
local theta correspondence. In a series of papers \cite{KR3,KR4,KR6,
  KS, S1, S2}, they, together with Sweet, completely determined the
structure of certain degenerate principal series representations and
described their constituents in the framework of local theta
correspondence in the $p$-adic case. Their work partly motivated and was
complemented by a similar analysis in the archimedean case, which was
carried out in \cite{L1, L2, HL, LZ1, LZ2, LZ3}.
This, together with their study of the local doubling zeta integral,
led Kudla and Rallis to formulate a conjecture known as the conservation
relation, which is an identity for the sum of  the first occurrences
of local theta correspondence in two different Witt towers. In
\cite{KR6}, they made a key progress towards this conjecture by proving one
inequality of this identity. The reverse inequality was shown in a
recent preprint \cite{SZ} of Sun and Zhu, thus completing the proof of the
conservation conjecture.
\vskip 10pt

\subsection{\bf First Term Range.}
The above developments imply that the local issues involved in the
Rallis' program are completely settled (except for a certain subtle issue
for some of the real cases, which will be discussed in the final section of this paper). In particular,
the above achievements culminate in a Rallis inner product formula for
the theta lift from $G(U_n)$ to $H(V_r)$ in the following cases:
\begin{itemize}
\item (Weil's convergent range) $r=0$ or $m-r > n+\epsilon_0$ (\cite{KR1, Li, I3, Y4});
\item  (First term range)  $r > 0$ and  $m
  \leq n+ \epsilon_0$, so that  $m - r \leq n+ \epsilon_0$, (\cite{KR5, GT, Y4}), 
\end{itemize}
where we define
\[  
\epsilon_0  = \begin{cases}
0, \text{ if $E \ne F$;} \\
\epsilon, \text{  if $E = F$.}\end{cases} 
\] 
In the first term range, the Rallis inner product formula takes the rough  form:
\[  
\langle \theta(\phi_1,f_1), \theta(\phi_2,f_2) \rangle = c \cdot 
{\rm Res}_{s =-s_0} L(s + \frac{1}{2}, \pi)   \cdot Z^*(-s_0, \phi_1
\otimes \overline{\phi_2}, f_1, f_2) 
\]
where 
\[  
-s_0  = -\frac{m - (n+\epsilon_0)}{2} \geq 0, 
\]
$c$ is some explicit nonzero constant, and $Z^*$ denotes the
normalized global doubling zeta integral.  Here, when $s _0 = 0$, one
actually has the value $L(\frac{1}{2}, \pi)$ rather than the residue
of $L(s+\frac{1}{2},\pi)$ at $s = 0$. 
\vskip 5pt

Using this, Yamana obtained in \cite{Y4} the local-global
criterion for the nonvanishing of the global theta lifting in the
first term range, thus completing Rallis' program when $m \leq n +
\epsilon_0$.

 \vskip 5pt

\subsection{\bf Purpose of this paper.}
After the above, one is left with the case when
\[  
r > 0, \quad     n+ \epsilon_0 < m \leq (n+ \epsilon_0) + r. 
\]  
In this case, one expects to have a Rallis inner product formula
involving the value of $L(s, \pi)$ at $s = s_0 > 0$. For this
purpose, it turns out that 
one needs to show a {\em second term identity} of the regularized Siegel-Weil formula. 
\vskip 5pt
 {\em The purpose of this paper is to supply this remaining global
   ingredient, i.e. we will prove the general second term identity of
   the regularized Siegel-Weil formula, thereby obtaining the Rallis
   inner product formula  in complete generality}.  
 \vskip 5pt

 We note that special cases of the second term identity have been
 known for some time. It was shown in \cite{KRS} when $(H(V_r) , G(W_n))
 = (\OO_4, \Sp_4)$ and in \cite{T1} for $(\U_3, \U_4)$. In addition, when
 the groups $H(V_r)$ and $G(W_n)$ are the symplectic
 (resp. split orthogonal) and split orthogonal (resp. symplectic) groups, a second
 term identity for the spherical vector  was shown in \cite{GT}
 (resp. \cite{K3}). The analogous spherical second
 term identity for unitary groups was shown by W. Xiong in \cite{X}. In a recent 
 paper \cite{Q}, the second-term identity for general vectors was proved by Y. Qiu when 
 $H(V_r)$ and $G(W_n)$ are orthogonal and symplectic with $n=r$.
 \vskip 10pt
 
\subsection{\bf The Regularized Theta Integral.}
 Let us  give a more precise description of the results of this paper. 
 Assume that we are outside the Weil's convergent range, so that $r >
 0$ and $m-r \leq  n+ \epsilon_0$, in which case we have
 \[  
0 < m  \leq 2 \cdot (n+ \epsilon_0) \quad \text{and} \quad r \leq n+
\epsilon_0. 
\]
 We shall further assume (as did Kudla-Rallis in \cite{KR5}) that $r \leq
 n$. This is only a condition when $\epsilon = \epsilon_ 0 =1$.
In any case, consider the Weil representation $\omega_{n,r}$ of
$G(W_n)\times H(V_r)$ and its automorphic realization
$\theta_{n,r}$. We are interested in the theta integral 
 \[  
I_{n,r}(\phi)(g)    =\frac{1}{\tau(H(V_r))} \cdot  \int_{[H(V_r)]} \theta_{n,r}(\phi)(g,h) \, dh. 
\]
This integral diverges, but under the above conditions
 Kudla-Rallis \cite{KR5} discovered a regularization of this theta
 integral, which gives a meromorphic function $B^{n,r}(s,\phi)$. One
 is interested in the behavior of $B^{n,r}(s,\phi)$ at 
 \[  
s = \rho_{H(V_r)} = \frac{m - r -\epsilon_0}{2}.  
\]
 It turns out that in the first term range, when $m \leq n +
 \epsilon_0$,  $B^{n,r}(s,\phi)$ has a pole of order at most $1$
 (which is attained for some $\phi$) whereas in the second term range,
 when 
 \[ 
n+\epsilon_0 < m \leq 2 \cdot (n+\epsilon_0), 
\]
 $B^{n,r}(s,\phi)$ has a pole of order at most $2$.  Thus, the Laurent
 expansion of $B^{n,r}(s,\phi)$ at $s = \rho_{H(V)}$ has the form
 \[  
B^{n,r}(s, \phi)  =  \frac{B^{n,r}_{-1}(\phi)}{s-\rho_{H_r}} +
B^{n,r}_0(\phi)+\cdots  \quad \text{in the first term range;} 
\]
 and
\[
 B^{n,r}(s, \phi)  =  \frac{B^{n,r}_{-2}(\phi)}{(s-\rho_{H_r})^2} +
   \frac{B^{n,r}_{-1}(\phi)}{s-\rho_{H_r}} +  B^{n,r}_0(\phi)+\cdots
 \quad \text{in the second term range.}  
\]
Let us note that each Laurent coefficient $B_d^{n,r}(\phi)$ is an
automorphic form on $G(W_n)$, and hence we view $B_d^{n,r}$ as a
linear map
\[
B_d^{n,r}:\omega_{n,r}\rightarrow \Aut(G(W_n)),
\]
where $\Aut(G(W_n))$ is the space of automorphic forms on $G(W_n)$.
\vskip 5pt

\subsection{\bf First Term Identity in First Term Range.}
The purpose of the Siegel-Weil formula is to identify the automorphic forms $B_d^{n,r}(\phi)$, as much as
possible, with the analogous Laurent coefficients of a
Siegel-Eisenstein series $A^{n,r}(s, \phi)$ associated to $\phi$ by
the formation of Siegel-Weil sections at $s$ equal to
\[
s_{m,n}: = (m - n -\epsilon_0)/2.
\]
Observe that in the first term range, $s_{m,n} \leq 0$, whereas in the second term range,
$s_{m,n} > 0$.   At  $s=s_{m,n} >0$, the Laurent expansion of the Siegel-Eisenstein series 
$A^{n,r}(s,\phi)$ has the form
\[  
A^{n,r}(s,\phi) = \frac{A^{n,r}_{-1}(\phi)}{s-s_{m,n}} + A^{n,r}_0(\phi)
+\cdots
\]
As for $B_d^{n,r}$, each
$A_d^{n,r}(\phi)$ is an automorphic form on $G(W_n)$ and hence
$A_d^{n,r}$ is viewed as a linear map
\[
A_d^{n,r}:\omega_{n,r}\rightarrow\Aut(G(W_n)).
\]
Assume that we are in the first
term range, so that $m \leq n+ \epsilon_0$ and $s_{m,n} \leq 0$.
Let $r' \geq r$ be defined by 
\[ \dim V_r + \dim V_{r'}  = 2 \cdot (n+ \epsilon_0).\]
 The space $V_{r'}$ is called the complementary space to $V_r$ with
 respect to $W_n$ and its dimension $m' = m_0 +2r'$ is such that $s_{m',n} \geq 0$.
 Ikeda  has defined in \cite{Ik} a $G(W_n) \times
 H(V_r)$-equivariant map 
\[  
\Ik^{n, r'} : \omega_{n,r'} \longrightarrow \omega_{n,r}, 
\]
which we shall call  the Ikeda map. 
Then the  first term identity  established in \cite{KR5, Mo, JS, I2,
  GT, Y2} for various dual pairs is the following identity: assuming that $V_r$ is not the split binary quadratic space, then for all $\phi \in \omega_{n,r}$, 
 \[  
a_{n,r'} \cdot A^{n, r'}_{-1} (\phi') = A^{n,r}_0 (\phi)   = 2 \cdot  B^{n,r}_{-1}(\phi),  
\]
 where $\phi' \in \omega_{n,r'}$ is such that 
 \[  \Ik^{n,r'}(\pi_{K_{H_{r'}}}\phi') =
 \phi, \]  
 $\pi_{K_{H_{r'}}}$ is the projection onto the $K_{H_{r'}}$-fixed
 space (with $K_{H_{r'}}$ a maximal compact subgroup of $H_{r'}(\A)$)
 and   $a_{n,r'}$ is some nonzero explicit constant. There is an
 analogous statement for the split binary quadratic case. We recall
 these results in Theorems \ref{T:anisot}, \ref{T:1st} and
 \ref{T:yamana} below.
  \vskip 5pt
 
 Observe here that $A^{n,r}_0$ is the zeroth Laurent coefficient at
 $s_{m,n} \leq 0$, whereas $A^{n,r'}_{-1}$ is the $-1^{\text{st}}$ Laurent
 coefficient at $s= s_{m', n} \geq 0$. One may interpret the identity
 $a_{n,r'} \cdot A^{n, r'}_{-1} (\phi') = 2  \cdot  B^{n,r}_{-1}(\phi)$ as saying that $A_{-1}^{n,r'}$ and
 $B_{-1}^{n,r}\circ\Ik^{n,r'}$ are proportional as linear maps
 $\omega_{n, r'}\rightarrow\Aut(G(W_n))$, and similarly for the
 identity $A^{n,r}_0 (\phi)   = 2 \cdot  B^{n,r}_{-1}(\phi)$.
 \vskip 5pt

\vskip 5pt

 \subsection{\bf The Main Results.}
 The goal of this paper is to prove the first and second term
 identities in the second term range. More precisely, we  show:
 \vskip 5pt
 
 \begin{Thm}[{\bf Siegel-Weil formula}] \label{T:intro}
 Suppose that $ 0 < r \leq n$ and $n+ \epsilon_0 < m  \leq n +
 \epsilon_0 + r$, so that we are in the second term range. Then one
 has:
 \vskip 5pt
 
 \noindent (i) (First term identity) For all $\phi \in \omega_{n,r}$, one has
 \[  
A^{n,r}_{-1}(\phi) =   B^{n,r}_{-2}(\phi).  
\]
 \vskip 5pt
 
 \noindent (ii) (Second term identity) For all $\phi \in \omega_{n,r}$, one has
 \[   
A^{n,r}_0(\phi) 
  =    B^{n,r}_{-1}(\phi)    - \kappa_{r,r'} \cdot \{ 
    B^{n,r'}_0(\Ik^{n,r}(\pi_{K_{H_r}}\phi))\} \mod {\im
  A^{n,r}_{-1}}. 
\]
  Here, $\kappa_{r,r'}>0$ is some explicit constant and $r' < r$ is such that 
  \[  
\dim V_r + \dim V_{r'}  =  2 \cdot (n + \epsilon_0), 
\]
so that  $V_{r'}$ is  the complementary space to $V_r$ with respect to
$W_n$. Moreover,   
\[  
\Ik^{n,r} :  \omega_{n,r}  \longrightarrow\omega_{n, r'} 
\]
is the Ikeda map which is  $G(W_n) \times H(V_{r'})$-equivariant. Finally, the term $\{ ...\}$ on the RHS is interpreted to be $0$ if $V_{r'}$ is anisotropic or is equal to the split binary quadratic form.
 
\end{Thm}
 \vskip 5pt
 
 Note that the equality in the second term identity in (ii) is viewed as an equality in $\Aut(G(W_n))/\im A^{n,r}_{-1}$ where $A^{n,r}_{-1}$ is viewed as a linear map
$A^{n,r}_{-1}:\omega_{n,r}\rightarrow\Aut(G(W_n))$.    
\vskip 10pt
 
The proof of this theorem is a significant refinement of the
techniques of \cite{GT} and is based on induction on the quantity:
 \[  
\mathcal{N} = m - (n + \epsilon_0). 
\]
Observe that, outside the Weil's convergent range,
    \begin{align}
   &\mathcal{N} < 0 \Longleftrightarrow \text{first term range}; \notag \\
   &\mathcal{N} = 0 \Longleftrightarrow \text{boundary case}; \notag \\ 
  &\mathcal{N} > 0 \Longleftrightarrow \text{second term range}. \notag
 \end{align}
To go beyond \cite{GT}, which is limited to the $G(W_n)$-span of the spherical vector, we need to make a more detailed study of 
a  $G(W_n)(\A)\times H(V_r)(\A)$-equivariant map $F^{n,r}(s,-)$ (see \S \ref{sectionf}) from the Weil representation to a family of  induced representations of $G(W_n)\times H(V_r)$ which arises naturally from the regularised theta integral. In particular, we need to make 
a careful study of the properties of the section $F^{n,r}(s,\phi)$ and its restriction $f^{n,r}(s,\phi)$ to $G(W_n)$, such as the behavior of $f^{n,r}(s,\phi)$ when restricted to the subgroup $G(W_{n-1})$, and the effect of the standard intertwining operator on $f^{n,r}(s,\phi)$. 
These key technical results are contained in Proposition \ref{P:key2}, Proposition \ref{P:key1} and 
Lemma \ref{L:comm}.  When $r= n$, the map $F^{n,r}(s,-)$  was
considered and studied in \cite[Prop. 2.4]{Q}. 

\vskip 5pt

 As we noted, the second term identity leads to the following Rallis
 inner product formula for the theta lifting from $G(U_n)$ to
 $H(V_r)$.  Since the Rallis inner product formula has been established in the convergent range and the first term range, we shall focus on the case 
 \[ (n+\epsilon_0) < m \leq 2 \cdot (n+\epsilon_0).\]
 Then we have:
 
 \vskip 5pt
 
 \begin{Thm}[{\bf Rallis inner product formula}] \label{T:intro2}
 Suppose that 
 \[ (n+\epsilon_0) < m \leq 2 \cdot (n+\epsilon_0) \quad \text{and} \quad  r \leq n \]
 so that we are either in the second term range or the convergent range, depending on whether
 $m \leq (n+\epsilon_0) +r$ or not.  Let
 $\pi$ be a cuspidal representation of $G(U_n)$ and consider its
 global theta lift $\Theta_{n,r}(\pi)$ to $H(V_r)$.   \vskip 5pt
 
 \noindent (i) 
Assume that
 $\Theta_{n,j}(\pi) = 0$ for $j < r$, so that  $\Theta_{n,r}(\pi)$   
 is cuspidal. Then
  for $\phi_1, \phi_2 \in \omega_{\psi, U_n, V_r}$ and $f_1, f_2 \in \pi$, 
 \[  
\langle \theta(\phi_1, f_1) , \theta(\phi_2, f_2) \rangle  
 =  [E:F] \cdot  \Val_{s=s_{m,n}}\left( L(s + \frac{1}{2}, \pi)  \cdot Z^*(s,
 \phi_1 \otimes \overline{\phi_2}, f_1, f_2)\right),
 \]
where 
 \[  
s_{m,n} = \frac{m - n-\epsilon_0}{2}  > 0 ,
\]
$L(s,\pi)$ is the standard L-function of $\pi$, and $Z^*(s, -)$ denotes
the normalized doubling zeta integral as in (\ref{E:Z^*}).  
\vskip 5pt

\noindent (ii) Assume further that for all places $v$ of $F$, the local theta lift $\Theta_{n,r}(\pi_v)$ is nonzero. Then $L(s+ \frac{1}{2},\pi)$ is holomorphic at $s = s_{m,n}$, so that in the context of (i),
 \[  
\langle \theta(\phi_1, f_1) , \theta(\phi_2, f_2) \rangle  
 =  [E:F]  \cdot L(s_{m,n} + \frac{1}{2}, \pi)  \cdot Z^*(s_{m,n},
 \phi_1 \otimes \overline{\phi_2}, f_1, f_2).
 \]

\end{Thm}
\vskip 10pt

As  a consequence, we will be able to deduce the following local-global
criterion (see Theorem \ref{T:main_nonvanishing}) for the nonvanishing
of global theta lifts.

 \vskip 5pt

\begin{Thm}[{\bf Local-Global nonvanishing criterion}]\label{T:intro3}
 Assume the same conditions on $(m, n)$ as in Theorem \ref{T:intro2}. 
Let $\pi$ be  a cuspidal representation of $G(U_n)$ and consider its
global theta lift $\Theta_{n,r}(\pi)$ to $H(V_r)$. Assume that
$\Theta_{n,j}(\pi) = 0$ for $j < r$, so that  $\Theta_{n,r}(\pi)$   
 is cuspidal.  
 \vskip 5pt
 
 \noindent (i)  If $\Theta_{n,r}(\pi)$ is nonzero, then
 \begin{enumerate}
\item[(a)] for all places $v$, $\Theta_{n,r}(\pi_v) \ne 0$, and
\item[(b)] $L(s_{m,n} + \frac{1}{2}, \pi)\ne  0$ i.e. nonzero holomorphic.  
\end{enumerate}
\vskip 10pt

\noindent (ii) The converse to (i) holds when one assumes one of the following conditions:
 \begin{itemize}
 \item  $\epsilon_0 = -1$;
 \item  $\epsilon_0 = 0$  and $E_v = F_v \times F_v$ for all archimedean places $v$ of $F$;
 \item  $\epsilon_0 = 1$ and $F$ is totally complex;
 \item $m = d(n)+1$.
 \end{itemize}
 
 \vskip 5pt
 
 \noindent (iii) In general, under the conditions (a) and (b) in (i),
 there is an $\epsilon$-Hermitian space $V'$ over $E$ such that 
 \begin{itemize}
 \item  $V' \otimes F_v \cong V \otimes F_v$  for every finite or complex place of $F$;
 \item  the global theta lift $\Theta_{U_n, V'}(\pi)$ of $\pi$ to $H(V')$ is nonzero.  
 \end{itemize}
 \end{Thm}
\vskip 10pt

The reason for not having the full converse to (i) is because of 
certain local issues at real places of $F$ which still need to be settled (see Proposition
\ref{P:localnon0}).  To be more specific, one needs to show that the
nonvanishing of the local theta lift  $\Theta_{n,r}(\pi_v)$ is
equivalent to the nonvanishing of the local normalized doubling zeta
integral associated to $\pi_v$ on a certain submodule of the
degenerate principal series representation.  Because the  structure of this  degenerate
principal series representation on
$G(W_n)(\R)$ is more complicated in the cases not covered in (ii),  we
could not show the desired equivalence. Once this archimedean issue
can be settled, the converse to (i) will hold in general.
\vskip 5pt

We note in closing that for the purpose of the local-global
nonvanishing criterion, it is  sufficient to show that the inner
product of the theta lifts is proportional to the $L$-value, without
having to be precise about the constant of proportionality. However,
a precise Rallis inner product formula (with every constant
determined) is useful for other purposes. For example, it was used in
the proof of cases of the refined Gross-Prasad conjecture in
\cite{GI1}, \cite{Ha} and \cite{Q}. In addition, it was used crucially in the work of Harris on
period relations, and the work of Harris-Li-Skinner \cite{HLS} on $p$-adic
L-functions.   

\vskip 10pt

\noindent{\bf Acknowledgments:} The debt that this paper owes to the
fundamental work of Kudla and Rallis should be evident to the reader. 
 We thank Atsushi Ichino for  his help in the proof of Lemma
 \ref{L:kappa}(ii), Shunsuke Yamana for sending us his preprint
 \cite{Y4} and for his helpful comments on this manuscript, Soo Teck Lee for his help in the proof of Proposition
 \ref{P:A0} and Chengbo Zhu for helpful discussion concerning the preprint \cite{LZ4}. 
\vskip 5pt

The second author thanks the National University of Singapore for
hosting his visit in Spring 2012 while he and the first author worked
on the paper.  The writing of the paper was completed when the first
and third authors visited the IHES in July 2012; they thank the IHES
for its support and for providing a peaceful yet stimulating working
environment. The first author is partially supported by a grant from
the National University of Singapore. The third author is partially
supported by NSF grant DMS-1215419. 

\vskip 5pt

Finally, we thank the referee of our paper for the many useful comments and suggestions, and for pointing out various inaccuracies in a first draft of this paper, especially with regards to some boundary cases of our main result which require some special considerations. 
\vskip 15pt

 \section{\bf Notation and Preliminaries} \label{S:notation}
 
 In this section, we fix some notation and introduce the objects which
 intervene in this paper.
 
 \subsection{\bf Fields.}
 
 Let $F$ be a number field with  ring of adeles $\A$. Fix a
 non-trivial additive character $\psi=\otimes_v'\psi_v$ on
 $F\backslash \A$.  We let $E$ be either $F$ or a quadratic extension of $F$.
We shall regard $\psi$ as a character of $\A_E$ by composition with the trace map $Tr_{E/F}$. 
Moreover, let $\chi_E$ denote the (possibly trivial) quadratic
character of $\A^{\times}$ associated to $E/F$ by global class field
theory. By $|-|$, we always mean the absolute value on
$\A^{\times}_E$ rather than on $\A^\times$. Occasionally, we use
the notation $|-|$ for the absolute value of a local field, but this should be 
 clear from the context. 
\vskip 5pt

For a vector space $X$ over $E$, by $\mathcal{S}(X)(\A)$ we mean the
space of Schwartz-Bruhat functions on $X\otimes \A$ and so
$\mathcal{S}(X)(\A)=\otimes'_v\mathcal{S}(X)(F_v)$, where
$\mathcal{S}(X)(F_v)$ is the space of Schwartz-Bruhat functions on
$X\otimes F_v$. We sometimes omit $(\A)$ or $(F_v)$ and simply write
$\mathcal{S}(X)$ when there is no danger of confusion.
 
\vskip 5pt

\subsection{\bf $\epsilon$-Hermitian Spaces.}
Throughout the paper, we shall fix the sign
\[  
\epsilon = \pm 1. 
\]
Let $V_0$ be an $m_0$-dimensional vector space over $E$ equipped with
a nondegenerate {\em anisotropic}  $\epsilon$-Hermitian form $(-,-)$.  
With $\mathbb{H}$ denoting the hyperbolic plane, i.e. the split
$\epsilon$-Hermitian space of dimension $2$, we set  
\[  V_r =  V_0 \oplus \mathbb{H}^r,\]
and let 
\[  m =  \dim_E V = m_0 + 2r. \]
The family of spaces $\{ V_r:  r \geq 0\}$ forms a Witt tower of $\epsilon$-Hermitian spaces. 
\vskip 5pt

We now make two basic definitions. Set
\[  \epsilon_0 = \begin{cases}
\epsilon, \text{  if $E = F$;} \\
 0, \text{  if $E \ne F$}, \end{cases} \] 
 and
\[  d(n)  = n + \epsilon_0 \]    
 for $n \in \Z$.
 \vskip 5pt

Let $H_r  = H(V_r)$ be the associated isometry group, so that
\[  H_r \cong \begin{cases}
\OO_m, &\text{  if $\epsilon_0 = 1$;} \\
\Sp_m, &\text{  if $\epsilon_0= -1$;} \\
\U_m, &\text{  if $\epsilon_0 = 0$.} \end{cases} \]
The family of groups $\{  H_r: r \geq 0\}$ forms a Witt tower of classical groups.  
\vskip 5pt

The space $V_r$ has  a maximal isotropic space $X_r$ of dimension $r$, so that
\[  V_r =  X_r  \oplus  V_0 \oplus X_r^*. \]
  Fix an ordered basis $\{x_1,\dots,x_r\}$ of $X_r$, with corresponding dual basis $\{x_r^*, \cdots ,x_1^* \}$ for $X_r^*$, so that
\[  X_r  = E \cdot x_1  \oplus\cdots \oplus E \cdot x_r \quad \text{and} \quad  
 X_r^*  = E \cdot x_r^*  \oplus\cdots\oplus E \cdot x_1^* \]
 with
 \[  (x_i, x_j) = (x_i^*, x_j^*) = 0 \quad \text{and} \quad (x_i, x_j^*) = \delta_{ij}. \]
Let 
\[ P(X_r) = M(X_r) \cdot N(X_r)\]
 be the maximal parabolic subgroup of $H_r$ which stabilizes the space
 $X_r$. Then its Levi factor is 
\[  M(X_r) \cong \GL(X_r) \times H(V_0).\]
To simplify notation, we shall sometimes write $P_r$ in place of $P(X_r)$. Note that in all cases, we have
\[ \pi_0(H_r) = \pi_0(H_0)  \quad \text{  if $r > 0$,} \]
except when $\epsilon_0 =1$ and $m_0 = 0$, in which case $H_r = \OO_{r,r}$ and $H_0 $ is trivial so that
\[  \# \pi_0(H_r/ H_0)  = 2. \]
Here $\pi_0$ indicates the set of connected components. 
\vskip 5pt

 The group $H_r$ comes equipped with a family of maximal compact subgroups $\{ K_{H_{r,v}} \}$, such that $K_{H_{r,v}}$ is hyperspecial for almost all places $v$ of $F$. We may and do assume that $K_{H_{r,v}}$ is a special maximal compact subgroup for all finite $v$. Then
 \[  K_{H_r} = \prod_v K_{H_{r,v}}\subset H_r(\A) \]
 is a maximal compact subgroup of $H_r(\A)$ and one has the Iwasawa decomposition
 \[  H_r(\A) = P(X_r)(\A) \cdot K_{H_r}. \]
 We may also ensure that $K_{H_r} \cap \GL(X_r)(\A)$ is a maximal compact subgroup of $\GL(X_r)(\A)$. 
\vskip 5pt

\subsection{\bf $-\epsilon$-Hermitian Spaces.}
Similarly, let $W_n$ be a $2n$-dimensional vector space over $E$
equipped with a nondegenerate $-\epsilon$-Hermitian form $\langle- , -
\rangle$ and a maximal isotropic subspace $Y_n$ of dimension $n$, so
that
\[  W_n  = Y_n   \oplus Y_n^*.\]
 Let $G = G_n =  G(W_n)$ be  the associated isometry group or the
 unique two-fold cover thereof, according to
\[  G_n  = \begin{cases} 
\Sp_{2n}, &\text{  if $\epsilon_0=1$ and $m_0$ is even;} \\
\Mp_{2n}, &\text{  if $\epsilon_0 = 1$ and $m_0$ is odd;}\\
\OO_{2n}, &\text{  if $\epsilon_0 = -1$;} \\
\U_{2n}, &\text{  if $\epsilon_0 = 0$.}  \end{cases} \]
Thus, in addition to the space $W_n$, the group $G_n$ depends on the space $V_0$ (or rather the parity of its dimension) in the first two cases.  
\vskip 5pt

 We  fix an ordered basis $\{y_1,\dots, y_n\}$ of $Y_n$ and corresponding dual basis $\{y_n^*,..., y_1^* \}$ for $Y_n^*$, so that
\[  Y_n  =  E \cdot y_1 \oplus \cdots \oplus E \cdot y_n \quad \text{and} \quad 
Y_n^* =  E \cdot y_n^* \oplus \cdots \oplus E \cdot y_1^*. \]
For any subspace $Y_r := \langle y_1,...,y_r \rangle \subset Y_n$, let 
\[
Q(Y_r) = L(Y_r) \cdot U(Y_r)
\] 
denote the maximal parabolic subgroup fixing $Y_r$. Then its Levi factor is
\[  L(Y_r)  \cong \GL(Y_r) \times G_{n-r}. \]
 As before, we shall sometimes write $Q_r$ in place of
 $Q(Y_r)$. Moreover, if there is a need to indicate that $Q_r$ is a
 subgroup of $G_n$, we shall write $Q_r^n$. When $r = n$, $Q(Y_r)$ is
 a  Siegel parabolic subgroup of $G_n$.
 \vskip 5pt
 
 The unipotent radical $U(Y_r)$ of $Q(Y_r)$ sits in a short exact sequence
 \[  
\begin{CD}
1 @>>>  Z(Y_r) @>>> N(Y_r) @>>> V_{n-r} \otimes Y_r @>>> 1 \end{CD} 
\]
where
\[  
Z(Y_r) = \{ \text{$\epsilon$-Hermitian forms on $Y_r^*$} \} \subset
\Hom_F(Y_r^*, Y_r). 
\]
Thus, when $r=n$, $N(Y_n) = Z(Y_n)$ is abelian.
\vskip 5pt

 \vskip 5pt
The following table summarizes the groups discussed so far:
\vskip 5pt
\begin{center}
 \begin{tabular}{|c|c|c|c|}
 \hline
 \multicolumn{2}{|c|}{}&  $G_n=G(W_n)$ & $H_r=H(V_r)$ \\
 \hline
\multirow{2}{*}{$\epsilon_0 =1$}& $m_0$ even & $\Sp_{2n}$ & $\OO_m$ \\
  \cline{2-4}
   & $m_0$ odd & $\Mp_{2n}$ & $\OO_m$ \\
\hline
\multicolumn{2}{|c|}{$\epsilon_0=-1$}&   $\OO_{2n}$ & $\Sp_m$  \\
  \hline
 \multicolumn{2}{|c|}{$\epsilon_0=0$}& $\U_{2n}$ & $\U_m$  \\
  \hline\hline
\multicolumn{2}{|c|}{Parabolic}&$Q_r=Q(Y_r)=L(Y_r)\cdot U(Y_r)$
&$P_r=P(X_r)=M(X_r)\cdot N(X_r)$\\ 
\hline
\multicolumn{2}{|c|}{Levi factor}&$L(Y_r)=\GL(Y_r)\times G_{n-r}$
&$M(X_r)=\GL(X_r)\times H(V_0)$\\ 
\hline
\end{tabular}
\end{center}

\vskip 5pt

\subsection{\bf Metaplectic case.} 
 The case when $G_n(\A)$ is the metaplectic group $\Mp_{2n}(\A)$ deserves
 further comments. In this case, a parabolic subgroup of $G_n(\A)$ is
 simply the inverse image of a parabolic subgroup of $\Sp_{2n}(\A)$. We
 will follow the notation and conventions in \cite{GS} for the  structural
 issues regarding $\Mp_{2n}(\A)$. For example, the Levi subgroup of a
 parabolic $Q(Y_r)(\A)$ is:
 \[  L(Y_r)(\A)  = \widetilde{ \GL}_r(\A) \times_{\mu_2} \Mp_{2n-2r}(\A) \]
 where $\widetilde{\GL}_r(\A)$ is the two-fold cover of $\GL_r(\A)$
 defined by the Hilbert symbol, namely
 \[ \widetilde{\GL}_r(\A)=\GL_r(\A)\times\{\pm 1\} \]
  as a set and the
   group structure is given by 
   \[ (g_1,\eta_1)\cdot (g_2, \eta_2)=(g_1g_2,
   \eta_1\eta_2(\det g_1, \det g_2)_{\A}) \quad \text{for $g_i\in\GL_r(\A)$ and
   $\eta_i\in\{\pm 1\}$,} \]
    where $(-,-)_\A$ is the Hilbert symbol. The
   determinant map of $\GL_r(\A)$ lifts to a homomorphism
 \[  \det:  \widetilde{\GL}_r(\A) \longrightarrow \widetilde{\GL}_1(\A). \] 
By the theory of Weil indices,  the additive character $\psi$ gives rise to a genuine character 
\[  \chi_{\psi} :  \widetilde{\GL}_1(\A) \longrightarrow \C^{\times}. \]
 Composing with $\det$, one obtains a genuine character $\chi_{\psi}
 \circ \det$ of $\widetilde{\GL}_r(\A)$. The map $\tau \mapsto
 \tilde{\tau}_{\psi} = (\chi_{\psi} \circ \det) \otimes \tau$ then
 gives an identification (depending on $\psi$) of the irreducible
 representations of $\GL_r(\A)$ with the irreducible genuine
 representations of $\widetilde{\GL}_r(\A)$. 
 \vskip 5pt
 
 In order to unify and simplify notation, we shall henceforth write
 $\tau$ for a genuine representation of $L(Y_r)(\A)$ in the metaplectic
 case, when we actually mean the genuine representation
 $\tilde{\tau}_{\psi}$. 
Also for a parabolic subgroup $Q(Y_r)(\A)$ of $\Mp_{2n}(\A)$ and a
 representation $\tau \boxtimes \sigma$ of $L(Y_r)(\A)= \widetilde{\GL}(Y_r)(\A)
 \times_{\mu_2} \Mp_{2n-2r}(\A)$, one may consider the normalized induced
 representation $\Ind_{Q(Y)(\A)}^{\Mp_{2n}(\A)} ( \tau|\det|^s \boxtimes \sigma)$, keeping in mind
  that $\tau$ denotes the genuine representation
 $\tilde{\tau}_{\psi}$. In particular, this notion of induced
 representation depends on $\psi$, though we suppress $\psi$ from the notation. 
 \vskip 5pt

\subsection{\bf Measures.}
Having fixed the additive character $\psi$ of  $F \backslash \A$,
 we fix the Haar measure $dx_v$ on $F_v$ (for all $v$) to be self-dual
 with respect to $\psi_v$. The product measure $dx$ on $\A$ is independent of the choice of 
 $\psi$ and is the Tamagawa measure of $\A$. 
 For any  algebraic group $G$ over $F$, we
 always use the Tamagawa measure on $G(\A)$ when $G(\A)$ is unimodular. This applies to
 the groups $G_n(\A)$ and $H_r(\A)$, as well as the Levi subgroups and unipotent radical of their parabolic subgroups. 
In addition, for any compact group $K$, we
 always use the Haar measure $dk$ with respect to which $K$ has volume
 $1$. 
 
\vskip 5pt

Finally, we need to fix a measure for the metaplectic group
$\Mp_{2n}(\A)$. Recall that one has a short exact sequence
 \[  \begin{CD}
 1 @>>> \mu_2 @>>> \Mp_{2n}(\A) @>\pi>> \Sp_{2n}(\A)@>>>1. \end{CD} \]
 Having fixed the Tamagawa measure $dg$ on $\Sp_{2n}(\A)$, we fix the
 measure $d \tilde{g}$ on $\Mp_{2n}(\A)$ such that $ \pi_*(d
 \tilde{g}) =  dg$, so that
 \[  \int_{\Sp_{2n}(F) \backslash \Mp_{2n}(\A)} d \tilde{g}  =1. \]
We shall call this $d\tilde{g}$ the Tamagawa measure for $\Mp_{2n}(\A)$. 
We adopt the same convention for the Tamagawa measure on the nonlinear cover $\tilde{\GL}_r(\A)$ of 
$\GL_r(\A)$ which intervenes in the parabolic subgroups of $\Mp_{2n}(\A)$. 
\vskip 5pt

 In any case, we write $\tau(G)$ for the Tamagawa number of $G$. When the identity component  
 $G^0$ has no $F$-rational characters, $\tau(G)$ is simply defined by:
 \[  \tau(G) = \int_{[G]} \, dg, \]
 and a more delicate definition is needed in the general case.
 In any case, we have
 \[  \tau(H_r)  = \begin{cases}
 1, &\text{ if $\epsilon_0 \ne 0$;} \\
 2, &\text{  if $\epsilon_0 = 0$,} \end{cases} \]
 except when
 \vskip 5pt
 
 \begin{itemize}
 \item  $\epsilon_0= 1$, $m_0 =1$ and $r=0$, in which case $H_r =
 \OO_1$ has $\tau(\OO_1) =1/2$  (see \cite{We2});
 \vskip 5pt
 
 \item $\epsilon_0 = 1$, $m_0 = 0$ and $r=1$, in which case $H_r = \OO_{1,1}$ is associated to the split binary quadratic form and $H_r^0  \cong \mathbb{G}_m$, so that $\tau(\OO_{1,1}) = 1/2$.   
 \end{itemize}
 \vskip 5pt

\subsection{\bf Complementary Spaces.}
With $W_n$ fixed,  one may associate to each $V_r$  a complementary
space $V_{r'}$ in the same Witt tower, characterized by 
\[  m_0 + r + r'   =   d(n). \]
 Of course, this only makes sense if $r' \geq 0$. Thus, the notion of
 complementary spaces is only relevant when
 \[  0 \leq m_0 +r \leq  d(n). \]
  When $V_{r'}$ exists, and $r \geq r'$, we may write:
  \[  V_r  = X'_{r-r'} \oplus V_{r'} \oplus {X'}^*_{r-r'} \]
  where
  \[  X'_{r-r'} = \langle  x_{r'+1},...., x_r \rangle. \]
We set
\[
m'=\dim V_{r'}=m_0+2r'.
\]
 
\vskip 5pt

\subsection{\bf Ikeda's map.}
Suppose that  $V_r \supset V_{r'}$ (not necessarily complementary spaces).   
Then one may write
\[  V_r  =  X'_{r-r'} \oplus V_{r'} \oplus   {X'}^*_{r-r'}. \]
We define a map 
\[  \Ik^{n,r, r'}:  \mathcal{S}(Y_n^* \otimes V_r)(\A)\longrightarrow 
\mathcal{S}(Y_n^* \otimes V_{r'})(\A), \]
given by
\[  \Ik^{n,r,r'}(\phi)(a) = \int_{(Y_n^* \otimes {X'_{r-r'}})(\A) }  \phi( x, a, 0)  \, dx,\]
for $a \in (Y_n^* \otimes V_{r'})(\A)$.
 Thus, $\Ik^{n, r,r'}$ is the composite
 \[  \begin{CD}
 \mathcal{S}(Y_n^* \otimes V_r)
= \mathcal{S}(Y_n^* \otimes V_{r'})\otimes 
\mathcal{S}(Y_n^* \otimes (X'_{r-r'} + {X'}^*_{r-r'}))\\
 @VVId \otimes \mathcal{F}_1V \\
  \mathcal{S}(Y_n^* \otimes V_{r'})\otimes
  \mathcal{S}(W_n\otimes X'_{r-r'})\\
  @VVId \otimes ev_0V \\
  \mathcal{S}(Y_n^* \otimes V_{r'})
  \end{CD}
  \]
  where 
  \[  \mathcal{F}_1 : \mathcal{S}(Y_n^* \otimes (X'_{r-r'} +
  {X'}^*_{r-r'}))
\longrightarrow \mathcal{S}(W_n\otimes X'_{r-r'})\]
  is the partial Fourier transform in the subspace $(Y_n^* \otimes
  {X'}^*_{r-r'})(\A)$, and $ev_0$ is  evaluation at 0. It is clear that
  if $r'' < r' < r$, one has
  \begin{equation} \label{E:ikeda}
    \Ik^{n,r', r''} \circ \Ik^{n, r, r'}  = \Ik^{n, r, r''}. 
\end{equation}
  \vskip 5pt

  In the special case when  $V_r$ and $V_{r'}$ are complementary
  spaces, we shall simply write $\Ik^{n,r}$ for $\Ik^{n,r,r'}$. The
  map $\Ik^{n,r}$ was used by T. Ikeda in his refinement (\cite{Ik}) of the
  first term identity of Kudla-Rallis (\cite{KR5}). Thus we call
  $\Ik^{n,r}$ (more generally $\Ik^{n,r,r'}$) an Ikeda map.
 
 \vskip 5pt
 
 \subsection{\bf Weil representation.}  \label{SS:weil}
 We choose a Hecke character $\chi$ as follows:
\[  
\chi = \chi_{V_0} = \begin{cases}
\text{the quadratic character of $\A^{\times}$ associated to
  $\disc(V_0)$,  if $\epsilon_0=1$;} \\
\text{the trivial character of $\A^{\times}$, if $\epsilon_0 = -1$;} \\
\text{a character of $\A^{\times}_E$ such that $\chi|_{\A^{\times}} =
  \chi_E^{m_0}$, if $\epsilon_0 = 0$.} \end{cases} 
\]
Note that when $\epsilon_0=\pm 1$, $\chi$ is completely determined
by $V_0$.
\vskip 5pt

Associated to the pair $(\psi, \chi)$,
the group $G_n(\A)\times H_r(\A)$ has a distinguished representation
$\omega_{n,r}  = \omega_{\psi, \chi, W_n, V_r}$ known as the Weil representation.
 The Weil representation can be realized on the space of
 Schwartz-Bruhat functions $\mathcal{S}(Y_n^*\otimes V_r)(\A)$. The
 action is given by the following formulas.
\vskip 5pt
\[
\begin{cases}
\omega_{n,r}(1, h)\phi(x)  = \phi(h^{-1} \cdot x)  \text{  if $h \in H_r(\A)$;} \\
\omega_{n,r}(a,1) \phi(x)  =  \chi(\det(a)) \cdot |\det (a)|^{m/2}
\cdot \phi( a^{-1} \cdot x),  \text{  for $a \in L(Y_n)(\A) =
  \GL(Y_n)(\A)$;} \\
\omega_{n,r}(u,1) \phi(x)  = \psi(\frac{1}{2} \cdot  \langle  u(x),
x\rangle)\cdot \phi(x),  \text{  for $u \in N(Y_n)(\A)  \subset
  \Hom_F(Y_n^*, Y_n)(\A)$,} \\
\omega_{n,r}(w_0, 1)\phi(x)=  \gamma_{V_r, W_n} \cdot  \int_{V_r \otimes Y_n^*(\A)} \phi(y)
\cdot \psi( \langle x,y\rangle) \, dy,
  \end{cases} \]
 where $w_0\in G_n(F)$ is the Weyl group element $y_i \mapsto - \epsilon \cdot y_i^*$
 and $y_i^* \mapsto y_i$, and $ \gamma_{V_r, W_n}$ is a certain root of unity which we will not explicate here.

\vskip 15pt

Note that the Ikeda map 
\[  \Ik^{n,r,r'}: \mathcal{S}(V_r \otimes Y_n^*)(\A) \longrightarrow
\mathcal{S}(V_{r'} \otimes Y_n^*)(\A) \]
is $G_n(\A)\times H_{r'}(\A)$-equivariant. 

\vskip 5pt

\subsection{\bf The Fourier transform $\mathcal{F}_{n,r}$.} 
On the other hand, one has the partial Fourier transform
 \[  
\mathcal{F}_{n,r} :   \mathcal{S}(Y_n^*\otimes V_r)(\A) \longrightarrow 
 \mathcal{S}(W_n\otimes X_r^*)(\A) \otimes \mathcal{S}(Y_n^*\otimes
 V_0)(\A)
\]
 which is given by integration over the subspace
$(Y_n^*\otimes X_r)(\A)$. 
 We may regard $\mathcal{F}_{n,r}(\phi)$ as a function 
 on $(W_n\otimes X_r^*)(\A)$ taking values in $\mathcal{S}(Y_n^*\otimes V_0)(\A)$. 
 Transferring the action of $G_n(\A)\times H_r(\A)$ via $\mathcal{F}_{n,r}$,  
 we see for example that:
 \[
 \begin{cases}  \omega_{n,r}(g,1)  \mathcal{F}_{n,r}(\phi) (x) =
   \omega_{n,0}(g) ( \mathcal{F}_{n,r}(\phi)(g^{-1} \cdot x) ) \text{
     for $g \in G_n(\A)$;} \\
 \omega_{n,r}(1,a) \mathcal{F}_{n,r}(\phi)(x)  =  |\det(a)|^n \cdot
 \mathcal{F}_{n,r}(\phi)( x \circ a), \text{ for $a\in \GL(X_r)(\A)$.} 
\end{cases} 
\] 
 A bit of care is needed to interpret the representation
 $\omega_{n,0}$ above when $m_0 =0$. In that case,
 $\mathcal{S}(Y_n^*\otimes V_0)(\A)= \C$ is a 1-dimensional
 representation of $G_n$. When $\epsilon_ 0= \pm 1$ (and $m_0 = 0$),
 $\omega_{n,0}$ is
 the trivial representation. 
 When $\epsilon_0 = 0$, one has an isomorphism
 \[  \iota:  E^{\times}/ F^{\times} \longrightarrow E_1 \]
 where $E_1$ denotes the set of norm 1 elements in $E$, given by
 \[  \iota(x)  = x/ x^c. \]
 Moreover, since  $m_0 = 0$, $\chi$ is a character of
 $\A_E^{\times}/\A^{\times}$, so that one has a character $\chi \circ
 \iota^{-1} \circ \det_{G_n}$ of $G_n(\A)$. Then one has
 \begin{equation} \label{E:omega0}
   \omega_{n,0} = \chi \circ \iota^{-1} \circ {\det}_{G_n}. 
\end{equation}
  
 \vskip 15pt
 
\subsection{\bf Automorphic forms.}
 For an algebraic group $G$ over $F$, we set
\[
    [G]:=G(F)\backslash G(\A).
\]
We denote by $\Aut(G)$ the space of automorphic forms on $G$. The elements in $\Aut(G)$ are
 smooth automorphic forms on $[G]$ and we do not impose the usual
$K$-finiteness condition. Hence  the full group $G(\A)$ acts on $\mathcal{A}(G)$.  

\vskip 15pt

\subsection{\bf Automorphic Theta distribution.}

The Weil representation $\omega_{n, r}$ has an automorphic realization 
\[  
\theta_{n,r}:  \mathcal{S}(Y_n^*\otimes V_r)(\A)  \longrightarrow 
\text{Functions on $[G_n] \times [H_r]$} 
\]
given by
\[  
\theta_{n,r}(\phi)(g,h) = \sum_{x \in (Y_n^*\otimes V_r)(F)}
(\omega_{n,r}(g,h)\phi) (x). 
\]

\vskip 15pt

 \section{\bf The Regularized  Theta Integral}  \label{S:theta}
 
 In this section, we recall the definition of the theta integral and
 its regularization defined by Kudla-Rallis.

 \vskip 5pt
 
 \subsection{\bf Theta Integral.} \label{SS:theta}
 The theta integral we are interested in is:
  \[  
I_{n,r}(\phi)(g) := \frac{1}{\tau(H_r)} \cdot  \int_{[H_r]}  \theta_{n,r}(\phi)(g,h) \, dh. 
\]
Note that the measure used in the above integral is effectively $\tau(H_r)^{-1} \cdot dh$, with respect to which $[H_r]$ has volume $1$ (when $H_r \ne \OO_{1,1}$). This agrees with the convention used in \cite{KR1,KR2, Ik, I1,I2, I3, Y2} when $H_r \ne \OO_{1,1}$, but the convention is different in the $\OO_{1,1}$ case.
\vskip 5pt

The theta integral $I_{n,r}(\phi)$ is an automorphic form on $G_n$ if the integral converges absolutely. 
 By Weil \cite{We}, it is known that the above integral converges if and only if 
  \[  r = 0 \quad \text{or} \quad m-r = m_0 + r >  d(n). \]
 We call this the Weil's convergent range.

 \vskip 5pt

 In particular, the pair $(W_n, V_0)$ is in this range. Then for
 $\phi_0 \in \mathcal{S}(Y_n^* \otimes V_0)(\A)$, we have the theta
 integral $I_{n,0}(\phi_0)\in \Aut(G_n)$. We set
 \[  
\Theta_{n,0}(V_0)  = \{I_{n,0}(\phi_0) :  \phi_0 \in
\mathcal{S}(Y_n^* \otimes V_0)(\A)\}\subset \Aut(G_n) 
\]
 which is an automorphic (sub)representation of $G_n$. If $V_0 = 0$,
 then $\Theta_{n,0}(V_0)$ is interpreted to be the trivial
 representation of $G_n$ if $\epsilon_0 = \pm 1$, and is interpreted to
 be the character $\chi \circ \iota^{-1} \circ \det_{G_n}$ as in
 (\ref{E:omega0}) if $\epsilon_0=0$.
 \vskip 5pt
 
The classical Siegel-Weil formula, first formulated in  \cite{We} and  extended by
Kudla-Rallis in \cite{KR1}, determines the theta integral
$I_{n,0}(\phi_0)$ as an Eisenstein series on $G_n$. 
 \vskip 10pt

 \subsection{\bf Regularization.} \label{SS:kr-reg}
 In \cite{KR5}, Kudla and Rallis defined a regularization of the
 integral $I_{n,r}(\phi)$ beyond the Weil's convergent range. More
 precisely, we suppose that 
 \begin{equation} \label{E:range}
 r>0 , \quad  \text{and} \quad  m - r \leq d(n).  
\end{equation}
Note that the second inequality implies that
\[   r \leq d(n)  \leq n+1. \]
We further impose (as did Kudla-Rallis) that
\begin{equation}\label{E:range2}
 r \leq n, 
\end{equation}
or equivalently that the $F$-rank of $G_n$ is at least that of $H_r$. 
Note that this is only an extra condition when $\epsilon_0 =1$; when
$\epsilon_0 =-1$ or $0$, it follows automatically from (\ref{E:range})
since $d(n) \leq n$. Indeed, the only case satisfying (\ref{E:range})
but not (\ref{E:range2}) is when $\epsilon_0 =1$ and $m_0 = 0$ so that
$H_r = \OO_{n+1,n+1}$ (split) and $G_n = \Sp_{2n}$.  We also note that
under (\ref{E:range}):
\[  0 < m \leq 2 \cdot d(n). \]
\vskip 5pt

 For a given place $v$,  one can find an element $z_G$ (resp. $z_H$)
 in the center of universal enveloping algebra of $\mathfrak{g}_v$
 (resp. $\mathfrak{h}_v$) for $v$ real (\cite{KR5}) or in the spherical Hecke
 algebra of $G_n(F_v)$ (resp. $H_r(F_v)$) for $v$ non-archimedean
 (\cite{I1, I2, T1})  so that
  \begin{align*}
    \omega_{n,r}(z_G)&=\omega_{n,r}(z_H),\\
    \omega_{n,r}(g,h)\omega_{n,r}(z_G)&=\omega_{n,r}(z_G)\omega_{n,r}(g,h)
\end{align*}
(i.e. the action of $z_G$ (and hence $z_H$) commutes with the action
of $G_n(\A)\times H_r(\A)$),
and such that the function $\theta_{n,r}(\omega_{n,r}(z_G)\varphi)(g,-)$ is
rapidly decreasing on a Siegel domain of $H_r(\A)$. It follows that
one can integrate $\theta_{n,r}(\omega_{n,r}(z_G) \phi)$ against any
automorphic form on $H_r$; we shall integrate it against an auxiliary
Eisenstein series on $H_r$ to be defined next. 
\vskip 15pt

\subsection{\bf Auxiliary Eisenstein series on $H_r$.}  \label{SS:aux}
Recall $P(X_r)$ is the parabolic subgroup of $H_r$ whose Levi is
$M(X_r)=\GL(X_r)\times H(V_0)$. Consider the  family of (normalized)
induced representation 
\[ 
 I_{H_r}(s) :=  \Ind_{P(X_r)(\A)}^{H_r(\A)}  |\det|^s \boxtimes
{\bf 1}_{H_0} 
\]
 where $|\det|^s$ is a character of $\GL(X_r)(\A)$ and  ${\bf 1}_{H_0}$ is
 the trivial representation of $H_0(\A)$. Let $f_s^0\in I_{H_r}(s)$ be
 the $K_{H_r}$-spherical standard section with $f_s^0(1)=1$. Then we define the
 Eisenstein series $E_{H_r}(s, -)$ on $H_r$ by
\[
E_{H_r}(s, h)=\sum_{\gamma\in P(X_r)(F)\backslash H_r(F)}f_s^0(\gamma h)
\]
for $h\in H_r(\A)$ and $\operatorname{Re}(s) \gg 0$. We call $E_{H_r}(s, -)$ the auxiliary Eisenstein series.
\vskip 5pt

We are interested in the point
 \[  
s =  \rho_{H_r}  :=  \frac{m- r - \epsilon_0}{2}.
 \]
 We have:
 \[  {\rm ord}_{s=\rho_{H_r}} E_{H_r}(s,-)  =  \begin{cases}
 0, \text{  if $H_r = \OO_{1,1}$, i.e. $m=2$, $r=1$ and $\epsilon_0=1$;} \\
 -1, \text{  otherwise.} \end{cases} \]
The leading Laurent coefficient of $E_{H_r}(s,-)$ at $s= \rho_{H_r}$ is a constant function and we  set:
 \[  \kappa_r  = \begin{cases}
 E_{H_r}(\rho_{H_r},h) , \text{  if $H_r = \OO_{1,1}$;} \\
{\rm Res}_{s = \rho_{H_r}} E_{H_r}(s,h), \text{  otherwise.}  \end{cases}
\] 
Note that when $H_r = \OO_{1,1}$ (so that $r=1$), one has
$P(X_r) = H_r^0 \cong \mathbb{G}_m$,  $\rho_{H_r}  =0$ and  $\kappa_r  =2$. For a detailed discussion of this degenerate case, the reader can consult \cite[\S 3.1, Pg. 499-503]{KRS}.

 \vskip 5pt

Now  the regularizing element $z_H$
 acts on $E_{H_r}(s,-)$ by a scalar $P_{n,r}(s)$:
  \[   
z_H \ast E_{H_r}(s,-)  = P_{n,r}(s) \cdot E_{H_r}(s,-). 
\]
Here the scalar $P_{n,r}(s)$ depends on the choice of $z_H$, though we
suppress the dependence from our notation.
The function $P_{n,r}(s)$ can be explicitly computed; see \cite[Lemma
5.5.3]{KR5}, \cite[p. 208]{I1}, \cite[p. 249]{I2} and
\cite[Cor. 2.2.5]{T1} for the
various cases.
\vskip 5pt

 \vskip 5pt

There is a standard intertwining operator 
  \[  
M_{H_r}(s) :  I_{H_r}(s) \longrightarrow I_{H_r}(-s) 
\]
defined for ${\rm Re}(s) \gg 0$ by
\[  
M_{H_r}(s) f_s(h) = \int_{N(X_r)(\A)} f_s(w_{H_r} u h) \, du 
\]
for $f_s\in I_{H_r}(s)$,
 where $w_{H_r}$ is the element in $H_r(F)$ defined by
  \[  
w_{H_r} :  \begin{cases}
  x_i \mapsto \epsilon \cdot x_i^* \text{  for $1 \leq i \leq r$;}  \\
  x_i^*  \mapsto  x_i \text{  for $1\leq i \leq r$;} \\
  v_0 \mapsto v_0  \text{  for $v_0 \in V_0$.} \end{cases} 
\]
  For almost all $s$, $M_{H_r}(s)$ is an $H_r(\A)$-equivariant
  isomorphism. In particular, there is a meromorphic scalar-valued
  function $c_r(s)$ such that
  \[   
M_{H_r}(s) f^0_s  = c_r(s) \cdot f^0_{-s}.  
\]
 Together with the functional equation of Eisenstein series, one has:
 \begin{equation}\label{E:spherical_functional_equation}
E_{H_r}(s, -)  = c_r(s) \cdot E_{H_r}(-s,-) . 
\end{equation}
The function $c_r(s)$ can be explicitly computed (see for example
\cite[\S 9]{Ik} and \cite[\S 9]{I2}), but we shall not need its
explicit value in this paper. The only information we need to know
about the function $c_r(s)$ is the following innocuous looking but
really quite important lemma:
\vskip 5pt

\begin{Lem} \label{L:kappa}
 (i) At the point $s = \rho_{H_r}$, 
 \[  {\rm ord}_{s = \rho_{H_r}} c_r(s)  = {\rm ord}_{s= \rho_{H_r}}E_{H_r}(s,-), \]
 and
 \[ {\rm Val}_{s = \rho_{H_r}} \left(  \frac{E_{H_r}(s,-)}{c_r(s)} \right) =
 \# \pi_0(H_r/H_0) = \begin{cases}
 1 \text{  if $H_r \ne \OO_{r,r}$;} \\
 2 \text{  if $H_r = \OO_{r,r}$.}\end{cases} \]
 In particular, 
 \[  {\rm ord}_{s = \rho_{H_r}} c_r(s)/c_{r-1}(s-1/2)
 = \begin{cases}
 -1,  \text{  if $H_r= \OO_{2,2}$ or if $r=1$ and $H_r \ne \OO_{1,1}$;} \\
 0, \text{  otherwise.}  \end{cases} \]
and its leading Laurent coefficient  there is $\kappa_r/
 \kappa_{r-1} \cdot \# \pi_0(H_{r-1}/H_r)$. Here, we have set $c_0(s) =1$ and $\kappa_0 =1$ by convention.
 
  \vskip 5pt
 
 (ii)  With respect to the Iwasawa decomposition $H_r(\A) =  M(X_r)(\A)
 \cdot N(X_r)(\A) \cdot  K_{H_r}$, one has:
\[ \frac{1}{\tau(H_r)} \cdot  \int_{H_r(\A)} \varphi(h) \,  dh  = \kappa_r  \cdot
\frac{1}{\tau(H_0)} \cdot \int_{M(X_r)(\A)}  \int_{N(X_r)(\A)}  \int_{K_{H_r}}
\varphi(m \cdot n \cdot k) \,   dm \,  dn   \,  dk,\] 
for any $\varphi \in C^{\infty}_c(H_r(\A))$, where $dh$, $dm$ and $dn$
are the Tamagawa measures on $H_r(\A)$, $M(X_r)(\A)$ and $N(X_r)(\A)$,
respectively and $dk$ is such that the volume of $K_{H_r}$ is 1.
 \end{Lem}
\vskip 5pt
\begin{proof}
(i)  We give the formal proof of (i) in \S \ref{SS:kappa}, after we
introduce the relevant notation for constant terms of Eisenstein
series. But let us mention here that, by the theory of Eisenstein
series, the analytic behavior of $E_{H_r}(s, -)$
is the same as that of its constant term along $N(X_r)$. Then (i) is saying
that the analytic behavior of this constant term  at $s=\rho_{H_r}$ is the same as that 
of  the intertwining operator $M_{H_r}(s)$. 
 
\vskip 5pt

(ii)  This follows from  \cite[Lemma 9.1]{I2} and the computations in
\cite[\S 9]{Ik} and \cite[\S 9]{I2}, especially \cite[Thm. 9.6 and
9.7]{Ik} and \cite[Thm. 9.5 and 9.6]{I2}. We give a sketch of the proof. 
\vskip 5pt

Since the case when $H_r = \OO_{1,1}$ can be easily verified by hand, we assume that $H_r \ne \OO_{1,1}$ henceforth.
According to \cite[Lemma 9.1]{Ik} and \cite[Lemma 9.1]{I2}, if we define the constant $\alpha$ by
\[  dh  = \alpha \cdot dm \, dn \, dk, \]
then 
\begin{equation} \label{E:meas-const}
 \alpha =  {\rm Res}_{s=1}  \frac{L(s, M_r)}{L(s, H_r)}  \cdot \left( \prod_v  \frac{L_v(1,H_r)}{L_v(1, M_v)} \cdot  c_{r,v}(\rho_{H_r}) \right), \end{equation}
where $L(s, H_r)$  is the Artin L-function associated to the Galois module $\Hom(H_r, \mathbb{G}_m) \otimes_{\Z} \Q$ with local component $L_v(s, H_r)$ at $v$, and analogously for $L(s,M_r)$.  Similarly, $c_{r,v}(s)$ is the local component of $c_r(s)$ so that $M_{H_r,v}(s) f_{s,v}^0 =  c_{r,v}(s) \cdot f_{-s,v}^0$. Note that the ratio
$L_v(1,H_r) / L_v(1, M_r)$ is equal to $1/ \zeta_{E_v}(1)$ or $1/\zeta_{F_v}(1)$ in the various cases. 
\vskip 5pt

Now the value of $c_{r,v}(s)$ has been computed in \cite[\S 9]{Ik}, \cite{I1} and \cite[\S 9]{I2}, and is essentially a ratio of products of local Hecke L-functions. When one evaluates $c_{r,v}(s)$ at $s = \rho_{H_r}$, one obtains local Hecke L-values at $s >1$ except for a term $\beta_v(s)$ in the numerator which satisfies  $\beta_v(\rho_{H_r}) = \zeta_{E_v}(1)$ or $\zeta_{F_v}(1)$.  This term is cancelled by $L_v(1,H_r) / L_v(1, M_r)$, so that the Euler product on the RHS of (\ref{E:meas-const}) is absolutely convergent.
 \vskip 5pt

Now the RHS of (\ref{E:meas-const}) looks very much like ${\rm Res}_{s= \rho_{H_r}} c_r(s)$. Indeed, 
\[  {\rm Res}_{s= \rho_{H_r}} c_r(s)   = {\rm Res}_{s= \rho_{H_r}} \beta(s) 
 \cdot \left( \prod_v  \frac{L_v(1,G)}{L_v(1, M_v)} \cdot  c_{r,v}(\rho_{H_r}) \right), \]
 where $\beta(s)$ is the global zeta function with local component $\beta_v(s)$. Hence, we see that
 \[  \alpha  = {\rm Res}_{s = \rho_{H_r}} c_r(s) \cdot  {\rm Res}_{s=1} \left( \frac{L(s, M_r)}{L(s, H_r)} \right) /
{\rm Res}_{s = \rho_{H_r}} \beta(s). \]
 \vskip 5pt
 
 Now an examination of the results in \cite{Ik, I1, I2} shows that the function $\beta(s)$ has the form
 \[  \zeta_{\ast}(s - (\rho_{H_r} -1)) \quad \text{or} \quad \zeta_{\ast}(2s - (2 \rho_{H_r}-1)), \]
 where $\ast = E$ or $F$.   
 On the other hand,   ${\rm Res}_{s=1} L(s, M_r) / L(s, H_r)$ is  equal to 
 ${\rm Res}_{s=1} \zeta_{\ast}(s)$. Hence we see that 
 \[  {\rm Res}_{s=1} \left( \frac{L(s, M_r)}{L(s, H_r)} \right) /
{\rm Res}_{s = \rho_{H_r}} \beta(s)  = 1 \quad \text{or}  \quad 2  \]
in the two respective cases, so that
\[  \alpha = {\rm Res}_{s= \rho_{H_r}} c_r(s) \quad \text{or} \quad 2 \cdot {\rm Res}_{s= \rho_{H_r}} c_r(s).\]
\vskip 5pt

Indeed, a careful examination shows that the second case happens precisely when
\[  H_0 = \begin{cases}
\text{$\OO_1$ or trivial, when $\epsilon_0  =1$;} \\
\text{trivial, when $\epsilon_0 = 0$.} 
\end{cases} \]
One can then recast the expressions for $\alpha$ into a uniform formula:
\[  \alpha = \frac{\tau(H_r)}{\tau(H_0)} \cdot  \# \pi_0(H_r/ H_0) \cdot  {\rm Res}_{s= \rho_{H_r}} c_r(s).\]
By part (i) of the lemma, we then deduce that
\[  \alpha =  \frac{\tau(H_r)}{\tau(H_0)} \cdot \kappa_r, \]
as desired.
\end{proof}
\vskip 5pt

\subsection{\bf Measure constants.}
Motivated by Lemma \ref{L:kappa}(ii), for any maximal parabolic
subgroup $P(X_{r-r'})$ of $H_r$, with Levi subgroup $\GL(X_{r-r'})
\times H_{r'}$, we define a constant $\kappa_{r,r'}$ by the
requirement that
\begin{equation} \label{E:iwasawa}
\frac{1}{\tau(H_r)} \cdot   dh  = \kappa_{r,r'} \cdot \frac{1}{\tau(H_{r'})}  \cdot dm \cdot  dn \cdot dk \end{equation}
where $dm$ and $dn$ are the Tamagawa measures of $M(X_{r-r'})$ and
$N(X_{r-r'})$, respectively. In particular, $\kappa_r = \kappa_{r,0}$.
\vskip 5pt

\subsection{\bf Regularized theta integral.}
The {\bf regularized theta integral} is defined to be the function
\[  
B^{n,r}(s, \phi)(g) =  \frac{1}{\tau(H_r)  \cdot \kappa_r \cdot P_{n,r}(s)} \cdot
\int_{[H_r]} \theta_{n,r}( \omega_{n,r}(z)  \phi)(g,h)\,  E_{H_r}(s,h) \, dh. 
\]
When $r=0$ and $m_0>0$, we set $B^{n,0}(s,\phi)(g)=I_{n,0}(\phi)(g)$ by convention.

The integral converges absolutely at all points $s$ where
$E_{H_r}(s,h)$ is holomorphic, and defines a meromorphic function of
$s$ (for fixed $\phi$). We note that the functional equation
(\ref{E:spherical_functional_equation}) implies
\begin{equation} \label{E:B}
  B^{n,r}(s, \phi)  = c_r(s)  \cdot B^{n,r}(-s, \phi). 
\end{equation}
 
 \vskip 5pt
 
 We can now  explain why the extra condition $r\leq n$ as in
 (\ref{E:range2}) is necessary. Indeed, if this condition is not
 satisfied, one cannot hope to regularize the theta integral in the same way as above.  
 Otherwise, one may integrate the regularized theta
 kernel against the Eisenstein series associated to 
the family of principal series representations induced from the
minimal parabolic subgroup of $H_r$; this gives a meromorphic function
in $r$ complex variables whose iterated residue at a specific point is
the regularized theta lift of the constant function of $H_r$. 
If this meromorphic function is not identically zero, then a Zariski open set of these
minimal principal series representations will have a nonzero theta lift
to $G_n$. For this to happen, it is necessary that $n \geq r$.
\vskip 10pt

 \subsection{\bf First and Second term range.}
 The analytic behavior of $B^{n,r}(s, \phi)$ at the point of interest
 $s = \rho_{H_r}$ is described as follows:
  \vskip 5pt

 \begin{Lem}
\begin{enumerate}[(i)]
\item If $m \leq d(n)$, then $B^{n,r}(s, \phi)$ has a pole of order at
  most $1$  at $s = \rho_{H_r}$. 
\vskip 5pt
\item  If $d(n) <  m \leq 2 \cdot d(n) $ (under the conditions in  (\ref{E:range})), then $P_{n,r}(s)$ has a
  simple zero at  $s = \rho_{H_r}$ and $B^{n,r}(s, \phi)$ has a pole of
  order at most $2$ at $s = \rho_{H_r}$. 
\end{enumerate}
\end{Lem}
\begin{proof}
See \cite[Lemma 5.5.6]{KR5},  \cite[Lemma 2.2]{I2} and \cite[\S 2]{T1}.
\end{proof}
\vskip 5pt

Because the analytic behavior of $B^{n,r}(s, \phi)$  at $s=\rho_{H_r}$
differs for different ranges of the pair $(n, m)$, we introduce the
following terminology outside the Weil's convergent range:
\begin{align*}
m<d(n)&:\text{the first term range};\\
m=d(n)&:\text{the boundary case};\\
d(n)<m\leq 2\cdot d(n)&:\text{the second term range}.
\end{align*}

Observe that if $V_r$ and $V_{r'}$ are complementary spaces, then
$(W_n, V_r)$ is in the first term range if and only if $(W_n, V_{r'})$
is in the second term range. Moreover, in the boundary case, we have
$V_r = V_{r'}$. Note also that when $m > 2 \cdot d(n)$, one is automatically in the Weil's convergent range.

\vskip 10pt
\subsection{\bf Laurent expansion.}
In the first term range and the boundary case, we may thus consider
the Laurent expansion at $s= \rho_{H_r}$: 
\[  B^{n,r}(s, \phi)  =  \frac{B^{n,r}_{-1}(\phi)}{s-\rho_{H_r}} +  B^{n,r}_0(\phi) + \cdots  \]
whereas in the second term range, we have
\[
 B^{n,r}(s, \phi)  =  \frac{B^{n,r}_{-2}(\phi)}{(s-\rho_{H_r})^2} +
   \frac{B^{n,r}_{-1}(\phi)}{s-\rho_{H_r}} +  B^{n,r}_0(\phi) + \cdots.
\]
The functions $B^{n,r}_d(\phi)$ are automorphic forms on $G_n$ and the
linear map $\omega_{n,r}\rightarrow\Aut(G_n)$ given by $\phi \mapsto 
B^{n,r}_d(\phi)$ is $G_n(\A)$-equivariant. Moreover, if $B^{n,r}_d$ is
the leading term in the Laurent expansion, it is $H_r(\A)$-invariant.
\vskip 5pt

The purpose of the regularized Siegel-Weil formula is to give an
alternative construction of the automorphic forms
$B^{n,r}_{-1}(\phi)$ and $B^{n,r}_{-2}(\phi)$. 
 \vskip 10pt

 \subsection{\bf Non-Siegel Eisenstein Series}\label{S:non_Siegel_Eisen}
 We now recall how the regularized theta integral $B^{n,r}(s, \phi)$
 can be expressed as a non-Siegel Eisenstein series on $G_n$.
Namely, one can massage the definition of $B^{n,r}(s, \phi)$ by
unfolding the auxiliary Eisenstein series $E_{H_r}(s,-)$.  As shown in \cite[\S
5.5, p.48-50]{KR5} and using Lemma \ref{L:kappa}(ii), one obtains:
\begin{Prop} \label{P:non-Siegel}
\[   B^{n,r}(s, \phi)  = E^{n,r}(s,  f^{n,r}(s, \pi_{K_{H_r}}(\phi))). \]
 \end{Prop}
\noindent The following explains the notation in the proposition: 
 \begin{enumerate}[$\bullet$]
 \item $E^{n,r}(s, -)$ refers to the Eisenstein series associated to
   the family of induced representations
 \[   I^n_r(s,\chi) =  \Ind_{Q(Y_r)(\A)}^{G_{n}(\A)}  (\chi \circ \det)
 |{\det}_{Y_r}|^s \boxtimes \Theta_{n-r,0}(V_0)  \]
 where we recall that the Levi factor of $Q(Y_r)$ is $L(Y_r) \cong
 \GL(Y_r) \times G_{n-r}$ and $\Theta_{n-r,0}(V_0)$ is as defined in
 \S \ref{SS:theta}. This is what we call the non-Siegel Eisenstein
 series  (though if $r = n$, it is of course a Siegel-Eisenstein
 series).
 
 \item $\pi_{K_{H_r}}$ is the projection operator onto the $K_{H_r}$-fixed subspace, defined by
 \[  \pi_{K_{H_r}}(\phi) =  \int_{K_{H_r}}  \omega_{n,r}(k)( \phi) \, dk. \]
 
 \item For $\phi \in \mathcal{S}(Y_n^* \otimes V_r)(\A)$, 
 \[  f^{n,r}(s,\phi) \in I^n_r(s,\chi) \]
  is a meromorphic section defined for ${\rm Re}(s) \gg 0$ by
 \begin{align*}   f^{n,r}(s,\phi)(g)  &= 
 \int_{\GL(X_r)(\A)} I_{n-r,0}(\omega_{n,r}(g,a)\mathcal{F}_{n,r}(\phi)
 ( \beta_0)(0 , -) )\,\cdot   |{\det}_{X_r}(a)|^{s  - \rho_H} \, da \\
 &=  \int_{\GL(X_r)(\A)} I_{n-r,0}(\omega_{n,r}(g)\mathcal{F}_{n,r}(\phi)
 ( \beta_0 \circ a )(0 , -) )\,\cdot   |{\det}_{X_r}(a)|^{s  +n -
   \rho_H} \, da.
\end{align*}
 Here we note that  $\mathcal{F}_{n,r}(\phi)$ is a Schwartz function on
 $X_r^* \otimes W_n = \Hom(X_r, W_n)$ taking values in 
 \[  \mathcal{S}(Y_n^* \otimes V_0)(\A)  =  \mathcal{S}(Y_r^* \otimes V_0)(\A)
 \otimes \mathcal{S}({Y'}^*_{n-r}\otimes V_0)(\A), \]
 and 
  \[  \beta_0 \in \Hom(X_r, W_{n}) \]
 is defined by
 \[  \beta_0(x_i)  = y_i  \quad \text{for $i = 1,\dots,r$,} \]
so that
 \[  \mathcal{F}_{n,r}(\phi)(\beta_0 \circ a)(0,-) \in \mathcal{S}({Y'}^*_{n-r} \otimes V_0)(\A). \]
The integral defining $f^{n,r}(s,\phi)$ converges when 
\[  {\rm Re}(s) > \frac{m}{2} - \frac{2n-r + \epsilon_0}{2} \]
 and extends to  a meromorphic section of $I^n_r(s, \chi)$ (since it
 is basically a Tate-Godement-Jacquet zeta integral, as we explain in
 the next section).  When $r=0$ and $m_0>0$, we set $f^{n,0}(s,\phi)(g)=I_{n,0}(\phi)(g)$ by convention.
  \end{enumerate}
  \vskip 5pt
  
  Thus we have expressed $B^{n,r}(s, \phi)$ as an Eisenstein series on
  $G_n$ associated to a meromorphic (non-standard) section of a family
  of non-Siegel principal series representations. However, our
  ultimate goal is to relate the first two Laurent coefficients of
  $B^{n,r}(s, \phi)$ to the Laurent coefficients of a Siegel
  Eisenstein series.  
 \vskip 15pt
 
 \section{\bf The section $f^{n,r}(s,\phi)$.} \label{sectionf}
 In this section, we  shall establish some important properties of the
 section $f^{n,r}(s,-)$ defined in \S \ref{S:non_Siegel_Eisen}, which
 will play a crucial role in the proof of the second term identity.
 \vskip 5pt
 
 \subsection{\bf Tate-Godement-Jacquet Zeta Integrals.}\label{zetaintegral}
We recall briefly the global  theory of the Tate-Godement-Jacquet zeta
integral developed in \cite{GJ}. If $\phi \in \mathcal{S}(M_{r \times
  r})(\A_E)$, then the Tate-Godement-Jacquet zeta integral (associated
to the trivial representation of $\GL_r(\A_E)$) is:
 \[ Z_r (s , \phi) :=  \int_{\GL_r(\A_E)} \phi(A)  \cdot |\det (A)|^s  \, dA. \]
 It converges for ${\rm Re}(s) > r$ and has meromorphic continuation to
 $\C$. Moreover, it satisfies a functional equation
 \begin{equation} \label{E:funct}
   Z_r(s,\phi)  = Z_r(r-s, \hat{\phi}),
\end{equation}
 where $\hat{\phi}$ is the Fourier transform of $\phi$ relative to the
 Haar measure on $M_{r\times r}(\A_E)$ determined by $\psi$ and the
 trace form on $M_{r \times r}(\A_E)$.  And $Z_r(s,\phi)$ has simple
 poles at $s=0$ and $r$, with $\mathrm{Res}_{s=0}Z_r(s,\phi)=-\phi(0)$
 and $\mathrm{Res}_{s=r}Z_r(s,\phi)=\hat{\phi}(0)$.
\vskip 5pt

\subsection{\bf The section  $\mathfrak{f}^{n,r}(s,\phi)$.}
 It will be convenient to express elements of $Y_n^* \otimes V_r$ as
 $3 \times 2$ matrices corresponding to the decompositions
 \[  
Y^*_n  = Y^*_r \oplus {Y'}^*_{n-r} \quad \text{and} \quad V_r  = X_r \oplus
V_0 \oplus X_r^*,
\]
so the first column of the matrix has entries from $Y^*_r\otimes X_r,
Y^*_r\otimes V_0$ and $Y^*_r\otimes X_r^*$ in this order, and the second
column has entries from ${Y'}^*_{n-r}\otimes X_r, {Y'}^*_{n-r}\otimes
V_0$ and ${Y'}^*_{n-r}\otimes X_r^*$.

 Then by definition
\begin{align*}
f^{n,r}(s, \phi)(g)  &= 
 I_{n-r,0} \Bigg( \int_{\GL(X_r)(\A)}\int_{({Y'}_{n-r}^* \otimes
   X_r)(\A)}\int_{(Y_r^* \otimes X_r)(\A)} \\
&\qquad\omega_{n,r}(g) \phi \begin{pmatrix}
 X_1 & X_2  \\
 0 & - \\
 0 & 0 \end{pmatrix} \cdot \psi(Tr(X_1A) ) \cdot
|\det(A)|^{s+ n - \rho_{H_r}} \, dX_1 \, dX_2 \, dA \Bigg).
\end{align*}
Now we set:
\begin{align*}  \mathfrak{f}^{n,r}(s,\phi)(g)(-)
&:= \int_{\GL(X_r)(\A)}\int_{({Y'}_{n-r}^* \otimes X_r)(\A)}\int_{(Y_r^* \otimes X_r)(\A)}\\
&\quad  \omega_{n,r}(g) \phi 
\begin{pmatrix} X_1 & X_2  \\ 0 & - \\0 & 0 \end{pmatrix}
\cdot \psi(Tr(X_1A) ) \cdot |\det(A)|^{s+ n - \rho_{H_r}} \, dX_1 \,
dX_2 \, dA.
\end{align*}
 In other words, $ \mathfrak{f}^{n,r}(s,\phi)$ is defined as
 $f^{n,r}(s,\phi)$ but with the anisotropic theta integral $I_{n-r,0}$
 suppressed. Then we note that
 \[  
  \mathfrak{f}^{n,r}(s,\phi) \in 
\Ind^{G_n(\A)}_{Q(Y_r)(\A)}  (\chi \circ \det) \cdot |\det|^s
\boxtimes \omega_{n-r,0}, 
\]
 with $\omega_{n-r,0}$ realized on $\mathcal{S}({Y'}^*_{n-r}\otimes V_0)(\A)$, and
 \begin{equation}\label{E:frak}   
f^{n,r}(s, \phi)(g)  =   I_{n-r,0} \left(\mathfrak{f}^{n,r}(s,
  \phi)(g) \right). 
\end{equation}
 Observe that  the formation of $\mathfrak{f}^{n,r}(s,-)$ is basically
 the Tate-Godemont-Jacquet zeta integral of a partial Fourier
 transform of $\phi$ (in the coordinate $X_1$).
 Since the anisotropic theta integral presents no convergence issues,
 we see that the meromorphic continuation of $f^{n,r}(s,\phi)$ follows
 immediately from the analytic theory of the Tate-Godement-Jacquet
 zeta integral. Moreover, applying the functional equation
 (\ref{E:funct}) for the Tate-Godement-Jacquet zeta integral, we
 deduce:
  \vskip 5pt
  
\begin{Lem} \label{L:tate}
One has:  
\[ 
\mathfrak{f}^{n,r}(s, \phi)(g)(-)=  \int_{\GL(X_r)(\A)}
\int_{({Y'}_{n-r}^* \otimes X_r)(\A)} \omega_{n,r}(g)\phi
\begin{pmatrix}
 A & X_2 \\
 0 & - \\
 0 & 0 \end{pmatrix} 
\cdot |\det(A)|^{-s  +r - n + \rho_{H_r}} \, dX_2 \, dA. \]
\end{Lem}
This is a somewhat simple expression to work with, as we shall see in
a moment.
\vskip 10pt
 
 \subsection{\bf Restriction and Ikeda map.}
Next we consider the restriction of the section $f^{n+1,r}(s, \phi)$
from $G_{n+1}$ to $G_n$; this restriction turns out
 to be related to the Ikeda map $\Ik^{n,r,r-1}$.
 More precisely, fix $\phi_1  \in \mathcal{S}(Y_1^* \otimes V_r)(\A)$
 satisfying:
 \vskip 5pt
 
 \begin{itemize}
 \item $\phi_1(0) = 1$;
 \item $\phi_1$ is $K_{H_r}$-invariant, so that $\pi_{K_{H_r}} \phi_1 = \phi_1$. 
 \end{itemize}
 \vskip 5pt
 
\noindent  For  any  $\phi \in \mathcal{S}({Y'}_n^* \otimes V_r)(\A)$, we set
 \[  \tilde{\phi} := \phi_1 \otimes \phi \in \mathcal{S}(Y_{n+1}^*
 \otimes V_r)(\A). \]
 Then 
 \[  \pi_{K_{H_r}}(\tilde{\phi}) = \phi_1 \otimes \pi_{K_{H_r}} \phi. \]
 We consider the restriction of $f^{n+1, r}(s, \tilde{\phi})$ to the
 subgroup $G_n \subset G_{n+1}$, where $G_n$ is the isometry group of
 the nondegenerate subspace $W_n = \langle y_2,...,y_{n+1},
 y_{n+1}^*,...,y_2^* \rangle \subset W_{n+1}$. With these conventions,
 the result is:
 \vskip 5pt
 
\begin{Prop}  \label{P:key2}
Suppose $r\geq 1$ and further assume $m_0>0$ when $r=1$. We have:
 \[ f^{n+1,r}(s, \pi_{K_{H_r}} \tilde{\phi})|_{G_{n}}
 = \alpha_r \cdot Z_1( -s  - (n+1-r) + \rho_{H_r}, \phi_1) \cdot
 f^{n,r-1}(  s + \frac{1}{2},  \Ik^{n,r, r-1}(\pi_{K_{H_r}}\phi)). \] 
 Here  $Z_1(s, \phi_1)$ is the Tate zeta integral
  \[  Z_1(s, \phi_1) =  \int_{\A_E^{\times}} \phi_1 (t y_1^* \otimes x_1)  \cdot |t|^s \, dt, \]
  and   $\alpha_r$ is a scalar defined by (\ref{E:alphar}) below. 
 \end{Prop}
 \begin{proof}
 By (\ref{E:frak}), it suffices to show the identity with
 $\mathfrak{f}^{n+1,r}(s,\phi)$ in place of $f^{n+1,r}(s,\phi)$ (since
 $I_{(n+1)  -r, 0 }  = I_{n - (r-1), 0}$).  Now  for $g\in G_n(\A)$,
 Lemma \ref{L:tate} gives
\begin{align*}
&  \mathfrak{f}^{n+1,r}(s, \pi_{K_{H_r}} \tilde{\phi})(g)(-)\\
=&\int _{\GL(X_r)(\A)}    \int_{({Y'}_{n+1-r}^* \otimes X_r)(\A)} 
(\omega_{n+1,r}(g) \pi_{K_{H_r}}\tilde{\phi})\begin{pmatrix}
A & Y  \\
0& - \\
0 & 0
 \end{pmatrix}    
\cdot |\det(A)|^{-s - (n+1-r) + \rho_{H_r}} \, dY \, dA\\
=& \int _{\GL(X_r)(\A)}    \int_{({Y'}_{n+1-r}^* \otimes X_r)(\A)} 
\phi_1\otimes\omega_{n,r}(g) \pi_{K_{H_r}}\phi\begin{pmatrix}
A & Y  \\
0& - \\
0 & 0
 \end{pmatrix}    
\cdot |\det(A)|^{-s - (n+1-r) + \rho_{H_r}} \, dY \, dA,
\end{align*}
where, for the second equality, we have used the fact that 
\[ \omega_{n+1,r}(g)
\pi_{K_{H_r}}\tilde{\phi}=\phi_1\otimes\omega_{n,r}(g)
\pi_{K_{H_r}}\phi, \]
since  $G_n(\A)$ acts trivially on $\phi_1$.
\vskip 5pt

Assume first that $r=1$ and $m_0>0$. Since $\phi_1$ is a function on the first column of the matrix in the integrand, we can just
decompose the integral as
\begin{align*}
&  \mathfrak{f}^{n+1,r}(s, \pi_{K_{H_r}} \tilde{\phi})(g)(-)\\
=&\left(\int _{\GL(X_r)(\A)}\phi_1\begin{pmatrix}A\\0\\0\end{pmatrix}
\cdot |\det(A)|^{-s - (n+1-r) + \rho_{H_r}}
\,dA\right)
\cdot\left(\int_{({Y'}_{n}^* \otimes X_r)(\A)} 
(\omega_{n,r}(g) \pi_{K_{H_r}}\tilde{\phi})\begin{pmatrix}
Y  \\
- \\
0
 \end{pmatrix}    
  \, dY\right).\\
\end{align*}
Noting that the first factor is the Tate integral $Z_1( -s  - (n+1-r) +
\rho_{H_r}, \phi_1)$ and the second one is $\mathfrak{f}^{n, r-1} (s + \frac{1}{2},  
\Ik^{n,r,r-1}(\pi_{K_{H_r}}\phi))(g)(-)$, we see that the proposition holds with $\alpha_1=1$.
\vskip 5pt

Assume now that $r>1$. To extract the first column of $A$, we  use the Iwasawa
decomposition on $\GL(X_r)(\A)\cong \GL_r(\A_E)$. Namely, we have
  \[   
A = k \cdot\begin{pmatrix}
  1 & u \\
  0 & 1 \end{pmatrix}\cdot\begin{pmatrix}
  t & 0 \\
  0 & B \end{pmatrix} =  
  k \cdot  \begin{pmatrix}
  t & uB \\
  0 & B \end{pmatrix}    
\]
  with
  \begin{itemize}
  \item $t \in \A_E^{\times}$;  
  \item $u \in (x_1 \otimes {X'}_{r-1}^*)(\A)\cong \A_E^{r-1}$;
  \item $B \in \GL(X'_{r-1})(\A) \cong \GL_{r-1}(\A_E)$;
  \item $k$ is an element in a maximal compact subgroup $K = K_{H_r} \cap \GL(X_r)(\A)$  of $\GL(X_r)(\A)$. 
  \end{itemize} 
  Accordingly, we have a constant $\alpha_r$ such that
   \begin{equation} \label{E:alphar}
    \int_{\GL_r(\A_E)} \varphi(A)  \,  dA =   \alpha_r  \cdot
    \int_{\A_E^{\times}} \int_{\GL_{r-1}(\A_E)} \int_{\A_E^{r-1}}
    \int_K   \varphi \left( k \cdot \left(  \begin{matrix} t & uB \\
    0 & B \end{matrix} \right) \right) \,  dt  \, dB \, du \, dk 
\end{equation}
    for any $\varphi \in C^{\infty}_c(\GL_r(\A_E))$.  Indeed, by
    \cite[Lemma 9.1]{I2}, one has:
  \[  \alpha_r  = \frac{{\rm Res}_{s=1} \zeta_E(s)}{\zeta_E(r)}. \]
  \vskip 5pt
 
 Since the function $\pi_{K_{H_r}} \tilde{\phi}$ is
 $K_{H_r}$-invariant, the integral over $dk$ simply gives the value
 $1$ and thus disappears. 
 Hence, after using the Iwasawa decomposition, we obtain:
  \begin{align*}
  \mathfrak{f}^{n+1,r}(s, \pi_{K_{H_r}} \tilde{\phi})(g)(-) =&
  \alpha_r \cdot    \int_t \, \int_{B} \,   \int_{u} \,    \int_{Y}  
 \phi_1\otimes\omega_{n,r}(g) \pi_{K_{H_r}}{\phi}
\begin{pmatrix}
  \begin{smallmatrix}
  t & uB  \\
  0 & B \end{smallmatrix}  & Y \\   
  0 & - \\
  0 & 0  
\end{pmatrix}\\    
&\qquad
|t|_E^{-s -(n+1-r)  + \rho_{H_r}}\cdot |\det(B)|^{-s -(n+1-r) +
  \rho_{H_r}}\,    dt \,  dB\,  du \,  dY,
\end{align*}
where $t\in\A^\times_E, B\in\GL_{r-1}(\A_E), u\in \A_E^{r-1}$ and $Y\in
({Y'}^*_{n-r}\otimes X_r)(\A)$.
  \vskip 5pt
             
 Noting that $\phi_1$ is a function of the
 first column of the $m \times (n+1)$ matrix in the integrand, we see
 that the above integral is equal to the product of
 \begin{equation} \label{E:zeta}   
\alpha_r \cdot  \int_{t\in \A_E^{\times}}  \,  \phi_1 
\begin{pmatrix}
 t \\
 0 \\
 0\\
 0
\end{pmatrix}
\cdot |t|_E^{-s -(n+1-r)+\rho_{H_r}} \,  dt
 \end{equation}
 and
 \begin{equation} \label{E:f}
 \int_u \int_B   \int_Y   (\omega_{n,r}(g) \pi_{K_H} \phi)
 \begin{pmatrix}
\begin{smallmatrix}
         u \\
         B 
\end{smallmatrix} & Y \\
         0 & - \\
         0 & 0
\end{pmatrix}
         \cdot  |\det (B)|^{-s -1 -(n+1-r)+ \rho_{H_r}}  \, dY \, dB \,
         du,   
\end{equation}
where we have made the substitution $u \mapsto uB^{-1}$. Now observe
that (\ref{E:zeta}) is equal to the Tate zeta integral 
\[  
\alpha_r \cdot Z_1(-s-(n+1-r) +\rho_{H_r}, \phi_1),
\]   
where $\phi_1$ is actually the function defined by $t\mapsto \phi_1(t y_1^*\otimes x_1)$.

On the other hand, in (\ref{E:f}), we may write the $m \times n$
matrix as:
\[          
\begin{pmatrix}
         u &\multirow{2}{*}{Y}\\
         B &\\
         0 & - \\
         0 & 0 
\end{pmatrix} = 
\begin{pmatrix}
\multicolumn{2}{c}{z}\\
             B & Y^{\dagger} \\
             0 & - \\
             0 & 0 \\
\end{pmatrix}
\]
with  $z \in \A_E^n$ occupying the first row of the matrix and
$Y^{\dagger}$ a $r  \times (n-r)$ matrix. Then, on replacing the
integrals over $u$ and $Y$ by  integrals over $z$ and $Y^{\dagger}$,
and noting that
\[  \rho_{H_r}  = \frac{m-r - \epsilon_0}{2} 
= \frac{m_0 + r-\epsilon_0}{2} =  \rho_{H_{r-1}} + \frac{1}{2}, 
\]
we see that (\ref{E:f}) is equal to
\[  
\int_B   \int_z \int_{Y^{\dagger}}        ( \omega_{n,r}(g) \pi_{K_H} \phi)
\begin{pmatrix}
\multicolumn{2}{c}{z}\\
             B & Y^{\dagger} \\
             0 & - \\
             0 & 0 \\
\end{pmatrix} 
\cdot |\det(B)|^{-(s+\frac{1}{2} )+ r -1 -n  +\rho_{H_{r-1}}} 
\, dz\, dY^{\dagger}\,dB
\] 
which is nothing but
 \[  
 \mathfrak{f}^{n, r-1} (s + \frac{1}{2},  
\Ik^{n,r,r-1}(\pi_{K_{H_r}}\phi))(g)(-).
\]
This completes the proof of the proposition. 
\end{proof}
 \vskip 10pt

\begin{Rmk}\label{rm_section}
 It is natural to extend the definition of
 $\mathfrak{f}^{n,r}(s,\phi)$ and $f^{n,r}(s,\phi)$ by
 regarding them as  functions on $G_n \times H_r$, rather than simply
 functions on $G_n$.  More precisely, we set
 \[  \mathfrak{F}^{n,r}(s,\phi)(g,h)(-) = 
  \int_{\GL(X_r)(\A)} \omega_{n,r}(g,h)(\mathcal{F}_{n,r}(\phi)
 ( \beta_0 \circ a)(0 , -) )\,\cdot   |{\det}_{X_r}(a)|^{s + n -
   \rho_H} \, da, 
\]
 and
 \[  F^{n,r}(s,\phi)(g,h)  :=  I_{n-r,0}(\mathfrak{F}^{n,r}(s,\phi)(g,h)). \]
Then
\[   \mathfrak{F}^{n,r}(s,\phi)|_{G_n} =  \mathfrak{f}^{n,r}(s,\phi) \quad \text{and} \quad
 F^{n,r}(s,\phi)|_{G_n} =  f^{n,r}(s,\phi).  \]
 Observe that
 \[  \mathfrak{F}^{n,r}(s,\phi) \in \Ind^{G_n(\A) \times H_r(\A)}_{Q(Y_r)(\A) \times P(X_r)(\A)}  
\left( ((\chi \circ {\det}_{Y_r})\cdot  |{\det}_{Y_r}|^s \boxtimes
  |{\det}_{X_r}|^{-s})  \boxtimes \omega_{n-r,0} \right), \]
 and
 \[ F^{n,r}(s,\phi)  \in    I^n_r(s,\chi) \boxtimes I_{H_r}(-s). \]
 Thus, 
 \[  F^{n,r}(s,-) :  \omega_{n,r} \longrightarrow I^n_r(s,\chi) \boxtimes I_{H_r}(-s) \]
 is a meromorphic family of $G_n(\A) \times H_r(\A)$-equivariant maps
 which explicitly realizes the  Howe duality correspondence. 
\end{Rmk}
 \vskip 15pt

 \section{\bf Siegel Principal Series} \label{S:siegel}
 
 In this section, we  review the theory of the Siegel principal series
 representation and its relation to the Weil representation over a
 local field. Hence throughout this section, we let $k$ be a
 (possibly archimedean) local field and let $L$ be $k$ if
 $\epsilon_0\neq 0$ and a quadratic \'{e}tale algebra over $k$ if
 $\epsilon_0=0$, so $L$
 is either a quadratic extension of $k$ or $k\times k$. We assume the
 spaces $W_n$ and $V_r$ are defined over $L$. Thus, unlike
  the global case, we include in our discussion 
  the split quadratic algebra $L=k\times k$, in which case 
 $G_n(k) \cong \GL_{2n}(k)$ and $H_r(k)\cong \GL_m(k)$.  In addition,  in
 this section, $\chi$ is any unitary character on $L^\times$.
\vskip 5pt

We shall
 consider the degenerate principal series representation
\[  I^n_n(s,\chi) := \Ind_{Q(Y_n)(k)}^{G_n(k)} (\chi \circ \det) \cdot |\det|^s. \]
 Here we recall that when $G_n(k) = \Mp_{2n}(k)$, then
 \[  I^n_n(s,\chi)  = \Ind_{Q(Y_n)(k)}^{G_n(k)} (\chi_{\psi}
 \circ \det)\cdot (\chi \circ \det)  \cdot |\det|^s. \] 
 \vskip 10pt

\subsection{\bf Degenerate principal series: Non-archimedean case.}
We first assume that $k$ is non-archimedean. The following proposition
summarizes the reducibility points of $ I^n_n(s, \chi)$.
\vskip 10pt
 
 \begin{Prop} \label{P:redu}
 Let $k$ be non-archimedean.
\begin{enumerate}[(i)]
\item The representation $I^n_n(s, \chi)$ is multiplicity-free.
\item Assume that $s \in \R$ and $\chi$ is unitary. Then
  $I^n_n(s,\chi)$ is reducible precisely in the situations given in
  the following table.
\begin{center}
{\renewcommand{\arraystretch}{1.5}
\begin{tabular}{|c|c|}
\hline 
$\epsilon_ 0 = 0$, $L\neq k\times k$ & $\chi|_{k^{\times}} = \chi_L^m$ and $s =
\tfrac{m-d(n)}{2}$ with $0 \le m \le 2 d(n)$ \\ \hline 
$\epsilon_ 0 = 0$, $L= k\times k$ & $\chi|_{k^{\times}} = 1$ and $s =
\tfrac{m-d(n)}{2}$ with $0 \le m \le 2 d(n)$ and $m \ne n$ \\ \hline 
$\epsilon_0 = -1$ & $\chi^2 =1$ and $s = \tfrac{m-d(n)}{2} $ with $m$
even and $0 \le m \le 2d(n) $ \\ \hline 
$\epsilon_0  =1$, $G_n=\Mp_{2n}$ & $\chi^2 = 1$ and $s =
\frac{m-d(n)}{2}$ with $m$ odd and $1\leq m \leq 2n+1$ \\ \hline
$\epsilon_0 =1$, $G_n = \Sp_{2n}$ & 
 (a) $\chi^2 =1$ but $\chi \ne 1$ and $s = \tfrac{m-d(n)}{2}$, with
 $m$ even  $2 \le m \le 2n$ \\ 
& or \\
 & (b) $\chi=1$  and $s = \tfrac{m-d(n)}{2} $ with   $m$ even and  $0
 \le m \le 2n+2$ \\
 \hline 
\end{tabular}
}
\end{center}
\end{enumerate}
\end{Prop}
\begin{proof}
See \cite{KR4, KS, S2, Y1}.
\end{proof}

Observe that the points of reducibility are essentially the points 
\[  s_{m,n} =  \frac{m - d(n)}{2},  \quad \text{ $0 \leq m \leq 2 d(n)$} \]
with some extra conditions in the various cases:
\begin{itemize}
\item when $\epsilon_0 = 1$, $m$ is even when $G_n$ is symplectic, and
  $m$ is odd when $G_n$ is metaplectic; moreover,  there is a quadratic
  space $V$ over $k$ of dimension $m$ and discriminant corresponding
  to the quadratic character $\chi$; 
  \item when $\epsilon_0  = -1$, there is a symplectic space $V$ of dimension $m$ (so that $m$ is even);
  \item when $\epsilon_0 = 0$ and $L = k \times k$, $m \ne n$. 
\end{itemize}

\vskip 5pt

\subsection{\bf Set $X_n$.}
For both archimedean and non-archimedean $k$, let us define
\[
X_n:=\{(s,\chi): (s,\chi) \text{ appears in the table of Proposition
  \ref{P:redu}}.\}
\]
If $k$ is non-archimedean, this set is precisely
the set of all $(s,\chi)$ with $s\in\R$ and $\chi$ unitary such that
$I_n^n(s,\chi)$ is reducible. If $k$ is archimedean, as we will see,
$I_n^n(s,\chi)$ is also reducible for $(s,\chi) \in X_n$, though
there are more points of reducibility which will play no
role in this paper.
\vskip 5pt

\subsection{\bf Siegel-Weil sections. }
Assume $k$ is archimedean or  non-archimedean. The structure of
$I^n_n(s,\chi)$ at the points of reducibility is
intimately connected with the theory of local theta
correspondence. To describe this, we recall the following construction.  For
the local Weil representation
$\omega_{n,r}$ realized on $\mathcal{S}(Y_n^* \otimes V_r)(k)$, we set
\[  
\Phi^{n,r}(\phi)(g)  = (\omega_{n,r}(g) \phi)(0), \quad \text{ for
  $\phi \in \mathcal{S}(Y_n^* \otimes V_r)(k)$.} 
\]
This defines an $H_r(k)$-invariant and $G_n(k)$-equivariant map
\[ 
\Phi^{n,r}: \mathcal{S}(Y_n^* \otimes V_r)(k) \longrightarrow
I_n^n(\frac{m - d(n)}{2},\chi),
\]
with $\chi = \chi_{V_0}$. 
We call the space of standard sections associated to the image of
$\Phi^{n,r}$ the space of {\em Siegel-Weil sections} associated to
$V_r$. 
\vskip 5pt

\subsection{\bf The representation $R_n(V_r)$.} 
Assume again that $k$ is  archimedean or  non-archimedean. Since
$\Phi^{n,r}$ is $H_r(k)$-invariant, it factors through the
maximal $H_r(k)$-invariant quotient $R_n(V_r)$ of $\omega_{n, r}$. It
was shown by Rallis that one has
\[  
\Phi^{n,r}:  \mathcal{S}(Y_n^* \otimes V_r)(k) \twoheadrightarrow
R_n(V_r) \hookrightarrow I^n_n(\frac{m - d(n)}{2},\chi), 
\]
so that the image of $\Phi^{n,r}$ is isomorphic to $R_n(V_r)$.  
 Thus $R_n(V_r)$ is a submodule of $I^n_n(\frac{m - d(n)}{2},\chi)$. 
 \vskip 10pt
 
 \subsection{\bf Structure of $I^n_n(s,\chi)$.} \label{SS:str} 
 Assume $k$ is non-archimedean. We can now describe the module
 structure of $I^n_n(s, \chi)$ at a point of reducibility.   
   For a point of reducibility
  $s_{m,n} = \frac{m - d(n)}{2} $ as given in Proposition
  \ref{P:redu}, we let $V$ and $V'$ be given as follows:
  \vskip 5pt
  
\begin{itemize}
  \item when $\epsilon_0 = 0$ and $L\neq k\times k$, $V$ and $V'$ are the two
    $\epsilon$-Hermitian spaces of dimension $m$ with $\disc(V) \in
    N_{L/k}L^{\times}$ and $\disc(V') \notin N_{L/k}L^{\times}$;
\item when $\epsilon_0 = 0$ and $L=k\times k$, $V$ is the unique $\epsilon$-Hermitian space of dimension $m$ and we set $V' = 0$;
  \item when $\epsilon_0 = -1$, $V$ is the unique symplectic space of
    dimension $m$;
  \item when $\epsilon_0 = 1$, $V$ and $V'$ are the two quadratic
    spaces of dimension $m$ with $V$ having  Hasse-Witt invariant $1$ 
    and $V'$ having  Hasse-Witt invariant $-1$ (here, a split quadratic space has Hasse-Witt invariant $1$).
\end{itemize}
Note that $V'$ is not defined when $\epsilon_0 = -1$ and may not exist
in the other cases when $m$ is too small.   We have seen that 
$R_n(V)$ and $R_n(V')$ are submodules of $I^n_n(s_{m,n}, \chi)$.
Here, when $\epsilon_0 = -1$, so that $G_n \cong \OO_{n,n}$, we
interpret 
\[  R_n(V') = R_n(V) \otimes {\det}_{G_n}, \]
and when $V'$ does not exist in the other cases, we interpret
$R_n(V') = 0$. With these notations, we have: 
 \vskip 5pt
 
 \begin{Prop} \label{P:str}
 Assume that $k$ is non-archimedean. 
 \vskip 5pt
 \begin{enumerate}[(i)]
 \item If $s_{m,n} = \frac{m - d(n)}{2}  <  0$, then $R_n(V)$ and
   $R_n(V')$ are irreducible unitary and $R_n(V) \oplus R_n(V')$ is
   the maximal semisimple submodule of  $I^n_n(s_{m,n}, \chi)$. The
   quotient
 \[ I^n_n(s_{m,n}, \chi)/  (R_n(V) \oplus R_n(V')) \]
is irreducible.
\item If $s_{m,n} = 0$, then
\[  I^n_n(0, \chi)  =  R_n(V) \oplus R_n(V') . \]
\item If $s_{m,n} >0$, then
\[  I^n_n(s_{m,n}, \chi)  =  R_n(V) +  R_n(V') \]
  and the intersection
  \[  R_n(V) \cap R_n(V') \]
  is  the unique irreducible submodule of $I^n_n(s_{m,n},\chi)$. If
  $V_c$ and $V_c'$ are the complementary spaces of $V$ and $V'$
  (relative to $W_n$), so that 
 $\dim V + \dim V_c = 2 \cdot d(n)$ with $\dim V_c < \dim V$, then one has
 \[  R_n(V)/R_n(V) \cap R_n(V') \cong R_n(V_c) \quad \text{and} \quad
 R_n(V')/R_n(V) \cap R_n(V')  \cong R_n(V'_c). \]
 Thus there is a short exact sequence
  \[  \begin{CD}
 0 @>>>  R_n(V) \cap R_n(V') @>>> I^n_n(s_{m,n}, \chi) @>>> R_n(V_c)
 \oplus R_n(V'_c) @>>> 0. \end{CD} \]
\end{enumerate}
 \end{Prop} 
 \begin{proof}
See \cite{KR4, KS, S2, Y1}.
\end{proof}
 
 \vskip 5pt
 
\subsection{\bf Archimedean case.}
When $F = \R$ or $\C$, the structure of the degenerate principal
series  $I^n_n(s, \chi)$ has also been completely determined in
\cite{L1, L2},\cite{HL}, \cite[Appendix A]{Lo}, and \cite{LZ3,
  LZ4}. Moreover, the relation of the submodules to the theory of
local theta correspondence was determined in
\cite{LZ1,LZ2,LZ3,LZ4}. The results are too intricate to state here,
but we note the following, which is sufficient for our purposes:
\vskip 5pt

\begin{Prop} \label{P:arch1}
Assume that $k = \R$ or $\C$.
\begin{enumerate}[(i)]
\item As a representation of the maximal compact subgroup $K$ of
  $G_n(k)$, the representation $I^n_n(s,\chi)$ is independent of $s$
  and is multiplicity-free  .
\item For $(s_{m,n},\chi) \in X_n$ with $s_{m,n} \geq 0$, 
\[  I^n_n(s_{m,n}, \chi)  = \sum_V  R_n(V) \]
where the sum runs over all $\epsilon$-Hermitian spaces $V$ of
dimension $m$ with the following additional condition and  convention:
\begin{itemize}
 \item if $\epsilon_0 = 1$, $\chi_V = \chi$;
\item if $\epsilon_0 = -1$, so that there is a unique symplectic space
  $V$ of dimension $m$, the RHS is $R_n(V) + R_n(V) \otimes {\det}$. 
\end{itemize}
 \end{enumerate}
\end{Prop}
 So, for example, when $k = \C$ and $(s_{m,n}, \chi) \in X_n$, one has:
\[ 
 I^n_n(s_{m,n}, \chi)  = \begin{cases}
R_n(V) \text{  if $\epsilon_ 0 \ne -1$;} \\
R_n(V) + R_n(V) \otimes {\det}, \text{  if $\epsilon_0 =
  -1$.} \end{cases} 
\] 
 
 \vskip 5pt
 
  \subsection{\bf Local intertwining operator.}
 We now let $k$ be an arbitrary local field. Define 
\[
c=\begin{cases} \text{the identity},&\text{if $L=k$};\\
\text{the non-trivial element in } \operatorname{Aut}(L/k),&\text{if $L$ is an \'etale
  quadratic $k$-algebra.}
 \end{cases}
\]
For each character $\chi$
 on $L^\times$, we define $\chi^c$ to be the character on $L^\times$ such that
$\chi^c(x)=\chi(c(x))$. 

There is a standard intertwining operator 
 \[  
M_n(s,\chi):  I^n_n(s, \chi) \longrightarrow I^n_n(-s, (\chi^c)^{-1}).
\] 
One may normalize this intertwining operator following \cite{LR} (see also
\cite[\S 8]{GI2}). We will not recall the precise definition of this normalization
here, but simply note that the normalized operator $M^*(s, \chi)$
satisfies the functional equation
 \[  M^*_n( -s, (\chi^c)^{-1}) \circ M^*_n(s, \chi)  = 1. \]
 We note that for $(s, \chi) \in X_n$, $(\chi^c)^{-1} = \chi$.
 \vskip 5pt
 
 We are particularly interested in the behavior of $M^*_n(s,\chi)$ at $s = 0$.
 \vskip 5pt
 
 \begin{Lem} \label{L:Mat0}
\begin{enumerate}[(i)]
 \item The normalized intertwining operator $M_n^*(s,\chi)$ is
   holomorphic at $s = 0$, and satisfies
 \[  M_n^*(0,\chi)^2 = 1. \]
\item If $k$ is non-archimedean, $M_n^*(0,\chi)$ acts as $+1$ on
  $R_n(V)$ and $-1$ on $R_n(V')$. 
\item The derivative ${M_n^*}'(0,\chi)$ commutes with $M_n^*(0,\chi)$
  and preserves each irreducible summand of $I^n_n(0,\chi)$.  
\end{enumerate}
 \end{Lem}
 \vskip 5pt
 
 \begin{proof}
 (i) The holomorphy of $M_n^*(s,\chi)$ at $s = 0$ follows from the
 fact that $M^*_n(-s, \chi) \circ M^*_n(s, \chi)  =1$ and $I^n_n(0,
 \chi)$ is semisimple. The holomorphy implies the remaining part of (i).
\vskip 5pt

(ii) That it indeed acts as in the lemma was checked in
 \cite{GT} when $\epsilon_0  =-1$, \cite[Lemma 7.2]{GS} when $\epsilon_0 =1$ and
 \cite{HKS, KS} when $\epsilon_0 = 0$; it
 is responsible for the
 phenomenon known as the local theta dichotomy, which was shown in
 \cite{HKS} and \cite{GS}.
 \vskip 5pt
 
 (iii) Differentiating the identity $M^*_n(-s, \chi) \circ M^*_n(s,
 \chi)  =1$ and evaluating at $s = 0$, we obtain
 \[  {M_n^*}'(0,\chi) \circ M_n^*(0,\chi) = 
 M_n^*(0,\chi) \circ  {M_n^*}'(0,\chi), \]
 as desired. If $k$ is non-archimedean, 
  the two irreducible summands of $I_n^n(0,\chi)$ are different
  eigenspaces for the action of $M^*_n(0,\chi)$ and thus each summand
  is  preserved by the derivative  ${M_n^*}'(0,\chi)$. In the
  archimedean case, we observe that  ${M_n^*}'(0,\chi)$ is
  $K$-equivariant. Indeed, one has
  \[  M_n^*(s,\chi) \circ I_n^n(s, \chi)(g)  =  I_n^n(-s,\chi)(g) \circ M_n^*(s,\chi) \]
  for all $g \in G_n$, and for $k \in K$, the operator
  $I_n^n(s,\chi)(k)$ is independent of $s$. Thus differentiating the
  above equation and evaluating at $s = 0$, we obtain
 \[  {M_n^*}'(0,\chi) \circ I_n^n(0, \chi)(k)  =  I_n^n(0,\chi)(k) \circ {M_n^*}'(0,\chi) \]
 for all $k \in K$.
  Now the desired result follows from the fact that $I^n_n(0,\chi)$ is
  $K$-multiplicity-free, as we noted in Prop. \ref{P:arch1}(i).  
\end{proof}
  \vskip 15pt

 \section{\bf Siegel Eisenstein Series} \label{S:Siegel-Eis}
 In this section, let us return to the setting of the number field
 $F$, and consider the global analog of the previous
 section.  In particular, we consider the global   Siegel principal series
 representation $I^n_n(s ,\chi)$  of $G_n(\A)$.  Here,  we  take
 $\chi=\chi_{V_r}=\chi_{V_0}$ to be a Hecke character fixed as in
 (\ref{SS:weil}), depending on the $\epsilon$-Hermitian space $V_r$. 
In particular, $\chi^c = \chi^{-1}$.  Also we let the set 
\[
X_n(\chi):=\{s\in\C: \text{$s$ appears in the table of Proposition \ref{P:redu} for the fixed $\chi$.}\}
\]
\vskip 5pt

 \subsection{\bf Siegel Eisenstein series.}
 
 For a standard section $\Phi_s \in I^n_n(s,\chi)$, we consider the Siegel
 Eisenstein series defined for ${\rm Re}(s) \gg 0$ by
 \[   
E(s, \Phi)(g) :=  \sum_{\gamma \in Q(Y_n)(F) \backslash
   G_n(F)} \Phi_s(\gamma g)
\]
for $g\in G_n(\A)$. Sometimes we write
\[
E(s, \Phi)=E^{n,n}(s,\Phi)
\]
when we want to emphasize the rank of the group.
It admits a meromorphic continuation to
$\C$. 
 \vskip 5pt
 
 The following proposition, due to Kudla-Rallis \cite{KR3} and V. Tan
 \cite{T2}, summarizes the analytic properties of the Siegel
 Eisenstein series.
 \vskip 5pt
 
 \begin{Prop} \label{P:analytic}
 Exclude the case when $G_n = \OO_{1,1}$.
 For ${\rm Re}(s) \geq 0$,  $E(s, \Phi)$ is holomorphic except at
 $s=s_{m,n} \in X_n(\chi)$ with $s_{m,n} > 0$. At these points,
 $E(s,\Phi)$ has a pole of order at most $1$. 
  \end{Prop}
 
 In view of the above proposition, we may consider the Laurent
 expansion of $E(s, \Phi)$ at $s = s_{m,n} > 0$:
 \[  E(s, \Phi_s) = \frac{E_{-1, s_{m,n}}(\Phi)}{s-s_{m,n}} +  E_{0, s_{m,n}}(\Phi) + \cdots\] 
 whereas at $s =0$, one has
 \[  E(s, \Phi_s) = E_{0, s_{m,n}}(\Phi) + E_{1, s_{m,n}}(\Phi) \cdot s +
 \cdots. \]
Here we many view each $E_{d, s_{m,n}}$ as a linear map
$I_n^n(s,\chi)\rightarrow\Aut(G_n)$.
Viewing the Laurent coefficients in this way, we note that when
$s_{m,n} > 0$, $E_{-1,s_{m,n}}$ is
 $G_n(\A)$-equivariant, where $E_{0, s_{m,n}}$ is $G_n(\A)$-equivariant
 modulo the image of $E_{-1, s_{m,n}}$.  When $s_{m,n} = 0$, the analogous
 statement holds for the first and second Laurent coefficients $E_{0, s_{m,n}}$
 and $E_{1, s_{m,n}}$.
 \vskip 5pt

 \subsection{\bf Global Siegel-Weil sections}
Analogously to the local case, we may  introduce the notion of a
Siegel-Weil section.
 For an $\epsilon$-Hermitian space $V_r$ of dimension $m=m_0+2r$ over $E$, we have  
 \[  \Phi^{n,r}: \mathcal{S}(Y_n^* \otimes V_r)(\A)\longrightarrow
 I^n_n(\frac{m - d(n)}{2},\chi) \]
 defined by
 \[  \Phi^{n,r}(\phi)(g) = \omega_{n,r}(g)\phi(0). \]
 Its image is isomorphic to the maximal $H(V_r)(\A)$-invariant quotient
 of $\omega_{n,r}$:
 \[  R_n(V_r)   = \otimes_v R_n(V_{r, v}). \]
 \vskip 5pt

 \subsection{\bf  Coherent vs. Incoherent.}
 Now suppose that $s_{m,n} \in X_n(\chi)$ and $s_{m,n} \geq 0$. 
 To describe the structure of $I^n_n(s_{m,n},\chi)$, let us make the
following definition following Kudla-Rallis \cite{KR5}.
\vskip 5pt

\begin{enumerate}[(1)]
\item Assume $\epsilon_0\neq -1$. Let $V_r^c$ be the space
  complementary to $V_r$ with respect to $W_n$ (so $\dim V_r^c=m' \leq m$ is such that
  $m+m'=2d(n)$). Let
  $\mathcal{C}=\{U_v\}$ be a collection of $\epsilon$-Hermitian spaces
  $U_v$ over $F_v$ such that
\begin{itemize}
\item $\dim U_v=m'$ for all $v$;
 \item $U_v= V_{r,v}^c$ for almost all $v$;
\item if $\epsilon_0 =1$, then $\chi_{U_v}=\chi_v$ for all $v$.
\end{itemize}
Then we call the collection $\mathcal{C}$ {\em coherent} if there is
a global $\epsilon$-Hermitian space $V$ over $F$ such that $V\otimes
F_v=U_v$. Otherwise we call the collection $\mathcal{C}$ {\em incoherent}.

\vskip 10pt

\item Assume $\epsilon_0=-1$, so $G_n=\OO_{n,n}$. Let $\mathcal{C}=\{\eta_v\}$ be a
  collection of  characters of $G_n$ with $\eta_v =  \mathbf{1}$ or $\det$ and such that
  $\eta_v=\mathbf{1}$ for almost all $v$. Then we call the collection
  $\mathcal{C}$ {\em coherent} if $\eta_v=\det$ for an even number of
  $v$, or equivalently $\otimes_v\eta_v$ is an automorphic determinant
  character. Otherwise we call $\mathcal{C}$ {\em incoherent}.
\end{enumerate}

 \vskip 5pt

For each $\mathcal{C}=\{U_v\}$ or $\{\eta_v\}$, we define
\[
R_n(\mathcal{C})=\begin{cases}\otimes'_v R_n(U_v),&\text{if
    $\epsilon_0\neq -1$}\\
\otimes'_v R_n(V^c_{r,v})\otimes\eta_v,&\text{if $\epsilon_0=-1$}.
\end{cases}
\]
By Proposition \ref{P:str} and \ref{P:arch1}, we see that  the maximal semisimple
 quotient of $I^n_n(s_{m,n},\chi)$ is given by
 \[  \bigoplus_{ \mathcal{C}  }   R_n(\mathcal{C}) \]
 where the sum runs over all the collections $\mathcal{C}$ (coherent or
 not) as defined above.
 \vskip 5pt

 \vskip 5pt
 
 As a special case, when $s_{m,n} = 0$, so that $ m = m' = d(n)$, the
 degenerate principal series $I(0,\chi)$ is semisimple and one has:
 \[  
I_n^n(0,\chi) = \left( \bigoplus_{\text{coherent
      $\mathcal{C}$}}R_n(\mathcal{C}) \right) \oplus \left(
  \bigoplus_{\text{incoherent $\mathcal{C}$}}R_n(\mathcal{C})\right).
\]
We call the first summand {\em the coherent submodule} and the second
{\em the incoherent submodule}.

\vskip 5pt

\subsection{\bf The leading term of $E(s,\Phi_s)$.}
 The following proposition describes the image of the leading term
 $E_{-1, s_{m,n}}$ or $E_{0, s_{m,n}}$ (see \cite{KR3,KR5, GT, T2, Y4}). 
 \vskip 5pt

 \begin{Prop}  \label{P:E-1}
\begin{enumerate}[(i)]
 \item When $s_{m,n} > 0$, 
 \[  \im E_{-1, s_{m,n}} \cong  \bigoplus_\mathcal{C}   R_n(\mathcal{C}), \]
 where $\mathcal{C}$ runs over coherent collections.
 \vskip 5pt
 
 \item When $s_{m,n} = 0$, 
 \[   \im E_{0, s_{m,n}} \cong  \bigoplus_\mathcal{C}  R_n(\mathcal{C}), \]
where $\mathcal{C}$ is as in (i).
\end{enumerate}
\end{Prop}
\vskip 5pt
  
 \subsection{\bf Intertwining operator.}
 One has the standard global intertwining operator 
 \[ M_n(s,\chi):  I^n_n(s,\chi) \longrightarrow I^n_n(-s, \chi), \]
which satisfies 
 \[ M_n(-s,\chi)\circ M_n(s,\chi) = 1.\]
 The normalization for the local intertwining operators is such that one has:
 \[  M_n(s, \chi) = \otimes_v M_{n,v}^*(s, \chi_v). \]
 Moreover, one has the functional equation
 \[  E(s, \Phi) =  E(-s,  M_n(s, \chi) \Phi). \]
\vskip 5pt

 We collect some of the important properties of the intertwining
 operator $M_n(s,\chi)$ that we need.
 \begin{Lem} \label{L:Mat02}
\begin{enumerate}[(i)]
\item $M_n(s,\chi)$ is holomorphic at $s =0$ and $M_n(0,\chi)^2  =1$.
\item $M_n(0,\chi)$ acts as $+1$ on the coherent submodule of
  $I_n^n(0,\chi)$ and $-1$ on the incoherent submodule.
\item The derivative $M_n'(0,\chi)$ commutes with $M_n(0,\chi)$ and
  preserves each irreducible submodule $R_n(\mathcal{C}) \subset
  I_n^n(0,\chi)$.
\end{enumerate}
\end{Lem}
\begin{proof}
(i) This is shown in the same way as  Lemma \ref{L:Mat0}(i).  
\vskip 5pt

(ii) The case for $G_n=\OO_{n,n}$ is \cite[Lemma 7.4 (i)]{GT}; so we assume $\epsilon_0 \ne -1$ here. First note that $I_n^n(0,\chi)$ is multiplicity free and
  semisimple, and hence part (i) of the lemma implies that $M_n(0,\chi)$
  acts on each irreducible summand as $\eta=\pm1$. Suppose
  $\Phi\in I_n^s(0,\chi)$ is in an irreducible subspace of the
  coherent submodule. By the functional equation of the Eisenstein
  series, we have $E(0, \Phi)=E(0, \eta\Phi)$. But Proposition
  \ref{P:E-1} (ii) implies that $E(0, \Phi)\neq 0$, which implies
  $\eta=1$. Next suppose $\Phi$ is in an irreducible subspace of
  the incoherent submodule. We may assume that $\Phi$ is factorizable as
  $\Phi=\otimes_v\Phi_v$, with $\Phi_v\in R_n(U_v)$ and
     \[ M_{n,v}^*(0,\chi_v) \Phi_{v} = \eta_{v} \Phi_{v} \]
  for some signs $\eta_v$.
  Now pick a  finite place $v_0$ inert in $E$, and define
  \[  \mathcal{C}=\{U_{v_0}'\}\cup \{U_v: v \ne v_0 \},\]
   where
  $U_{v_0}'$ is as defined as in the beginning of \S \ref{SS:str}, so that $U'_{v_0} \ne U_{v_0}$. 
  Then $\mathcal{C}$ is a coherent collection. If we pick a non-zero
  $\Phi'_{v_0}\in R_n(U_{v_0}')$, then Lemma \ref{L:Mat0} implies that
  \[ M_{n,v_0}^*(0,\chi_{v_0}) \Phi_{v_0} = -\eta_{v_0} \Phi_{v_0}. \]
  Now let
  \[ \Phi'=\Phi'_{v_0} \otimes (\otimes_{v\neq v_0}\Phi_v). \] 
  This is in the coherent submodule, and hence 
  \[  M_n(0,\chi) \Phi' = \Phi'. \] 
  But we also have
  \[ M_n(0,\chi) \Phi' = \left( -\eta_{v_0}\prod_{v\neq v_0}\eta_v \right) \cdot \Phi' = -\eta \cdot \Phi'. \] 
Hence $\eta=-1$ as desired. It is worth noting that the proof of (ii)
is somewhat indirect as its local counterpart Lemma \ref{L:Mat0}(ii)
has only been shown in the non-archimedean case. 
\vskip 5pt

(iii) The first statement is proved in the same way as the analogous
statement in Lemma \ref{L:Mat0}(iii).  The second part also follows
from the local analogues (Lemma \ref{L:Mat0}(iii)) by arguing
semi-locally. Namely, we may assume that each section $\Phi\in
R_n(\mathcal{C})$ can be factorized as $\Phi=\otimes_{v \in
  S}\Phi_v\otimes\Phi_S$, where $\Phi_S=\otimes_{v\notin S}\Phi_v$ and
$S$ is a finite set of places such that $\Phi_v$ is spherical for
$v\notin S$. Accordingly, we may decompose the intertwining operator as
$M_n(0,\chi)=\otimes_{v\in S}M^*_{n,v}(0,\chi_v)\otimes M^*_{n,S}(0,
\chi_S)$, where $M^*_{n,S}(0,\chi_S)$ is defined analogously to
$\Phi_S$. By Leibnitz rule, the derivative $M_n'(0,\chi)$ is computed as
\begin{align*}
M_n'(0,\chi) &=\left (\sum_{v_0\in S} {M^*}'_{n,v_0}(0,\chi_{v_0}) \otimes
\left( \underset{\substack{v \ne v_0,\\v\in S}}{\otimes} M^*_{n,v}(0,\chi_v)
\right)\right)\otimes M^*_{n,S}(0,\chi_S)\\
&+\underset{v\in S}{\otimes}M^*_{n,v}(0,\chi_v)\otimes {M^*}'_{n,S}(0,
\chi_S).
\end{align*}
For the unramified part, one sees that
$M^*_{n,S}(0,\chi_S)\Phi_S=c_S\Phi_S$ and
${M^*}'_{n,S}(0,\chi_S)\Phi_S=d_S\Phi_S$ for some scalars $c_S$ and
$d_S$. For each $v\in S$, $M^*_{n,v}(0,\chi_v)$ acts as $\pm 1$ on the
local component of $R_n(\mathcal{C})$, and the derivative
${M^*}'_{n,v}(0,\chi_v)$ preserves the component by Lemma \ref{L:Mat0}
(iii). Hence if one applies $\Phi$ to the above expression of the
derivative $M_n'(0,\chi)$, one sees that $M_n'(0,\chi)\Phi\in
R_n(\mathcal{C})$.
\end{proof}

\vskip 5pt

\subsection{\bf The Laurent coefficients $A^{n,r}_d$.}
Now we consider the Siegel-Weil sections arising from $V_r$ (with
$\dim_E V_r = m$). At the point $s_{m,n} \in X_n(\chi)$ with $s_{m,n} \geq 0$, we have the Laurent
expansion
\[  
E(s, \Phi^{n,r}(\phi)) = 
\frac{A^{n,r}_{-1}(\phi)}{s-s_{m,n}} +  A^{n,r}_0(\phi) + \cdots
\quad \text{if $s_{m,n} > 0$;} 
\] 
or
\[  
E(s, \Phi^{n,r}(\phi)) = 
A^{n,r}_0(\phi) +  A^{n,r}_1(\phi)\cdot s  + \cdots  \quad \text{if
  $s_{m,n} = 0$.} 
\] 
Thus $A_d^{n,r}$ is viewed as a linear map:
\[  A_d^{n,r}  = E_{d, s_{m,n}} \circ \Phi^{n,r}:  \mathcal{S}(Y_n^*
\otimes V_r)(\A) \longrightarrow \mathcal{A}(G_n). \]
\vskip 5pt

Let us close this section with the following proposition, which we do
not need for this paper, but is interesting to take note of nevertheless.
\begin{Prop}  \label{P:A0}
If $A^{n,r}_d$ denotes the leading Laurent coefficient above (so $d=0$
or $-1$), then $A^{n,r}_d$ is  $G_n(\A)$-equivariant. If
$A_{d+1}^{n,r}$ is the second Laurent coefficient, then
\[  A^{n,r}_{d+1}: \mathcal{S}(Y_n^* \otimes V_r)(\A) \longrightarrow
\mathcal{A}(G_n)/ \im A^{n,r}_d \]
is $G_n(\A)$-intertwining.
\end{Prop}

\begin{proof}
To ease the notation we denote $I_n^n(s,\chi)$ simply by $I(s,\chi)$.
Since 
\[   A_j^{n,r}  = E_{j, s_{m,n}} \circ \Phi^{n,r}, \]
we see that $A^{n,r}_d$ is $G_n(\A)$-equivariant but $A_{d+1}^{n,r}$
is $G(\A)$-equivariant modulo the image of $E_{d, s_{m,n}}$. Indeed, for
any $g \in G(\A)$, one has
\[  A_{d+1}^{n,r}( \omega_{n,r}(g) \phi)  
=  g \cdot A_{d+1}^{n,r}(\phi)  + E_{d, s_{m,n}}( I'(s_{m,n}, \chi)(g)
\Phi^{n,r}(\phi) ) \]
where
\[   I'(s_{m,n}, \chi)(g) = \frac{d}{ds} I(s, \chi)(g)  |_{s= s_{m,n}}
\in {\rm End}_K(I(s_{m,n},\chi)). \] 
Here let us note that ${\rm End}_K$ indicates the space of $K$-equivariant operators.
Thus we would like to show that
\[  E_{d, s_{m,n}}( I'(s_{m,n}, \chi)(g) \Phi^{n,r}(\phi) )  \in {\rm Im} A^{n,r}_{d}. \]
Now we need to argue semi-locally as we did in the proof of Lemma
\ref{L:Mat02} (iii). Namely fix $\Phi^{n,r}(\phi)$ so that it is
factored as 
\[  \Phi^{n,r}(\phi)=\otimes_{v\in
  S}\Phi_v^{n,r}(\phi_v)\otimes \Phi_S^{n,r}(\phi_S) \]
  for a finite set $S$ of places such that for each $v\notin S$, $\Phi_v^{n,r}(\phi_v)$ is
spherical. Also since $g\in G_n(\A)$ is fixed, we assume $S$ is large
enough so that $g_v$ belongs to a hyperspecial maximal compact subgroup for $v\notin S$. By
Leibnitz rule, one has
\begin{align*}
&I'(s_{m,n}, \chi)(g)\left( \Phi^{n,r}(\phi)\right)\\ 
=&\left(\sum_{v_0\in S}I_{v_0}'(s_{m,n}, \chi_{v_0})(g_{v_0}) \Phi_{v_0}^{n,r}(\phi_{v_0})
\otimes  \left( \prod_{\substack{v \ne v_0\\v\in S}}  I_v(s_{m,n}, \chi_v)(g_v)
  \Phi_v^{n,r}(\phi_v) \right)\right)\otimes I_{S}(s_{m,n},
\chi_{S})(g_{S}) \Phi_{S}^{n,r}(\phi_{S})\\
&+\left(\prod_{v\in S}I_v(s_{m,n},
  \chi_v)(g_v)\Phi_v^{n,r}(\phi_v)\right)\otimes
I_{S}'(s_{m,n},\chi_{S})(g_{S}) \Phi_{S}^{n,r}(\phi_{S}).
\end{align*}

Note that for any $v\notin S$, $I'_v(s_{m,n},
\chi_v)(g_v)\Phi_v^{n,r}(\phi_v)=0$ because $I_v(s, \chi_v)(g_v)$ is independent of $s$ (since $g_v$ lies in the hyperspecial maximal compact subgroup). Hence in the above expression of the derivative, the second
term is simply zero. Thus the derivative is written as
\[
I'(s_{m,n},\chi)(g)\left( \Phi^{n,r}(\phi)\right)
=\sum_{v_0\in S}I_{v_0}'(s_{m,n}, \chi_{v_0})(g_{v_0}) \Phi_{v_0}^{n,r}(\phi_{v_0})
\otimes  \left( \prod_{v \ne v_0}  I_v(s_{m,n}, \chi_v)(g_v)
  \Phi_v^{n,r}(\phi_v) \right).
\]

Now we claim that for any $v_0$,
 \[   I_{v_0}'(s_{m,n}, \chi_{v_0})(g_{v_0})
 \Phi_{v_0}^{n,r}(\phi_{v_0})  \otimes \left( \prod_{v \ne v_0}
   I_v(s_{m,n}, \chi_v)(g_v) \Phi_v^{n,r}(\phi_v) \right) \]
is either incoherent or lies in ${\rm Im} \Phi^{n,r}$, so that
\[  E_{d, s_{m,n}}\left(I_{v_0}'(s_{m,n}, \chi_{v_0})(g_{v_0})
  \Phi_{v_0}^{n,r}(\phi_{v_0}) \otimes \left( \prod_{v \ne v_0}
    I_v(s_{m,n}, \chi_v)(g_v) \Phi_v^{n,r}(\phi_v) \right)  \right) \in
{\rm Im} A_{d}^{n,r}. \]
This implies the proposition
(on recalling  Proposition \ref{P:E-1}).
\vskip 5pt

For each $v$, $I_v(s_{m,n}, \chi_v)(g_v)
\Phi_v^{n,r}(\phi_v)\in\im\Phi_v^{n,r}$ because $\Phi_v^{n,r}$ is
$G_v(F_v)$-equivariant. Hence it remains to examine the derivative term at $v_0$.
 In what follows,
the subscript $v_0$ is omitted because everything is over $F_{v_0}$.
\vskip 5pt
\noindent (a) When $v_0$ is non-archimedean, the claim is obvious,
since  $ I'(s_{m,n}, \chi)(g)
\Phi^{n,r}(\phi)$  either  lies in ${\rm Im}
\Phi^{n,r}$ or else projects to $R_n(U)$ where $U$ is of
complementary dimension to $V_r$ but belongs to a different Witt
tower. 

\vskip 5pt

\noindent (b) When $v_0$ is complex, there is nothing to show since
$I(s_{m,n}, \chi)  = R_n(V)$ by \cite{LZ3}.
\vskip 5pt

\noindent (c)  When    $v_0$ is real and $\epsilon_0 = -1$, the same
argument as in the non-archimedean case works, since the structure of
$I(s_{m,n},\chi)$ here is similar to that in the non-archimedean
case. However, when $\epsilon_ 0 \ne -1$,  
 the claim is not so obvious, since the module structure of
 $I^n_n(s_{m,n},\chi)$ is more complicated. Since we do not need this
 proposition in the rest of the paper, we shall only give a sketch of
 the argument in these cases.  
\vskip 5pt

Recall first that $I^n_n(s_{m,n} ,\chi)$ is $K$-multiplicity-free
(Prop. \ref{P:arch1}(i)) and the operator $I(s_{m,n}, \chi)(k)$ is
independent of $s$, so that $I'(s_{m,n}, \chi)(k) \Phi^{n,r}(\phi) \in
{\rm Im} \Phi^{n,r}$. Thus the main point is to verify the claim for
elements  $X \in \mathfrak{p}$, where $\mathfrak{g} = \mathfrak{k}
\oplus \mathfrak{p}$ is the Cartan decomposition of the Lie algebra
$\mathfrak{g}$ of $G_n$. 
\vskip 5pt

In \cite{L1, L2}, the module structure of $I(s,\chi)$ is completely
determined by studying how the K-types in $I(s,\chi)$ are moved around
by the operators  $I(s,\chi)(X)$ with $X \in \mathfrak{p}$. In
particular, the transition coefficients from one $K$-type to another
are precisely determined and turn out to be linear functions of
$s$. Reducibility occurs at those $s$ when some of the transition
coefficients vanish.  Since the derivatives of these transition
coefficients are nonzero constants, 
we deduce from these results in \cite{L1, L2} that the derivative
$I'(s_{m,n},\chi)(X)$ moves the $K$-types in $I(s_{m,n},\chi)$ in the
same way as $I(s,\chi)(X)$ does for generic $s$. 
\vskip 5pt

By examining the different cases, one arrives at the following simple
conclusion.   If $\Phi $ belongs to a particular $K$-type $\tau$ of
$I(s_{m,n},\chi)$, then  $I'(s_{m,n},\chi)(X)(\Phi)$ can only have
components belonging to  $K$-types $\tau'$ which are ``adjacent" to
$\tau$. Here, $\tau'$ is adjacent to $\tau$ if its highest weight
differs from $\tau$ in only one coordinate and in that coordinate, the
difference is the minimum allowed in the given case. 
\vskip 5pt

Now consider the natural projection map 
\[ \pi:  I(s_{m,n},\chi) \longrightarrow \oplus_U R_n(U) \]
 where $\dim U$ is complementary to $\dim V_r$.  The above discussion
 implies that if $\pi(\Phi) \in R_n(U_0)$, then
 $\pi(I'(s_{m,n},\chi)(X)(\Phi))$ can only have nonzero components in
 $R_n(U)$ for those $U$ whose signature is adjacent to that of $U_0$
 in the obvious sense.  
 Thus, for example, in the case $\epsilon_0 = 0$, if the signature of
 $U_0$ is $(p,q)$, then the adjacent signatures are $(p,q)$, $(p+1,
 q-1)$ and $(p-1, q+1)$. 
 This proves the desired claim and hence the proposition.
 \end{proof}
 \vskip 15pt

\section{\bf The Siegel-Weil formula}
In this section, we recall the known cases of the Siegel-Weil formula.  
\vskip 5pt

\subsection{\bf Anisotropic case.}
We first consider the case $r=0$, i.e. the pair $(W_n, V_0)$, so $H_0=H(V_0)$ is
anisotropic. In this case, the Siegel-Weil formula is due to Weil
\cite{We}, Kudla-Rallis \cite{KR1}, Ichino \cite{I3}  and Yamana \cite{Y2}.

\begin{Thm}  \label{T:anisot}
For $\phi \in \mathcal{S}(Y_n^* \otimes V_0)(\A)$, the Eisenstein
series $E (s,  \Phi^{n,0}(\phi))$ is holomorphic at $s = s_{m,n} = (m-d(n))/2$
and 
\[  E ( s_{m,n},
\Phi^{n,0}(\phi))=  c_{m,n} \cdot \frac{\tau(H_0)}{[E:F]}   \cdot I_{n,0}(\phi)  \]
with 
\[  c_{m,n} =  \begin{cases}
1, \text{  if $s_{m,n} > 0$, } \\
2, \text{   if $s_{m,n} \leq  0$.}\end{cases} \]
Moreover, the term $\tau(H_0)/ [E:F] = 1$, except when $H_0 = \OO_1$ in which case it is equal to $1/2$. 
\end{Thm}
\vskip 5pt
More precisely, Weil \cite{We} established the case when $m > 2 \cdot
d(n)$, Kudla-Rallis \cite{KR1} established the case when $\epsilon_0 =
1$, Ichino \cite{I3} showed the case when $\epsilon_0=0$ and $d(n) < m
\leq 2 \cdot d(n)$, and Yamana \cite{Y2} completed the case
$\epsilon_0 = 0$ and $m \leq d(n)$. There is nothing to check when
$\epsilon_ 0 = -1$. 
\vskip 5pt

For later purposes, it is necessary for us to compute the constant term of $I_{n,0}(\phi)$ 
 along the maximal parabolic subgroup $Q_1= Q(Y_1)$. We have:
 \vskip 5pt
 
 \begin{Prop} \label{P:aniso}
 For 
 \[ \phi \in  \mathcal{S}(Y_n^* \otimes V_0)(\A) 
=   \mathcal{S}(y_1^* \otimes V_0)(\A) \otimes  
\mathcal{S}({Y'}_{n-1}^* \otimes V_0)(\A) \]
 we have
 \[  I_{n,0}(\phi)_{U_1}|_{\GL(Y_1)(\A) \times G_{n-1}(\A)} 
= \chi \cdot |-|^{\frac{1}{2} m_0}  \boxtimes
I_{n-1,0}(\phi(0,-)). \] 
  \end{Prop}
 \vskip 10pt
 
 \begin{proof}
 If $V_0 = 0$, there is nothing to prove. Thus we assume that $V_0 \ne
 0$; in particular, $V_0$ is not symplectic and so $G_n$ is not an
 orthogonal group.
 This implies that the unipotent radical $U_1$ of $Q_1$ sits in a
 short exact sequence
 \[  \begin{CD} 
 1 @>>> Z_1 @>>> U_1 @>>> U_1/Z_1 @>>> 1 \end{CD} \]
 with
 \[  Z_1 \cong \{ \text{$\epsilon$-Hermitian forms on $Y_1^*$} \} \cong \mathbb{G}_a. \]
Now  we compute the constant term. For $g = (t, g_0) \in \GL(Y_1)(\A)
\times G_{n-1}(\A)$, we have:
 \begin{align*} 
 I_{n,0}(\phi)_{U_1}(g) 
 = & \tau(H_0)^{-1} \cdot \int_{[U_1]} \, \int_{[H_0]} \theta_{n,0}(\phi)(ug, h) \, dh \, du  \\
 = & \tau(H_0)^{-1} \cdot \int_{[H_0]} \int_{[U_1]} \left( \sum_{v_0 \in (y_1^* \otimes V_0)(F)}
   \sum_{\alpha \in ({Y'}_{n-1}^* \otimes V_0)(F)}   \omega_{n,0}(ug, h)  \phi
   (v_0, \alpha)  \right) \, du \, dh.
 \end{align*}
 We may break the sum over $v_0$ into two parts, corresponding to $v_0
 = 0$ and $v_0 \ne 0$. 
 For $v_0 = 0$, the term in the parenthesis is:
 \[    \sum_{\alpha \in ({Y'}_{n-1}^* \otimes V_0)(F)} \omega_{n,0}(g,h)\phi(0,
 \alpha) = \chi(t) \cdot |t|^{\frac{1}{2} m_0} \cdot
 \theta_{n-1, 0}(\phi(0,-))(g_0, h). \]
Since this  is independent of $u$, the integral over $[U_1]$
disappears and the outer integral over $[H_0]$ gives
$I_{n-1,0}(\phi(0,-))$: this is the desired RHS in the identity of the
proposition.
 \vskip 5pt
 
 To prove the proposition, it remains to show that the contribution
 from the term $v_0 \ne 0$ vanishes. For this, we shall show that the
 integral
 \[  \int_{[Z_1]}  \sum_{v_0 \ne 0}  \sum_{\alpha  \in ({Y'}^*_{n-1}
   \otimes V_0)(F)}   \omega_{n,0}(zg, h)  \phi (v_0, \alpha)   \, dz  =
 0. \] 
 Indeed, 
 \[  \omega_{n,0}(zg, h)\phi(v_0, \alpha) 
= \psi(\frac{1}{2} \cdot z \cdot (v_0,v_0)) \cdot \phi(v_0, \alpha). \]
 Since $V_0$ is anisotropic and $v_0 \ne 0$, we see that $(v_0,v_0) \ne 0$ so that 
 \[  z \mapsto \psi(\frac{1}{2} \cdot z \cdot (v_0,v_0)) \]
 is a nontrivial character on $[Z_1]$. This implies that the above
 integral over $[Z_1]$ vanishes, as desired.
 \end{proof}
 
 \vskip 5pt
  
\subsection{\bf First term identity.}
 In this section, we recall the regularized Siegel-Weil formula in the
 first term range.  The following theorem was shown by Kudla-Rallis \cite{KR5}, Moeglin \cite{Mo}, Ichino
\cite[Theorem 3.1]{I1} and \cite[Theorems 4.1 and 4.2] {I2}), Jiang-Soudry \cite[Theorem 2.4]{JS}  and Yamana \cite[Proposition 5.8]{Y2}: 
\begin{Thm} \label{T:1st}
(i) Suppose that $m+ m' = 2 \cdot d(n)$ with $m > m'>0$ and
$m'=m_0+2r'$ (here we allow $r' = 0$).
 Then for $\phi \in \mathcal{S}(Y_n^* \otimes V_r)(\A)$, we have:
\[  
 A^{n,r}_{-1}(\phi) = \kappa_{r,r'}  \cdot
 B^{n,r'}_{\ast}(\Ik^{n,r}(\pi_{K_{H_r}}\phi)),  
\]
where $\kappa_{r,r'}$ is an explicit nonzero constant defined in
(\ref{E:iwasawa}), and
\[  \ast = \begin{cases}
0, \text{  if $H_{r'} = \OO_{1,1}$ or $r'=0$;} \\
-1 \text{  otherwise.} \end{cases} \]
\vskip 5pt

(ii) In the boundary case, where $m = m' = d(n)$ (and $r > 0$),  one
has:
\[ A^{n,r}_0(\phi)  = 2 \cdot B^{n,r}_{-1}(\phi) . \]
If $H_r = \OO_{1,1}$ (so that $G_n = \Sp_2$), however, both sides of the above identity are $0$, in which case 
one has
\[ A^{n,r}_1(\phi)  = 2 \cdot \frac{\tau(H_r)}{[E:F]} \cdot B^{n,r}_{0}(\phi) = B^{n,r}_0(\phi). \]
\end{Thm}
\vskip 5pt

It should be mentioned that Theorem \ref{T:1st}(i)  may not seem so natural, since the pair $(n,r')$ is in the first
term range, but we are considering $\phi \in \mathcal{S}(Y_n^* \otimes
V_r)(\A)$ which is the Weil representation in the second term
range. When $r' = 0$ in Theorem \ref{T:1st}(i), Theorem \ref{T:anisot} gives  an alternative,  more natural statement. When $r' >0$, such a more natural statement, similar to Theorem \ref{T:anisot} and Theorem \ref{T:1st}(ii),  has been given by Yamana in \cite[Theorem 2.2]{Y2}:
\vskip 5pt

\begin{Thm}  \label{T:yamana}
Suppose that $0<m'=m_0+2r' <
d(n)$ with $r' >0$.  Then for $\phi \in  \mathcal{S}(Y_n^*
\otimes V_{r'})(\A)$, $E(s, \Phi^{n,r'}(\phi))$ is holomorphic at $s =
s_{m',n} < 0$,  and
\[   A_0^{n, r'}(\phi)  =    2\cdot B^{n,r'}_{-1}(\phi). \]
 If $H_{r'} = \OO_{1,1}$, however, both sides of the identity are zero, in which case we have:
 \[   A_1^{n, r'}(\phi)  =  2 \cdot \frac{\tau(H_{r'})}{[E:F]}\cdot B^{n,r'}_0(\phi) = B^{n,r'}_0(\phi). \]
  \end{Thm}
 \vskip 5pt
 In particular, Theorem \ref{T:anisot} and Theorem \ref{T:yamana}
 complete  the theory of the (regularized) Siegel-Weil formula when
 $m \leq d(n)$.
 
\vskip 15pt

\section{\bf The  Second Term Identity}
Now we come to the heart of this paper: the derivation of the first
and second term identities in the second term range. The goal is to
relate the automorphic forms $B^{n,r}_{-1}(\phi)$ and
$B^{n,r}_{-2}(\phi)$  with $A^{n,r}_{0}(\phi)$ and
$A^{n,r}_{-1}(\phi)$.
 
 \vskip 10pt
 \subsection{\bf Main result.}
 The following is the main result of this paper.
 \vskip 5pt
 
 \begin{Thm} \label{T:main}
Assume the pair $(n, r)$ is in the second term range.
\begin{enumerate}[(i)]
 \item (First term identity) For $\phi \in  \mathcal{S}(Y_n^* \otimes
   V_r)(\A)$, one has
 \[   A^{n,r}_{-1}(\phi) =  B^{n,r}_{-2}(\phi). \]
 \item (Second term identity)  For $\phi \in  \mathcal{S}(Y_n^* \otimes V_r)(\A)$, one has
  \[   A^{n,r}_0(\phi)  =   B^{n,r}_{-1}(\phi)  -
  \kappa_{r, r'} \cdot
\{  B^{n,r'}_0(\Ik^{n,r}(\pi_{K_{H_r}}\phi)) \}  \, \mod{\im
  A^{n,r}_{-1}}, \]
 where $\kappa_{r,r'}$ is a nonzero explicit constant defined in
 (\ref{E:iwasawa}) and $V_{r'}$ is the complementary space of $V_r$
 (relative to $W_n$), so that $m_0+r+r' = d(n)$. Moreover, the term in $\{ ....\}$ on the RHS is interpreted as $0$ if $r' = 0$ or $H_{r'} = \OO_{1,1}$.  
\end{enumerate}
 \end{Thm}
 
 \vskip 10pt
   
 \subsection{\bf Strategy.} 
 Before plunging into the proof of the theorem, let us give a sketch of the main idea.
 The strategy of the proof has already been used in \cite{GT} to prove
 a weak form of the theorem in the case when $\epsilon_0 = -1$.  It is
 based on induction on the quantity 
 \[   \mathcal{N} =  m - d(n). \]
 Note that
 \begin{align*}
   &\mathcal{N} < 0 \Longleftrightarrow \text{first term range};\\
   &\mathcal{N} = 0 \Longleftrightarrow \text{boundary case}; \\ 
  &\mathcal{N} > 0 \Longleftrightarrow \text{second term range}.
 \end{align*}
 Let us illustrate how one starts the induction argument by going from
 the theorem in the boundary case ($\mathcal{N} = 0$)  to the first
 case in the second term range ($\mathcal{N} = 1$).
Thus suppose we are dealing with the Weil representation of
$G(W_{n+1})\times H(V_r)$ with $m = d(n+1)$.  Then we are in the
boundary case, and for $\phi \in \mathcal{S}(Y_{n+1}^* \otimes
V_r)(\A)$ we have the first term identity supplied by Theorem \ref{T:1st}(ii):
\[  A^{n+1,r}_{0}(\phi) = 2  \cdot B_{-1}^{n+1, r}(\phi).  \]
 We may take the constant term of both sides with respect to the maximal parabolic 
 $Q^{n+1}(Y_1) = L^{n+1}(Y_1) \cdot U^{n+1}(Y_1)$ of $G_{n+1}$, which gives
 \[  A^{n+1,r}_0(\phi)_{U^{n+1}(Y_1)}  = 2 \cdot
 B^{n+1,r}_{-1}(\phi)_{U^{n+1}(Y_1)}, \]
 which is an identity of automorphic forms on $L(Y_1) = \GL(Y_1)
 \times G(W_{n})$, where $W_{n} = Y'_n \oplus {Y'_n}^*$. (Note that the
 superscript $^{n+1}$ in the groups $Q^{n+1}(Y_1)$ etc indicates the
 rank of the ambient group $G_{n+1}$.)
 \vskip 5pt
 
 Now note that both sides of this last equation are constant terms of
 the Laurent coefficients of some Eisenstein series. Thus we may
 compute them by first computing the constant term of the relevant
 Eisenstein series, followed by extracting the relevant Laurent
 coefficients.    The constant terms along  $U^{n+1}(Y_1)$ of
 Eisenstein series on $G_{n+1}$ are simply the sum of Eisenstein
 series on $L(Y_1)$, as given in Lemma \ref{L:constant} below.  We can
 compare terms on both sides with the same $\GL(Y_1)$ part, and thus
 obtain an identity of Eisenstein series on $G(W_n)$. It remains to
 identify these Eisenstein series  and the sections where they are
 evaluated with the desired objects coming from the Weil
 representation of $G_n\times H_r$. The theorem will then follow by
 extracting the relevant Laurent coefficients. 
 
 \vskip 5pt
 
 \subsection{\bf Constant term of Eisenstein series.}
 We now describe the constant term of the relevant  Eisenstein series on 
 $G_n$ along the unipotent radical of the parabolic whose Levi part is
 isomorphic to $GL_1\times G_{n-1}$. To be specific, we shall express
 elements of $G_n$ as matrices using the ordered basis $\{
 y_1,...,y_r, y_r^*, ...., y_1^*\}$ of $W_n$. We let
\[
Y_1=\langle y_1\rangle
\]
and 
\[
Q(Y_1)=L(Y_1)\cdot U(Y_1)
\]
 the parabolic subgroup that fixes $Y_1$ with Levi part
\[
L(Y_1)=GL(Y_1)\times G(W_{n-1})\cong\GL_1 \times G_{n-1}.
\]
While the computation of the constant terms of Eisenstein series is
standard, it is necessary to introduce some notation to state the
results. 
\vskip 5pt

 Let $f_s$ be a standard section of
\begin{equation}\label{E:non-Siegel-principal}
I^n_r(s, \chi) : = {\rm Ind}_{Q(Y_r)(\A)}^{G_n(\A)} (\chi\circ \det) \cdot |\det|^s \boxtimes 
 \Theta_{n-r,0}(V_0).
\end{equation}
Since the elements of the inducing data $(\chi \circ \det) \cdot
|\det|^s \boxtimes  \Theta_{n-r,0}(V_0)$ are automorphic forms on
$L(Y_r)$, via evaluation at the identity element of $L(Y_r)$, we may
regard $f_s$ as a $\C$-valued function on $U(Y_r)(\A) L(Y_r)(F)
\backslash G_n(\A)$. 
 \vskip 5pt

 Now let $E^{n,r}(s, f_s)$ be the associated Eisenstein series, i.e.
\[
E^{n,r}(s, f_s)(g):=\sum_{\gamma\in Q(Y_r)(F)\backslash G_n(F)}f_s(\gamma g)
\]
for $g\in G_n(\A)$ and $\operatorname{Re}(s) \gg 0$. Assume that $1 \leq r < n$. Then in computing the constant term
 $ E^{n,r}(s,f)_{U(Y_1)}$, we first need to enumerate the double coset
 space $Q(Y_r) \backslash G_n / Q(Y_1)$. In this case, one knows that
 the double coset representatives can be chosen to be the elements
 $1$, $w^+$ and $w^-$,  where
  \[  w^+ = \left( \begin{array}{cc}
J_{r+1}   & 0 \\
0 & J_{r+1}  \end{array}\right)  \]
with
\[  J_{r+1} = \left( \begin{array}{cccc}
0 & 0 & 1 & 0   \\
 0 & I_{r-1} & 0 &  0    \\
 1  & 0 & 0     & 0 \\
 0 & 0 & 0 &  I_{n-r-1} \end{array} \right) \]  
 and
 \[  w^-  =  \left( \begin{array}{ccc}
 0 & 0 & 1 \\
 0 & I_{2n-2} & 0 \\
 -\epsilon &  0 & 0 \end{array} \right). \]
Note that there is a distinguished set of coset representatives, given
by Weyl group elements of minimal length, and the elements $w^+$ and
$w^-$ given above are not equal to these distinguished
elements. However, they are more convenient for our purposes; for
example they have the convenient property that they are essentially their own
inverses. 
\vskip 10pt

Associated to the Weyl group element $w=w^+$ or $w^-$ is the standard
intertwining operator $M(w,s)$:
\[  M(w, s)(f)(g) = \int_{(U(Y_1)(F)  \cap w Q(Y_r)(F) w^{-1})\backslash U(Y_1)(\A)}
f_s(w^{-1} u  g) \, du.\]
To study these intertwining operators in greater depth, let us write
\[  
U(Y_1)  \cap w Q(Y_r) w^{-1}
=  (U(Y_1) \cap wU(Y_r) w^{-1}) \cdot (U(Y_1) \cap w L(Y_r) w^{-1}),  
\]
and note that
\begin{align}\label{E:intertwining}
  M(w, s)(f)(g)
&=  \int_{ (U(Y_1)(\A)  \cap w Q(Y_r)(\A) w^{-1})\backslash
  U(Y_1)(\A)} \, 
\int_{[U(Y_1) \cap w L(Y_r)  w^{-1}]}  f_s( w^{-1} \cdot v\cdot u \cdot g) \, dv \, du \notag \\
&=  \int_{ U(Y_1)^w(\A)} 
\, \left(   \int_{[L(Y_r) \cap w^{-1} U(Y_1)  w]} f_s( v \cdot w^{-1} \cdot u \cdot g) \, dv  \right) \, du 
\end{align}
where we have set
\[  U(Y_1)^w :=  U(Y_1) \cap wQ(Y_r) w^{-1} \backslash U(Y_1). \]
Now we can make (\ref{E:intertwining}) explicit as follows:
\begin{enumerate}[$\bullet$]
\item When $w = w^+$, $L(Y_r)\cap w^{-1} U(Y_1) w$ is the unipotent
  radical of the maximal parabolic $Q^{n-r}_1$ of $G_{n-r} \subset
  L_r$ whose Levi is isomorphic to $\GL_1\times G_{n-r-1}$. (To be
  more precise, $Q^{n-r}_1$ is the parabolic that fixed the space
  $\langle y_{r+1}\rangle$. For the notation $Q^{n-r}_1$, the superscript
  indicates the rank of the ambient group and the subscript indicates
  the rank of the $\GL$-part of the Levi.)
Thus the inner integral is the constant term of $f_s$ along this
unipotent radical.  We denote this constant term by $R^{n-r}_{Q_1}(f_s)$,
namely
\begin{equation}\label{E:constant_R}
R^{n-r}_{Q_1}(f_s)(g)=\int_{[L(Y_r) \cap w^{-1} U(Y_1)  w]} f_s(vg)\,dv
\end{equation}
for $g\in G_n(\A)$. One can see that
\[  R^{n-r}_{Q_1}(f_s) \in {\rm Ind}^{G_n(\A)}_{Q^n_{r,1}(\A)}  
(\chi \circ \det)|\det|^s \boxtimes |-|^{- (n-r - \frac{1}{2} +
  \frac{\epsilon_0}{2})} \otimes R^{n-r}_{Q_1}\Theta_{n-r,0}(V_0) \]
 where $Q^n_{r,1}$ is the non-maximal parabolic of $G_n$ with Levi
 factor $\GL(Y_r) \times GL_1 \times G_{n-r-1}$, and
 $R^{n-r}_{Q_1}\Theta_{n-r,0}(V_0)$ is the space generated by the
 constant terms of the automorphic forms in $\Theta_{n-r,0}(V_0)$
 along the unipotent radical of $Q^{n-r}_1$.  
 By Proposition \ref{P:aniso}, we have
 \[  R_{Q_1}^{n-r} \Theta_{n-r,0}(V_0) = \chi |-|^{\frac{1}{2} m_0}
 \boxtimes \Theta_{n-r-1,0}(V_0), \]
 so that
 \[   R^{n-r}_{Q_1}(f_s) \in {\rm Ind}^{G_n(\A)}_{Q^n_{r,1}(\A)} \left(
   (\chi \circ \det) |\det|^s \boxtimes
  \chi |-|^{\frac{1}{2} m_0 - (n-r - \frac{1}{2} +
    \frac{\epsilon_0}{2})} \boxtimes \Theta_{n-r-1,0}(V_0)\right). \]
Thus
 \[  M(w^+,s)(f)(g) = \int_{U_1^{w^+}(\A)} R^{n-r}_{Q_1}(f_s) (w^+ u
 g) \, du. \]
 On restricting $M(w^+,s)f$ to the subgroup $L(Y_1) = \GL(Y_1) \times
 G_{n-1}$, we have:
 \[  M(w^+,s)(f)|_{L(Y_1)} \in \chi \cdot  | - |^{m/2} \boxtimes I^{n-1}_r(s,\chi).
 \]
 \vskip 10pt

\item When $w = w^-$,  $L(Y_r) \cap w^{-1} U(Y_1) w$ is the unipotent
  radical of a maximal parabolic subgroup of $\GL(Y_r) \subset
  L(Y_r)=GL(Y_r)\times G_{n-r}$. Since the function $f$ transforms
  under $\GL(Y_r)$ by the  character $(\chi \circ \det)\cdot
  |\det|^s$, we see that the integrand of  this inner integral of
  (\ref{E:intertwining}) is independent of $v$ and so the inner
  integral simply disappears and gives
\[  M(w^-, s)(f)(g) = \int_{U_1^{w^-}(\A)} f_s(w^- u g) \, du. \]
One has:
\[  M(w^-, s)(f)|_{L(Y_1)} \in \chi \cdot |- |^{ -s+ n - \frac{r-
    \epsilon_0}{2}}  \boxtimes  I^{n-1}_{r-1}(s-\frac{1}{2},\chi).\]
\end{enumerate}

\vskip 10pt

With the above notation, we can now state the lemma which computes the constant term
$E^{n,r}(s,f)_{U(Y_1)}$ as follows.
 \begin{Lem}  \label{L:constant}
 Let $f_s$ be a standard section of $I^n_r(s,\chi)= {\rm
   Ind}_{Q(Y_r)}^{G_n}(\chi \circ \det)  |\det|^s \boxtimes 
 \Theta_{n-r,0}(V_0)$ and let $E^{n,r}(s,f)$ be the associated Eisenstein series.
 \begin{enumerate}[(i)]
 \item If $1 \leq  r < n$,
\begin{align*}
  E^{n,r}(s,f)_{U(Y_1)}    
&=\chi | - |^{s+ n - \frac{r - \epsilon_0}{2} } \cdot E^{n-1,r-1}(s+\frac{1}{2}, f|_{G_{n-1}}) \\
&+ \chi |-|^{m/2} \cdot E^{n-1,r} ( s,  M(w^+,s)(f)|_{G_{n-1}} )\\
&+ \chi |-|^{-s + n - \frac{r- \epsilon_0}{2}} \cdot E^{n-1,r-1}(s -
\frac{1}{2},  M(w^-,s)(f)|_{G_{n-1}}).
\end{align*}
 
 \item If $r =n$, 
 \begin{align*}
  E^{n,r}(s,f)_{U(Y_1)}& = 
  \chi |-|^{s + \frac{d(n)}{2}} \cdot E^{n-1,n-1}(s+ \frac{1}{2},  f|_{G_{n-1}})\\
  &+\chi |-|^{-s +  \frac{d(n)}{2}} \cdot  E^{n-1,n-1}(s -
  \frac{1}{2}, M(w^-,s)(f)|_{G_{n-1}}).
\end{align*}
\end{enumerate}
\end{Lem}
 \vskip 10pt
 
By applying the same type of computation to the auxiliary Eisenstein
series $E_{H_r}$,
we have
\begin{Lem}\label{L:constant_aux}
Let $f_s$ be  a standard section in
$I_{H_r}(s)=\Ind_{P(X_r)(\A)}^{H_r(\A)}|\det_{X_r}|^s\boxtimes\mathbf{1}_{H_0}$,
and let $E_{H_r}(s, f)$ be associated Eisenstein series. Let
$P(X_1)=M(X_1)\cdot N(X_1)$
be the parabolic subgroup that preserves $X_1=\langle x_1\rangle$, where
$M(X_1)=GL(X_1)\times H_{r-1}$ is the Levi part and $N(X_1)$ is the
unipotent radical. Then 
\begin{align*}
  E_{H_r}(s,f)_{N(X_1)}    
&=\chi | - |^{s+ \frac{m-r +\epsilon_0}{2} } \cdot E_{H_{r-1}}(s+\frac{1}{2}, f|_{H_{r-1}}) \\
&+ \chi |-|^{-s + \frac{m-r+ \epsilon_0}{2}} \cdot
E_{H_{r-1}}(s-\frac{1}{2},  M(w^-,s)(s)(f)|_{H_{r-1}})
\end{align*}
 where 
  \[  w^-  =  \left( \begin{array}{ccc}
 0 & 0 & 1 \\
 0 & I_{m-2} & 0 \\
 \epsilon &  0 & 0 \end{array} \right). \]
 
\end{Lem}
 
 \vskip 5pt
 
 \subsection{\bf Proof of Lemma \ref{L:kappa}(i).} \label{SS:kappa}
 We can now give a proof of Lemma \ref{L:kappa}(i) using Lemma
 \ref{L:constant_aux} and induction on $r$. 
 \vskip 5pt
 
 Let us choose $f_s=f_s^0$ to be the spherical section
 in Lemma \ref{L:constant_aux}, and write $E_{H_r}(s, f^0)=E_{H_r}(s)$. 
 If $H_r = \OO_{1,1}$ (i.e. $m=2$, $r=1$ and $\epsilon_0=1$), it is easy to see directly that 
 $E_{H_1}(s)$ and $c_1(s)$ are entire, with $\kappa_r =2$ and $c_1(s) =1$, so that Lemma \ref{L:kappa}(i) holds in this case.  Thus, we assume that $H_r \ne \OO_{1,1}$ henceforth, so that $E_{H_r}(s)$ has a pole of order $1$ at $s = \rho_{H_r}$.
  \vskip 5pt
 
 Taking
 the constant term $E_{H_r}(s)_{N(X_1)}$ along $N(X_1)$,  
 we observe that on the RHS of the identity in Lemma \ref{L:constant_aux}, the first term is holomorphic at $s = \rho_{H_r}$. Thus, we have
 \[  -1 = {\rm ord}_{s = \rho_{H_r}} E_{H_r}(s)_{N(X_1)} =  {\rm ord}_{s= \rho_{H_r}} E_{H_{r-1}}(s-\frac{1}{2},  M(w^-,s)(s)(f^0)|_{H_{r-1}}). \]
 By the analog of  \cite[Lemma 1.2.2]{KR3}, one has
\[  
M_{H_{r-1}}(s - \frac{1}{2}) \left( (M(w^-, s) f^0)|_{H_{r-1}}\right)   
= M_{H_r}(s) f^0 |_{H_{r-1}}. \]
Thus 
  \begin{align*}
   E_{H_{r-1}}(s - \frac{1}{2}, (M(w^-, s) f^0)|_{H_{r-1}})   
   &= E_{H_{r-1}}(-s + \frac{1}{2}, M_{H_{r-1}}(s - \frac{1}{2})
   \left( ( M(w^-, s) f^0)|_{H_{r-1}}\right)\\
   &=E_{H_{r-1}}(-s+ \frac{1}{2},  M_{H_r}(s) f^0 |_{H_{r-1}} )\\
   &= c_r(s) \cdot E_{H_{r-1}}(-s + \frac{1}{2}, f_{r-1}^0)\\
   &= \frac{c_r(s)}{c_{r-1}(s- 1/2)} \cdot E_{H_{r-1}}(s - \frac{1}{2}),
   \end{align*}
where for the first equation we used the functional equation for
$E_{H_{r-1}}$ and also $f_{r-1}^0$ denotes the spherical section
corresponding to the group $H_{r-1}$. 
\vskip 5pt

When $r = 1$  (recalling that $H_1 \ne \OO_{1,1}$ here), the last
expression is simply $c_1(s)$ and so we see that $c_1(s)$ has a pole
of order $1$ at $s = \rho_{H_1}$ and
\[  \kappa_1 = {\rm Res}_{s = \rho_{H_1}} E_{H_1}(s) = {\rm Res}_{s = \rho_{H_1}}c_1(s), \]
as desired. This completes the verification of the base case of the induction. 
  \vskip 5pt
  
  For $r>1$, we see that 
  \[  {\rm ord}_{s = \rho_{H_r}} E_{H_r}(s)  =  {\rm ord}_{s = \rho_{H_r}}  \frac{c_r(s)}{c_{r-1}(s- 1/2)} \cdot E_{H_{r-1}}(s - \frac{1}{2}). \]
  On noting that $\rho_{H_r} - 1/2 = \rho_{H_{r-1}}$, we conclude  that
 \[  {\rm ord}_{s = \rho_{H_r}} \left( \frac{E_{H_r}(s)}{c_r(s)}  \right)  =  {\rm ord}_{s = \rho_{H_{r-1}}}\left(
  \frac{E_{H_{r-1}}(s)}{c_{r-1}(s)}  \right).  \]
  The result thus follows by induction on $r$. This completes the proof of Lemma
 \ref{L:kappa}(i).
 \vskip 15pt
 
 \section{\bf Proof of  Theorem \ref{T:main}: Base Step}
 
 In this section, we shall begin our proof of Theorem \ref{T:main} by
 establishing the base case, which goes from $\mathcal{N} = 0$ to
 $\mathcal{N} = 1$. 
 
 \subsection{\bf Base step of induction.}  \label{SS:base}
 Choose once and for all $\phi_1\in \mathcal{S}(Y_1^* \otimes
 V_r)(\A)$ satisfying:
 \vskip 5pt
 
 \begin{itemize}
 \item $\phi_1(0) = 1$;
 \item $\phi_1$ is $K_{H_r}$-invariant, so that $\pi_{K_{H_r}} \phi_1 = \phi_1$. 
 \end{itemize}
 \vskip 5pt
 
 Let $Y_n' = \langle y_2, ...,y_{n+1} \rangle$ so that ${Y_n'}^* = \langle y_2^*,...,y_{n+1}^* \rangle$.
 For any $\phi \in \mathcal{S}({Y_n'}^* \otimes V_r)(\A)$, we set
 \[  \tilde{\phi} := \phi_1 \otimes \phi \in \mathcal{S}(Y_{n+1}^*
 \otimes V_r)(\A). \]
 Then 
 \[  \pi_{K_{H_r}}(\tilde{\phi}) = \phi_1 \otimes \pi_{K_{H_r}}
 \phi. \]
Note that the group $G_n$ acts trivially on $\phi_1$, i.e. for
$g\in G_n(\A)$,
$\omega_{n+1,r}(g)\tilde{\phi}=\phi_1\otimes\omega_{n,r}(g)\phi$.
\vskip 5pt

 We consider the boundary case: $(W_{n+1}, V_r)$ with $m =
 d(n+1)$. Then from the first term identity for the boundary case, we
 have
 \[   A^{n+1,r}_{0}(\tilde{\phi}) = 2 \cdot
 B_{-1}^{n+1, r}(\tilde{\phi}).   \]
 We compute the constant term along $U(Y_1)$ to get:
  \begin{equation}  \label{E:compare}
   A^{n+1,r}_0(\tilde{\phi})_{U(Y_1)} = 2 \cdot
   B^{n+1,r}_{-1}(\tilde{\phi})_{U(Y_1)}. 
\end{equation}
 Now we may appeal to Lemma \ref{L:constant} to calculate both
 sides. We explicate the results below.
  \vskip 5pt

\subsection{\bf The A-side.}  Let us start with  the LHS of this
identity (\ref{E:compare}), namely ``the $A$-side''. We have:
  \[  A^{n+1,r}_0(\tilde{\phi})_{U(Y_1)} = {\rm Val}_{s=0}
  E^{n+1,n+1}(s, \Phi^{n+1,r}(\tilde{\phi}))_{U(Y_1)}. \]
 On the other hand, Lemma \ref{L:constant} (ii) says that, as an
 automorphic form on $L(Y_1) = \GL(Y_1) \times G(W_n)$, 
\begin{align} \label{E:A-side}
E^{n+1,n+1}(s,\Phi^{n+1,r}(\tilde{\phi}))_{U(Y_1)}
&=   \chi |-|^{s + \frac{d(n+1)}{2}}  \cdot E^{n,n}(s+\frac{1}{2}, \Phi^{n+1,r}(\tilde{\phi})|_{G_n})\\
 &+\chi  |-|^{-s + \frac{d(n+1)}{2}} \cdot  E^{n,n}(s - \frac{1}{2},
   M(w^-, s)(\Phi^{n+1,r}(\tilde{\phi}))|_{G_n}).\notag
\end{align}
 Observe that at $s=0$, $\GL(Y_1)$ acts by the same character $\chi
 \cdot | -|^{\frac{d(n+1)}{2}}$ for both terms of RHS of
 (\ref{E:A-side}).
 \vskip 5pt
 
 Now note that for $g \in G_n(\A)$,
 \[  \Phi^{n+1,r}(\tilde{\phi})(g) =  \omega_{n+1,r}(g) \tilde{\phi}(0)
 = \phi_1(0) \cdot \omega_{n,r}(g) \phi (0)
 = \Phi^{n,r}(\phi)(g), \]
 so that  the first term on the RHS of (\ref{E:A-side}) is simply 
 \[   E^{n,n}(s+\frac{1}{2}, \Phi^{n,r}(\phi))   \]
as a function on $G_n$. Thus, taking the value of the first term on
the RHS at $s = 0$, we get
 \begin{equation}\label{E:A-side1}
 A^{n,r}_0( \phi).
\end{equation}

  Next let us consider the second term on the RHS of
  (\ref{E:A-side}). By applying the functional equation, we obtain
\[   E^{n,n}\left(s - \frac{1}{2}, M(w^-, s)(\Phi^{n+1,r}(\tilde{\phi}))|_{G_n}\right)
  =  E^{n,n}\left(\frac{1}{2} -s,  M_n(s - \frac{1}{2},\chi)\left(
  M(w^-, s)(\Phi^{n+1,r}(\tilde{\phi}))|_{G_n}\right)\right),\]
where recall $M_n(s,\chi)$ is the intertwining operator for the Siegel
principal series. Moreover, it was shown in \cite[Lemma 1.2.2]{KR3} that 
  \[    M_n(s - \frac{1}{2},\chi)\left( M(w^-,
  s)(\Phi^{n+1,r}(\tilde{\phi}))|_{G_{n}}\right)
= M_{n+1}(s,\chi)(\Phi^{n+1,r}(\tilde{\phi}))|_{G_{n}}. \]
  Thus the second term on the RHS of (\ref{E:A-side}) is equal to:
 \[   E^{n,n}(\frac{1}{2} -s,   M_{n+1}(s)(
 \Phi^{n+1,r}(\tilde{\phi}))|_{G_n}). \]
Moreover, the intertwining operator $M_{n+1}(s)$ is holomorphic at $s=0$ and 
   \[   M_{n+1}(0)( \Phi^{n+1,r}(\tilde{\phi}))  = \Phi^{n+1,r}(\tilde{\phi}) \]
   by Lemma \ref{L:Mat02}(ii).  Thus, on taking the value at $s= 0$,
   the second term of the RHS of (\ref{E:A-side}) gives
   \begin{equation}\label{E:A-side2}
A^{n,r}_0(\phi)  -
   E^{n,n}_{-1, s_{m,n}}(M_{n+1}'(0,\chi)(\Phi^{n+1,r}(\tilde{\phi}))|_{G_n} ),
\end{equation}
where we note that $s_{m,n}  = 1/2$ here.
  By Lemma \ref{L:Mat02}(iii) we know the derivative $M_{n+1}'(0,\chi)$
  preserves each irreducible submodule of $I_{n+1}^{n+1}(0,\chi)$, and hence
  \[  M'_{n+1}(0,\chi)(\Phi^{n+1,r}(\tilde{\phi}))  = \Phi^{n+1,r}(\varphi)
  \quad \text{for some $\varphi \in
  \mathcal{S}(Y_{n+1}^* \otimes V_r)(\A) $,} \]
  so that
  \[  M'_{n+1}(0,\chi)(\Phi^{n+1,r}(\tilde{\phi}))|_{G_n} =
  \Phi^{n,r}(\varphi|_{Y_n^* \otimes V_r}) \in \im \Phi^{n,r}. \]
So (\ref{E:A-side2}) is written as
\[A^{n,r}_0(\phi)  -
   A^{n,r}_{-1}(\varphi|_{Y_n^* \otimes V_r}).
\]
From this together with (\ref{E:A-side1}), we conclude
\begin{equation}\label{E:A-side-fin}
A^{n+1,r}_0(\tilde{\phi})_{U(Y_1)}  = 2 \cdot A^{n,r}_0(\phi)
 \mod{\im (A^{n,r}_{-1})}.
\end{equation}
   \vskip 10pt
   
\subsection{ \bf The B-side.}
 Now we consider the RHS (``the B-side'') of (\ref{E:compare}), which
 is decidedly more complicated. Lemma \ref{L:constant}(i) implies that
 $ B^{n+1,r}_{-1}(\tilde{\phi})_{U(Y_1)}$ is the residue at 
 \[  s = \rho_{H_r}= \frac{m-r-\epsilon_0}{2} =  \frac{d(n+1) - r
   -\epsilon_0}{2} = \frac{n+1-r}{2} \]
 of the function 
 \begin{align}\label{E:B-side}
   & \chi | - |^{s+ n+1 - \frac{r - \epsilon_0}{2} } \cdot
   E^{n,r-1}(s+\frac{1}{2}, f^{n+1,r}(s, \pi_{K_{H_r}}
   \tilde{\phi})|_{G_{n}}) \\
+ &\chi |-|^{m/2} \cdot E^{n,r} ( s,  M(w^+,s) f^{n+1,r}(s,
\pi_{K_{H_r}} \tilde{\phi})|_{G_{n}} ) \notag \\
+ &\chi |-|^{-s + n+1 - \frac{r- \epsilon_0}{2}} \cdot E^{n,r-1}(s -
\frac{1}{2},  M(w^-,s) f^{n+1,r}(s, \pi_{K_{H_r}}
\tilde{\phi})|_{G_{n}}).  \notag
\end{align}

 The first thing we should mention is that at $s = \rho_{H_r}$,
 $\GL(Y_1)$ acts by the character
 $\chi |-|^{\frac{1}{2} m + (n+1-r)}$  in the first term whereas it
 acts by $\chi |-|^{d(n+1)/2} =\chi |-|^{m/2}$ in the other two
 terms. This implies that the first term does not have a residue at $s
 = \rho_{H_r}$ because both terms on the A-side of (\ref{E:compare}) has
 $\GL(Y_1)$ acting via $\chi \cdot |-|^{m/2}$. Alternatively one
 can directly prove that the first term of (\ref{E:B-side}) does not have a
 residue as follows:
 \begin{Lem}
 The Eisenstein series $E^{n,r-1}(s+\frac{1}{2},  f^{n+1,r}(s,
 \pi_{K_{H_r}} \tilde{\phi})|_{G_{n}})$ is holomorphic at $s =
 \rho_{H_r}$. 
 \end{Lem}
 \begin{proof}
 Since $m = d(n+1)$, the complementary space $V_{r'}$ of $V_r$ relative to $W_n$ has $r' = r-1$.  
 By Proposition \ref{P:key2}, we see that
 \[ f^{n+1,r}(s, \pi_{K_{H_r}} \tilde{\phi})|_{G_{n}}  
  = \alpha_r \cdot  Z_1(- s  - \rho_{H_r}, \phi_1) \cdot  f^{n,r-1}(
  s + \frac{1}{2},  \Ik^{n,r}(\pi_{K_{H_r}}\phi)). \] 
   Thus,
  \[  E^{n,r-1}(s+\frac{1}{2},  f^{n+1,r}(s, \pi_{K_{H_r}} \tilde{\phi})|_{G_{n}})
  = \alpha_r \cdot  Z_1( -s  - \rho_{H_r}, \phi_1) \cdot E^{n, r-1}
  (s+\frac{1}{2}, f^{n,r-1}(  s + \frac{1}{2},
  \Ik^{n,r}(\pi_{K_{H_r}} \phi))).  \]
  At $s = \rho_{H_r}$, 
  $Z_1(-s - \rho_{H_r}, \phi_1)$ is holomorphic,   since one is
  considering a Tate zeta integral at $-2 \rho_{H_r}  = -( n+1-r) \leq
  -1$. On the other hand, since $\Ik^{n,r}$ is
  $H_{r-1}(\A)$-equivariant,
  $\Ik^{n,r}(\pi_{K_{H_r}} \phi)$ is $K_{H_{r-1}}$-invariant. Hence,
  \begin{align*}
  &E^{n, r-1} (s+\frac{1}{2}, f^{n,r-1}(  s + \frac{1}{2},  \Ik^{n,r}(\pi_{K_{H_r}}\phi)))\\
  = &B^{n, r-1}(s+\frac{1}{2} ,  \Ik^{n,r}(\pi_{K_{H_r}}\phi))\\ 
  =  &\frac{1}{\tau(H_{r-1}) \cdot \kappa_{r-1}\cdot P_{n,r-1}(s+\frac{1}{2})} \cdot
  \int_{[H_{r-1}]} \theta(\omega_{n,r-1}(z) \Ik^{n,r}(\pi_{K_{H_r}}
  \phi))(g,h) \cdot E_{H_{r-1}}(s+\frac{1}{2},h) \, dh.
  \end{align*} 
 At $s = \rho_{H_r} = \rho_{H_{r-1}} + 1/2$,  one may verify (using formulas for 
 $P_{n,r-1}(s)$ in \cite{KR5, I2, JS}) that
 $P_{n,r-1}(s +1/2)$ and $ E_{H_{r-1}}(s+\frac{1}{2}, h)$ are both
 holomorphic and nonzero. This proves the lemma. 
  \end{proof}

 \vskip 5pt
 Now we examine the remaining two terms corresponding to $w^+$ and $w^-$ in turn.
 \vskip 5pt
 
 \noindent\underline{{\bf The term for $w^+$} :}
 We note the following key proposition:
 \begin{Prop}  \label{P:key1}
 \[
 M(w^+,s) f^{n+1,r}(s, \pi_{K_{H_r}} \tilde{\phi})|_{G_{n}}   =  f^{n,r}(s, \pi_{K_{H_r}} \phi). \]
  \end{Prop}
 \vskip 5pt
 
 \begin{proof}
 By definition, for $g \in G_n$,
\[
 M(w^+,s) f^{n+1,r}(s, \pi_{K_{H_r}} \tilde{\phi}) (g)  
 =  \int_{(U^{n+1}_1)^{w^+}(\A)} R^{n+1-r}_{Q_1}( f^{n+1,r}(s,\pi_{K_H}  \tilde{\phi}))(w^+ ug)  \, du,
 \]
where recall that as defined in (\ref{E:constant_R}),
$R_{Q_1}^{n+1-r}$ indicates the ``constant term'' along the unipotent radical
of the parabolic of $G_{n+1-r}$ whose Levi is $GL_1\times G_{n-r}$,
and 
 \begin{align*}
 &f^{n+1,r}(s,\pi_{K_H}  \tilde{\phi})(g)\\
 =&\int_{\GL(X_r)(\A)} I_{n+1-r, 0}(\omega_{n+1,r}(g)(\mathcal{F}_{n+1,r}(
 \pi_{K_H} \tilde{\phi}) (\beta_0 \circ a) (0_r,-))  \cdot
 |\det(a)|^{s+n+1-\rho_{H_r}} \, da
\end{align*}
 where $0_r$ is the zero element in $Y_r^* \otimes V_0$.  
Thus
\begin{align*}
 &  R^{n+1-r}_{Q_1}( f^{n+1,r}(s,\pi_{K_H}  \tilde{\phi}))(g)\\
=&\int_{\GL(X_r)(\A)} R^{n+1-r}_{Q_1} \left( I_{n+1-r,
     0}(\omega_{n+1,r}(g)(\mathcal{F}_{n+1,r}( \pi_{K_H} \tilde{\phi})
   (\beta_0 \circ a) (0,-))  \right) \cdot
 |\det(a)|^{s+n+1-\rho_{H_r}} \, da.
\end{align*}
Let us set
 \[  \Phi_a := \mathcal{F}_{n+1,r}( \pi_{K_H} \tilde{\phi}) (\beta_0
 \circ a) (0,-) \in \mathcal{S}({Y'_{n+1-r}}^* \otimes V_0)(\A). \]
 By Proposition \ref{P:aniso}, we have an identity of functions on $G_{n-r}$:
 \[   R^{n+1-r}_{Q_1}( I_{n+1-r,0}(\Phi_a))  = I_{n-r, 0}(\Phi_a(0,-)) \]
 where
 \[  \Phi_a(0,-) \in \mathcal{S}({Y'_{n-r}}^* \otimes V_0)(\A) \]
 with $Y_{n-r}^* = \langle y_{r+2}^*,...., y_{n+1}^* \rangle$. 
\vskip 5pt

Hence
 \begin{align*}
 & M(w^+,s) f^{n+1,r}(s, \pi_{K_{H_r}} \tilde{\phi}) (g)\\
=&\int_{(U^{n+1}_1)^{w^+}(\A)}  
  \int_{\GL(X_r)(\A)}    I_{n-r,0}( \omega_{n+1,r}(w^+ u g,a)
  \mathcal{F}_{n+1,r}(\pi_{K_H} \tilde{\phi})(\beta_0 )(0_{r+1}, -) )
  \cdot |\det (a)|^{s-\rho_{H_r}} \,da\, du 
\end{align*}
where $0_{r+1}$ is the zero element in $Y_{r+1}^* \otimes V_0$.
Now we note that a typical element in
$(U^{n+1}_1)^{w^+}$ has the form
\[  \left( \begin{array}{ccc}
1 &  \underline{u} & 0 \\
0 &   I_r & 0 \\
  0 & 0 &   I_{n-r} \end{array} \right)  \in \GL(Y_{r+1}) \subset \GL(Y_{n+1}) \subset G_{n+1} \]
where 
\[  \underline{u} = (u_2, u_3,...,u_r, u_1) \in F^r. \]
Moreover, 
\[  \omega_{n+1,r}(w^+ u g,a) \mathcal{F}_{n+1,r}(\pi_{K_H}
\tilde{\phi})(\beta_0 )(0_{r+1},-)  
=  \omega_{n+1,r}(g,a)\mathcal{F}_{n+1,r}(\pi_{K_H} \tilde{\phi})(
u^{-1} \circ  w^+ \circ  \beta_0 )(0_{r+1},-) \]
and it is easy to see that
\[ u^{-1} \circ  w^+ \circ  \beta_0  = \beta_0' 
 - y_1 \otimes \left( u_1x_1^* + u_2 x_2^*  + u_3 x_3^* +\cdots+ u_r x_r^*\right) \]
 with
 \[  \beta_0' = y_{r+1} \otimes x_1^* + y_2 \otimes x_2^* +\cdots+ y_r \otimes x_r^*. \]
Thus, the integral over $[(U^{n+1}_1)^{w^+}]$ is simply the Fourier transform in the subspace 
$Y_1 \otimes X_r^*$ (followed by evaluation at $0$). In other words, we
are looking at the composite
\vskip 5pt

\[  \begin{CD}
\mathcal{S}(Y_{n+1}^* \otimes V_r) @>\mathcal{F}_{n+1,r}>>
\mathcal{S}(W_{n+1}\otimes X_r^*)\otimes \mathcal{S}(Y_{n+1}^*
\otimes V_0)\\
& &@|  \\
& & \mathcal{S}((Y_1 \oplus Y_1^*) \otimes X_r^*) \otimes  \mathcal{S}(W_n \otimes  X_r^*) \otimes \mathcal{S}(Y_{n+1}^* \otimes V_0) \\
&  &@VV\int_{(U^{n+1}_1)^{w^+}} = \text{Fourier transform in $Y_1 \otimes X_r^*$}V  \\
& &  \mathcal{S}(Y_1^* \otimes (X_r \oplus X_r^*))  \otimes  \mathcal{S}(W_n \otimes  X_r^*) \otimes \mathcal{S}(Y_{n+1}^* \otimes V_0) \\
& & @| \\
&   & \mathcal{S}(Y_1^* \otimes V_r)\otimes \mathcal{S}(W_n \otimes
X_r^*)\otimes \mathcal{S}(Y_n'^* \otimes V_0)\\
& & @VVev_0 \text{ on $Y_1^* \otimes V_r$}V \\
& &  \mathcal{S}(W_n\otimes X_r^*)\otimes \mathcal{S}(Y_n'^*
\otimes V_0). 
 \end{CD} \]
\vskip 5pt

\noindent  But this composite is none other than the map
\[  \text{(evaluation at $0$ on $\mathcal{S}(Y_1^* \otimes V_r)$)}
\boxtimes \mathcal{F}_{n,r}. \]
Thus, for $\tilde{\phi} = \phi_1 \otimes \phi$, we have
\begin{align*} 
& M(w^+,s) f^{n+1,r}(s, \pi_{K_{H_r}} \tilde{\phi}) (g)\\
=&\int_{\GL(X_r)(\A)}  I_{n-r,0}\left(\omega(g,a)
  \mathcal{F}_{n,r}(\pi_{K_H} \phi)( \beta'_0)(0_r,-)\right) \,
|\det(a)|^{s - \rho_{H_r}} \, da\\
=&  f^{n,r}(s, \pi_{K_H} \phi).
\end{align*}
This proves the proposition.
\end{proof}
 
 The proposition implies that 
 \[ E^{n,r} ( s,  M(w^+,s) f^{n+1,r}(s, \pi_{K_{H_r}} \tilde{\phi})|_{G_{n}} ) = B^{n,r}(s, \phi) \]
 so that we have
\begin{equation}\label{E:B-side-fin1}
\left(\text{the residue at $s= \rho_{H_r}$ of the second term of (\ref{E:B-side})}\right)
=B^{n,r}_{-1}(\phi),
\end{equation}
when viewed as an automorphic form on $G_n$. This is one of the
desired terms in the second term identity. We note that the term
$B^{n,r}_{-1}$ is not the leading term of $B^{n,r}(s, \phi)$, which in
fact has a pole of order $2$.  Thus this pole of order $2$ must be
cancelled by a pole of order $2$ from the term associated
to $w^-$ which we will study next. 
\vskip 10pt
 
\noindent\underline{{\bf The term for $w^-$}:}
We now consider the ``$w^-$-term'' of (\ref{E:B-side}), namely the Eisenstein series
\begin{equation} \label{E:Eis}
E^{n,r-1}(s - \frac{1}{2},  M(w^-,s) f^{n+1,r}(s, \pi_{K_{H_r}}
\tilde{\phi})|_{G_{n}}).  
\end{equation}
 We shall apply an argument similar to how we handled the second term on the A-side. 
 Note that we have the standard intertwining operator 
 \[  M_n(w_{r-1}, s):  I^n_{r-1}(s,\chi) \longrightarrow
 I^n_{r-1}(-s,\chi),\]
where $I_{r-1}^n$ is as defined in (\ref{E:non-Siegel-principal}) and
\[
w_{r-1}=
\begin{pmatrix}&&I_{r-1}\\ & I_{2n-2r+2}&\\-\epsilon
  I_{r-1}&&\end{pmatrix}.
\]
 The functional equation for Eisenstein series says that (\ref{E:Eis})  is equal to
 \[  E^{n,r-1}\left(\frac{1}{2}-s, M_n(w_{r-1}, s-
 \frac{1}{2})\left(M(w^-,s) f^{n+1,r}(s, \pi_{K_{H_r}}
   \tilde{\phi})|_{G_{n}}\right)\right).  \]
 Then as in \cite[Lemma 1.2.2]{KR3}, one has 
 \[ M_n(w_{r-1}, s- \frac{1}{2})\left( M(w^-,s)f |_{G_n}\right)
  =  M_{n+1}(w_{r}, s) f|_{G_n} \]
  for $f_s \in I^{n+1}_r(s,\chi)$, where $w_r$ is defined analogously
  to $w_{r-1}$. Thus we are interested in 
  \begin{equation} \label{E:Eis2}
    E^{n,r-1}(\frac{1}{2}-s, M_{n+1}(w_{r}, s)  f^{n+1,r}(s,
    \pi_{K_{H_r}} \tilde{\phi})|_{G_{n}}).  
\end{equation}
  
Now we note:
    
\begin{Lem}  \label{L:comm}
\begin{align*}
& E^{n,r-1}(\frac{1}{2}-s, M_{n+1}(w_r, s)  f^{n+1,r}(s,
\pi_{K_{H_r}} \tilde{\phi})|_{G_{n}})\\
=&c_r(s) \cdot E^{n, r-1}(\frac{1}{2}-s,    f^{n+1,r}(-s,
\pi_{K_{H_r}} \tilde{\phi})|_{G_{n}}).
\end{align*}
 \end{Lem}
 \begin{proof}
 We have the functional equation for the Eisenstein series $E^{n+1,r}$:
 \begin{align*}
&E^{n+1,r}(-s,  M_{n+1}(w_{r}, s)  f^{n+1,r}(s, \pi_{K_{H_r}}
  \tilde{\phi}))\\
=& E^{n+1,r}(s,  f^{n+1,r}(s, \pi_{K_{H_r}} \tilde{\phi}))\\
=& B^{n+1, r}(s, \tilde{\phi}) \\
=& c_r(s) \cdot B^{n+1,r}(-s,\tilde{\phi}),
\end{align*}
where the last equality follows by  (\ref{E:B}). Thus we have:
\[  E^{n+1,r}(-s,  M_{n+1}(w_{r}, s)  f^{n+1,r}(s, \pi_{K_{H_r}}
\tilde{\phi}))  =     c_r(s) \cdot   E^{n+1,r}(-s,   f^{n+1,r}(-s,
\pi_{K_{H_r}} \tilde{\phi})). \]
Now consider the constant term of both sides along the unipotent
radical $U^{n+1}(Y_1)$. By Lemma \ref{L:constant}, we see that each
side is the sum of three automorphic forms on $L^{n+1}(Y_1) = \GL(Y_1)
\times G_n$. Moreover, for generic $s$, these three automorphic forms
can be distinguished from each other by the action of $\GL(Y_1)$.
Thus we may equate the terms on both sides with the same
$\GL(Y_1)$-action. Considering the terms on each side corresponding to
the trivial element of the Weyl group, we obtain
  \[  E^{n,r-1}(\frac{1}{2}-s, M_{n+1}(w_r, s)  f^{n+1,r}(s,
  \pi_{K_{H_r}} \tilde{\phi})|_{G_{n}})
  = c_r(s) \cdot E^{n, r-1}(  \frac{1}{2}-s,    f^{n+1,r}(-s,
  \pi_{K_{H_r}} \tilde{\phi})|_{G_{n}}),  \]
  which is the assertion of the lemma.
\end{proof}

\begin{Rmk} One could directly prove  that 
\[  M_{n+1}(w_r, s)  f^{n+1,r}(s,\pi_{K_{H_r}} \tilde{\phi})=c_r(s)  f^{n+1,r}(-s, \pi_{K_{H_r}} \tilde{\phi}) \] 
as sections of the induced representation $I^{n+1}_r(-s,\chi)$. Recall  from Remark ~\ref{rm_section} at the end of \S  \ref{sectionf} that $f^{n,r}(s,\phi)=F^{n,r}(s,\phi)|_{G_n}$, where   
\[ F^{n,r}(s,-) : \omega_{n,r} \longrightarrow I^{n}_r(s,\chi)\boxtimes I_{H_r}(-s) \]
is a $G_{n}(\A)\times H_r(\A)$-equivariant map defined there. When $n\geq r$, one can establish the identity 
\[  M_{n}(w_r,s)F^{{n},r}(s, - )=M_{H_r}(s)F^{{n},r}(-s,- ), \]
 and when $\phi$ is $K_{H_r}$-invariant, one has:
 \[  M_{H_r}(s)F^{{n},r}(-s,\phi)=c_r(s)F^{n,r}(-s,\phi) \]
 on $G_n(\A)\times K_{H_r}$ . Then one can specialize to the situation of $F^{n+1,r}(s,\pi_{K_{H_r}}\tilde{\phi})$.
\end{Rmk}

\vskip 10pt
  
  By Proposition \ref{P:key2}, we have:
  \[ f^{n+1,r}(-s, \pi_{K_{H_r}} \tilde{\phi})|_{G_{n}}
 = \alpha_r \cdot  Z_1(s - \rho_{H_r}, \phi_1) \cdot  f^{n,r-1}(  -s +
 \frac{1}{2},  \Ik^{n,r}(\pi_{K_{H_r}}\phi)), \] 
  where  we recall that  we are assuming $ m = d(n+1) = n+1
  +\epsilon_0$ so that
  \[  \rho_{H_r} = \frac{n+1-r}{2}. \]
  From the above lemma, we see that (\ref{E:Eis2})  is equal to
  \[    c_r(s) \cdot  \alpha_r \cdot Z_1(s - \rho_{H_r}, \phi_1)  \cdot
  B^{n,r-1}(\frac{1}{2}-s, \Ik^{n,r}(\pi_{K_{H_r}}\phi)),
   \]
  which is in turn equal to
  \[  \frac{c_r(s)}{c_{r-1}(s -\frac{1}{2})} \cdot  \alpha_r \cdot
  Z_1(s- \rho_{H_r}, \phi_1) \cdot B^{n, r-1}(s - \frac{1}{2},
  \Ik^{n,r}(\pi_{K_{H_r}}\phi)). \]
  \vskip 5pt
  
  Now we examine the analytic behavior of this function at 
  $s =\rho_{H_r}$. Though the answer will turn out to be uniform, it is convenient  to consider 3 different cases
  separately:
  \vskip 5pt
  
  \noindent (a)  Assume first that $r \geq 2$ and $H_{r-1} \ne \OO_{1,1}$. Then we have:
  \begin{itemize}
  \item $B^{n,r-1}(s - \frac{1}{2} , \Ik^{n,r}(\pi_{K_{H_r}}\phi))$
    has a pole of order $1$ at $s = \rho_{H_r} = \rho_{H_{r-1}} +
    \frac{1}{2}$, since the case $(W_n, V_{r-1})$ is in the first term
    range. Its residue there is precisely
    $B_{-1}^{n,r-1}(\Ik^{n,r}(\pi_{K_{H_r}}\phi))$. 
    \item   the Tate zeta integral $ Z_1(s -\rho_{H_r}, \phi_1)$ has a
      pole of order $1$, and its residue there is equal to $-
      \phi_1(0) =-1$. 
   \item the function $c_r(s)/c_{r-1}(s - 1/2)$ is holomorphic at $s =
     \rho_{H_r}$ and its value at $s = \rho_{H_r}$ is equal to
     $\kappa_r/ \kappa_{r-1}$ by Lemma \ref{L:kappa}(i).
  \end{itemize}
  \vskip 5pt
  
  Thus (\ref{E:Eis}) has a pole of order $2$ at $s = \rho_{H_r}$ and the leading term of (\ref{E:Eis}) there is equal to
 \[   - \frac{\kappa_r}{\kappa_{r-1}} \cdot \alpha_r \cdot
 B_{-1}^{n,r-1}(\Ik^{n,r}(\pi_{K_{H_r}}\phi))   \]
 while its residue at $s = \rho_{H_r}$ (which is the second Laurent
 coefficient) is
 \[   - \frac{\kappa_r}{\kappa_{r-1}} \cdot \alpha_r \cdot
 B_0^{n,r-1}(\Ik^{n,r}(\pi_{K_{H_r}}\phi)) 
 + C_r  \cdot B_{-1} ^{n,r-1}(\Ik^{n,r}(\pi_{K_{H_r}}\phi))\]
for some constant $C_r$.  Moreover, by Theorem \ref{T:1st} (the first
term identity in the first term range), we see that
\[  B_{-1}^{n,r-1}(\Ik^{n,r}(\pi_{K_{H_r}} \phi))=
\kappa_{r,r-1}^{-1} \cdot A_{-1}^{n,r}(\phi). \]

To sum up, we have obtained
\begin{align}\label{E:B-side-fin2}
&\left(\text{Residue at $s=\rho_{H_r}$ of the third term of
  (\ref{E:B-side})}\right)\\
=&- \frac{\kappa_r}{\kappa_{r-1}} \cdot \alpha_r \cdot
 B_0^{n,r-1}(\Ik^{n,r}(\pi_{K_{H_r}}\phi)) \mod{\im A_{-1}^{n,r}}\notag
\end{align}
when viewed as an automorphic form on $G_n$.
\vskip 15pt

\noindent (b) When $r=2$ and $H_{r-1}  =\OO_{1,1}$, we have:
\vskip 5pt
\begin{itemize}
\item the Tate zeta integral has a simple pole with residue $-1$  as above;
\item the function $c_2(s)/c_1(s)$ has a simple pole at $s = \rho_{H_2}$ with residue $\kappa_2/\kappa_1 = \kappa_2$;
\item the term $B^{n, 1}(s -\frac{1}{2}, \Ik^{n,2}(\pi_{K_{H_2}} \phi))$ is holomorphic at $s= \rho_{H_2}$.
\end{itemize}
\vskip 5pt

Thus,  (\ref{E:Eis}) has a pole of order $2$ at $s = \rho_{H_2}$ and the leading term of (\ref{E:Eis}) there is equal to 
\[  - \frac{\kappa_2}{\kappa_1} \cdot \alpha_2 \cdot B^{n,1}_0(\Ik^{n,2}(\pi_{K_{H_2}}\phi)), \]
while its residue there is equal to
\[   C \cdot B_0^{n,1}(\Ik^{n,2}(\pi_{K_{H_2}}\phi))   - \frac{\kappa_2}{\kappa_1} \cdot \alpha_2 \cdot 
B^{n,1}_1(\Ik^{n,2}(\pi_{K_{H_2}}\phi))
\]
for some constant $C$. Now observe that 
\[ B^{n,1}_1(\Ik^{n,2}(\pi_{K_{H_2}}\phi)) = 0 \]
because $B^{n,1}(s,\Ik^{n,2}(\pi_{K_{H_2}}\phi))$ is an even function in $s$. Moreover, 
by Theorem \ref{T:1st}(i), 
\[  B_0^{n,1}(\Ik^{n,2}(\pi_{K_{H_2}}\phi))  = \kappa_{2,1}^{-1} \cdot A^{n,2}_{-1}(\phi) \in {\rm Im} A^{n,2}_{-1}. \]
Thus, we conclude that when $H_{r-1} = \OO_{1,1}$,
  \begin{equation} \label{E:B-side-fin3}
   \left(\text{Residue at $s=\rho_{H_r}$ of the third term of
  (\ref{E:B-side})}\right) \in {\rm Im} A^{n,r}_{-1}. \end{equation}
\vskip 10pt

\noindent (c) When $r=1$ so that $H_{r-1}$ is anisotropic, we have
\vskip 5pt
\begin{itemize}
\item the Tate zeta integral has a simple pole with residue $-1$  as above;
\item the function $c_1(s)/c_0(s)$ has a simple pole at $s = \rho_{H_1}$ with residue $\kappa_1$;
\item $B^{n, r-1}(s -\frac{1}{2}, \Ik^{n,1}(\pi_{K_{H_r}} \phi))= I^{n,0}(\Ik^{n,1}(\phi))$.
\end{itemize}
Thus, (\ref{E:Eis}) has a pole of order $2$ at $s = \rho_{H_1}$ and the leading term of (\ref{E:Eis}) there is equal to 
\[  - \frac{\kappa_1}{\kappa_0} \cdot \alpha_1 \cdot  I^{n,0}(\Ik^{n,1}(\phi)), \]
while its residue there is equal to
\[   C \cdot  I^{n,0}(\Ik^{n,1}(\phi)) \]
for some constant $C$. 
On the other hand, Theorem \ref{T:1st}(i) implies that
 \[ I^{n,0}(\Ik^{n,1}(\phi)) = \kappa_{1,0}^{-1} \cdot A^{n,1}_{-1}(\phi)  \in {\rm Im} A^{n,1}_{-1}, \]
 so that
    \begin{equation} \label{E:B-side-fin4}
   \left(\text{Residue at $s=\rho_{H_r}$ of the third term of
  (\ref{E:B-side})}\right) \in {\rm Im} A^{n,r}_{-1}. \end{equation}
\vskip 10pt
   
\subsection{\bf Conclusion of proof.}
  We can now assemble the pieces together and prove Theorem
  \ref{T:main} in the case $\mathcal{N} = 1$. Namely by (\ref{E:A-side-fin}),
  (\ref{E:B-side-fin1}),  (\ref{E:B-side-fin2}),  (\ref{E:B-side-fin3}) and (\ref{E:B-side-fin4}), the equation
  (\ref{E:compare}) is written as
  \begin{equation}\label{E:second_term_one_step_before}
 A^{n,r}_0(\phi)  
=   B^{n,r}_{-1}(\phi)  -
 \{  \frac{\kappa_r}{\kappa_{r-1}} \cdot \alpha_r  
  \cdot B_0^{n, r-1}(\Ik^{n,r}(\pi_{K_{H_r}}\phi)) \}  
\qquad\mod {\im A^{n,r}_{-1}}
\end{equation}
where $\{ ...\}$ on the RHS is interpreted to be $0$ when $r=1$ or $H_{r-1} = \OO_{1,1}$.
\vskip 5pt

  Moreover, we observe that the two terms on the $B$-side corresponding
  to $w^+$ and $w^-$ actually have poles of order $2$ at $s =
  \rho_{H_r}$. These poles of order 2 must cancel each other, so that
  we have:
 \begin{equation}\label{E:first_term_one_step_before}
B^{n,r}_{-2}(\phi)   =\frac{\kappa_r}{\kappa_{r-1}} \cdot
 \alpha_r   \cdot B_{\ast}^{n,r-1}(\Ik^{n,r}(\pi_{K_{H_r}}\phi))  =
 \frac{\kappa_r}{\kappa_{r-1}} \cdot \alpha_r  \cdot
 \kappa_{r,r-1}^{-1}   \cdot
  A_{-1}^{n,r}(\phi),
\end{equation}
where $\ast$ is as in Theorem \ref{T:1st}(i).
To complete the proof of  Theorem \ref{T:main}, it remains to note:
\vskip 5pt

\begin{Lem} \label{L:kappa3}
The constant
\[  \frac{\kappa_r}{\kappa_{r-1}} \cdot \alpha_r  \cdot \kappa_{r, r-1}^{-1}  = 1. \]
\end{Lem}
\begin{proof}
If $r=1$, then the identity follows because $\kappa_0 = 1 = \alpha_1$ and $\kappa_{1,0} = \kappa_1$. 
For $r >1$, the desired identity  follows by ``Iwasawa decomposition in stages". More precisely,
consider the non-maximal parabolic subgroup $R$ of $H_r$ stabilizing
the flag of isotropic spaces:
\[   X_1 = \langle x_1 \rangle  \subset X_r  \subset V_r.  \]
Then $R$ is contained in two maximal parabolic subgroups $P(X_1)$ and $P(X_r)$.
Now, relative to the Iwasawa decomposition $H_r(\A) = R(\A) \cdot K_{H_r}$,  
there is a constant $\gamma$ such that  
  \[  dh  = \gamma \cdot d_l r \cdot dk,  \]
  where $d_lr$ is the left Haar measure of $R(\A)$.
 The constant $\gamma$ can be computed in two ways. First, one may
 express $dh$ using the Iwasawa decomposition of $H_r$ associated to
 $P(X_1)$ and then using the  Iwasawa decomposition  of $L(X_1) =
 \GL(X_1) \times H_{r-1}$ associated to $P^{r-1}({X'}_{r-1})$. This
 gives
 \[  \gamma =  \frac{\tau(H_r)}{\tau(H_{r-1})} \cdot \kappa_{r, r-1} \cdot \frac{\tau(H_{r-1})}{\tau(H_0)} \cdot \kappa_{r-1}  = \frac{\tau(H_r)}{\tau(H_0)} \cdot \kappa_{r,r-1} \cdot \kappa_{r-1}.  \]
 On the other hand, one may express $dh$ using the Iwasawa
 decomposition of $H_r$ associated to $P(X_r)$ followed by the Iwasawa
 decomposition associated to the maximal parabolic of $L(X_r) =
 \GL(X_r)$ stabilizing the line $X_1$. This yields
 \[  \gamma =\frac{\tau(H_r)}{\tau(H_0)} \cdot  \kappa_r \cdot \alpha_r \]
which proves the lemma.
 \end{proof}
\vskip 5pt

If one uses the result of the lemma, one can see that
(\ref{E:second_term_one_step_before}) and
(\ref{E:first_term_one_step_before}) are written as in Theorem
\ref{T:main}. This completes the base step in the inductive
proof of Theorem \ref{T:main}. 

  \vskip 15pt
  
 \section{\bf Proof of Theorem \ref{T:main}: Inductive Step}
 In this section, we complete the proof of Theorem \ref{T:main} by
 establishing the inductive step of the argument.  Thus, suppose we
 want to prove Theorem \ref{T:main} for the pair $(W_n, V_r)$ with
 associated 
 $\mathcal{N}_0 = m - d(n)$. We consider the pair $(W_{n+1}, V_r)$
 with associated $\mathcal{N} = m - d(n+1) = \mathcal{N}_0 -1 \geq 1$, so
 that we may appeal to the induction hypothesis.
 \vskip 5pt
 
 As in the proof of the base case, we fix $\phi_1 \in
 \mathcal{S}(y_1^* \otimes V_r)(\A)$ satisfying the conditions in \S
 \ref{SS:base}. Then for arbitrary  $\phi \in
 \mathcal{S}(Y_n'^* \otimes V_r)(\A)$, set $\tilde{\phi} = \phi_1
 \otimes \phi$.
 Applying the induction hypothesis, we have:
 \begin{equation} \label{E:first} 
 A^{n+1, r}_{-1}(\tilde{\phi}) =   
 B^{n+1,r}_{-2}(\tilde{\phi}), 
\end{equation}
 and 
\begin{equation} \label{E:second}
   A_0^{n+1,r}(\tilde{\phi})  =  
   B^{n+1,r}_{-1}(\tilde{\phi}) -   \kappa_{r,r'} \cdot B_0^{n+1,
     r'}(\Ik^{n+1,r}(\pi_{K_{H_r}}\tilde{\phi}))   +
   A^{n+1,r}_{-1}(\varphi) 
\end{equation}
 for some $\varphi \in \mathcal{S}(Y_{n+1}^* \otimes V_r)(\A)$, where $r'$ satisfies $m_0 + r + r'  = d(n+1)$.
 Note that the second term on the RHS of (\ref{E:second}) is always present here, because $H_{r'}$ cannot be anisotropic or equal to $\OO_{1,1}$ in the present setting. Indeed,
  the complementary space of $V_r$ with respect to $W_n$ is $V_{r'-1}$ so that $r' \geq 1$ and $H_{r'}$ is not anisotropic. On the other hand, if $H_r = \OO_{r,r}$ (so that $m_0 = 0$ and $\epsilon_0 =1$), then $r+ r'  = n+2$, and with our assumption that $r \leq n$, we see that $r' > 1$ in the case $\epsilon_0 =1$ and $m_0 = 0$.  
 \vskip 5pt
 
 \subsection{\bf First term identity.}
 Let us examine (\ref{E:first}) first. Computing the constant term along $U(Y_1)$, we have:
 \[   A^{n+1, r}_{-1}(\tilde{\phi})_{U(Y_1)} =   B^{n+1,r}_{-2}(\tilde{\phi})_{U(Y_1)}. \] 
 Using Lemma \ref{L:constant} to compute the constant terms on both
 sides, we shall extract the terms on both sides where $\GL(Y_1)
 \subset L(Y_1)$ acts by the character $\chi \cdot |-|^{m/2}$. 
 Then on the B-side, we see that only the term for $w^+$ contributes,
 and Proposition \ref{P:key1} gives:
  \[ \left(\text{the $\chi \cdot |-|^{m/2}$-part of
    $B^{n+1,r}_{-2}(\tilde{\phi})_{U(Y_1)}$}\right) = B^{n,r}_{-2}(\phi). \]
 On the other hand, on the A-side, only the first term in Lemma \ref{L:constant}(ii)
 contributes, giving
 \[  \left(\text{the $\chi \cdot |-|^{m/2}$-part of $A^{n+1, r}_{-1}
   (\phi)_{U(Y_1)}$}\right) = A^{n,r}_{-1}(\phi). \]
 This proves Theorem \ref{T:main}(i). 
 \vskip 10pt
 
 \subsection{\bf Second term identity.}
 Now consider (\ref{E:second}). Computing the constant term along $U(Y_1)$ gives:
\begin{align}\label{E:induction_step}
A_0^{n+1,r}(\tilde{\phi})_{U(Y_1)} &=  
B^{n+1,r}_{-1}(\tilde{\phi})_{U(Y_1)} \\
& - \kappa_{r,r'}  
\cdot B_0^{n+1,  r'}(\Ik^{n+1,r}(\pi_{K_{H_r}}\tilde{\phi}))_{U(Y_1)} \notag\\
&+ A^{n+1,r}_{-1}(\varphi)_{U(Y_1)}.\notag
\end{align}
Each constant term in this identity can be computed by using Lemma
\ref{L:constant} just as we did for the base step, and is written as a
sum of two or three automorphic forms on
$\GL(Y_1)\times G_{n}$. Then by extracting the terms with $\GL(Y_1)$ acting
via $\chi \cdot |-|^{m/2}$, we can deduce the desired second term
identity. To be specific, for the LHS of (\ref{E:induction_step}), only the trivial Weyl group
element contributes, and we have
\begin{equation}\label{E:induction_step1}
\left(\text{the $\chi \cdot |-|^{m/2}$-part of
  $A_0^{n+1,r}(\tilde{\phi})_{U(Y_1)} $}\right)
= A_0^{n,r}(\phi),
\end{equation}
when viewed as an automorphic form on $G_n$. For the first term of the
RHS of (\ref{E:induction_step}), only the Weyl group element $w^+$ contributes
and we have
\begin{equation}\label{E:induction_step2}
\left(\text{the $\chi \cdot |-|^{m/2}$-part of
  $B^{n+1,r}_{-1}(\tilde{\phi})_{U(Y_1)}$}\right)
= B_{-1}^{n,r}(\phi).
\end{equation}
For the third term of
the RHS of (\ref{E:induction_step}), we have
\begin{equation}\label{E:induction_step3}
 \left(\text{the $\chi \cdot |-|^{m/2}$-part of $A^{n+1,
     r}_{-1}(\varphi)_{U(Y_1)} $}\right)\in {\rm Im} A^{n,r}_{-1}
\end{equation}
as we have noted above for the first term identity. 
\vskip 5pt

Finally for the
second term of the RHS of (\ref{E:induction_step}), 
only the Weyl
group element $w^-$ contributes and one sees that 
\begin{align}\label{E:induction_step4}
&\left(\text{the $\chi \cdot |-|^{m/2}$-part of $B_0^{n+1,
    r'}(\Ik^{n+1,r}(\pi_{K_{H_r}}\tilde{\phi}))_{U(Y_1)}$}\right)\\
&=\Val_{s=\rho_{H_{r'}}} E^{n, r'-1}\left( s - \frac{1}{2}, M(w^-,s)
f^{n+1. r'}(s, \Ik^{n+1,r}(\pi_{K_{H_r}}\tilde{\phi}))|_{G_n})\right)\notag.
 \end{align}
This is the most complex case. First by the arguments in the
$w^-$-term for the base step and Lemma \ref{L:comm}, we can write the
Eisenstein series in (\ref{E:induction_step4}) as
\begin{align} \label{E:r'}
&E^{n, r'-1}\left( s - \frac{1}{2}, M(w^-,s)
f^{n+1. r'}(s, \Ik^{n+1,r}(\pi_{K_{H_r}}\tilde{\phi}))|_{G_n}\right)\\
=&c_{r'}(s) \cdot E^{n, r'-1}\left(\frac{1}{2}-s, f^{n+1. r'}(-s,
  \Ik^{n+1,r}(\pi_{K_{H_r}}\tilde{\phi})|_{G_n})\right).\notag
\end{align}
 Noting that
\[  \Ik^{n+1, r}(\pi_{K_{H_r}} \tilde{\phi}) = \Ik^{1, r,r'}(\phi_1)
\otimes \Ik^{n,r,r'}(\pi_{K_{H_r}}\phi) \]
is $K_{H_{r'}}$-invariant,
 Proposition \ref{P:key2} implies that (\ref{E:r'}) is  equal to
\begin{align}\label{E:r'-1}   
&c_{r'}(s) \cdot \alpha_{r'} \cdot Z_1(s  -
(n+1-r') + \rho_{H_{r'}}, \Ik^{1,r,r'}(\phi_1))\\
&\hskip 1in \cdot E^{n,r'-1}\left(\frac{1}{2}-s, f^{n,
  r'-1}(\frac{1}{2}-s, 
\Ik^{n,r', r'-1}(\Ik^{n,r,  r'}(\pi_{K_{H_r}}\phi)))\right)\notag\\
=& c_{r'}(s) \cdot \alpha_{r'} \cdot Z_1(s  -
(n+1-r') + \rho_{H_{r'}}, \Ik^{1,r,r'}(\phi_1))\notag\\
&\hskip 1in \cdot   E^{n,r'-1}\left(\frac{1}{2}-s, f^{n,
  r'-1}(\frac{1}{2}-s, \Ik^{n, r, r'-1}(\pi_{K_{H_r}} \phi))\right).\notag
\end{align}
Note that since $m_0+ r + r' = d(n+1)$, we have $m_0 + r+ (r'-1) =
d(n)$, so that $V_r$ and $V_{r'-1}$ are complementary with respect to
$W_n$ and 
\[  \Ik^{n,r, r'-1} = \Ik^{n,r}. \]
Thus (\ref{E:r'-1}) is equal to
\[  c_{r'}(s) \cdot \alpha_{r'} \cdot  Z_1(s  -
(n+1-r') + \rho_{H_{r'}}, \Ik^{1,r,r'}(\phi_1))  \cdot B^{n, r'-1}(\frac{1}{2}-s,
\Ik^{n,r}(\pi_{K_{H_r}}\phi)) \]
 which is in turn equal to
 \begin{equation}\label{E:r'-2}
 \frac{c_{r'}(s)}{c_{r'-1}(s - \frac{1}{2})} \cdot \alpha_{r'}
 \cdot Z_1(s  - (n+1-r') + \rho_{H_{r'}}, \Ik^{1,r,r'}(\phi_1))  \cdot
 B^{n, r'-1}(s-\frac{1}{2}, \Ik^{n,r}(\pi_{K_{H_r}}\phi)).
\end{equation}
 We thus need to consider the zeroth term in the Laurent expansion of (\ref{E:r'-2}) at
 $s = \rho_{H_{r'}} = \rho_{H_{r'-1}} + 1/2 $.
 \vskip 5pt
 
 As in the argument for the base step in the previous section, it will
 be convenient to consider different cases separately. If $H_{r'-1}$
 is anisotropic or equal to $\OO_{1,1}$, then as we have seen in the
 previous section,  
 $B^{n, r'-1}(s-\frac{1}{2}, \Ik^{n,r}(\pi_{K_{H_r}}\phi))$ is
 holomorphic at $s = \rho_{H_{r'}}$ with vanishing first derivative
 there, whereas the product of the other terms has a simple pole
 (arising from the pole of $ c_{r'}(s) / c_{r'-1}(s - \frac{1}{2})$ at
 $s = \rho_{H_{r'}}$). Thus, on taking the zeroth Laurent coefficient,
 one obtains 
 \[  C \cdot B^{n, r'-1}_0(\Ik^{n,r}(\pi_{K_{H_r}}\phi))   \in {\rm Im} A^{n,r}_{-1} \]
  by Theorem \ref{T:1st}(i).   This completes the proof of the inductive step for these cases.
  \vskip 5pt
  
  Finally, we assume that $H_{r'-1}\ne \OO_{1,1}$ is isotropic.
   Then we know from Lemma
 \ref{L:kappa}(i) that $c_{r'}(s)/c_{r'-1}(s- 1/2)$ is holomorphic with
 value $\kappa_{r'}/\kappa_{r'-1}$. The zeta integral is evaluated at
 $-(r-r') \leq -1$ where it is holomorphic, whereas 
 $B^{n, r'-1}(s-\frac{1}{2}, \Ik^{n,r}(\pi_{K_{H_r}}\phi))$ has a pole
 of order $1$.  Taking the zeroth term in the Laurent expansion, 
 (\ref{E:r'-2}) becomes
\begin{align}\label{E:r'-3}
 &\kappa_{r'}/\kappa_{r'-1} \cdot \alpha_{r'} \cdot
 Z_1(-(r-r'), \Ik^{1,r,r'}(\phi_1)) \cdot  B^{n,
   r'-1}_0(\Ik^{n,r}(\pi_{K_{H_r}}\phi))  \\
 &\hskip 2.5in \mod {\im(B^{n,r'-1}_{-1} \circ \Ik^{n,r} \circ
   \pi_{K_{H_r}})}.\notag
\end{align}
 By the first term identity in Theorem \ref{T:1st} and Lemma
 \ref{L:kappa3}, (\ref{E:r'-3}) becomes
  \begin{equation} \label{E:comple}
    \kappa_{r',r'-1} \cdot Z_1(-(r-r'), \Ik^{1,r,r'}(\phi_1)) \cdot
    B^{n, r'-1}_0(\Ik^{n,r}(\pi_{K_{H_r}}\phi))  \, \mod {\im
    A^{n,r}_{-1}}. 
\end{equation}
 
 \vskip 5pt
 It remains to obtain a better handle of  the term
 \[   Z_1(-(r-r'), \Ik^{1,r,r'}(\phi_1)):= {\rm Val}_{s=0} 
\int_{\A_E^{\times}}  \int_{\A_E^{r-r'}}  \phi_1 \left(   \begin{matrix}
x_{r-r'}  \\
 t \\
 0_{m-r+r'-1} \end{matrix} \right) \, |t|^{s - (r-r')} \, dx_{r-r'}  \,  d^{\times} t. \] 
 This is given by the following lemma (see \cite[Pg. 231, (7.1)]{I1}).
 \vskip 5pt
 
 \begin{Lem}
 Let $\varphi \in \mathcal{S}(\A^k)$ with $k \geq 2$ be such that
 $\varphi$ is $K$-invariant, where $K \subset \GL_k(\A)$ is a maximal
 compact subgroup and where  $\GL_k(\A)$ acts on  $ \mathcal{S}(\A^k)$ by $g \cdot \varphi(x)  = \varphi(g^{-1} \cdot x)$. Then
\[  
  {\rm Val}_{s=0} \left( 
\int_{t \in \A^{\times} }   \int_{x_{k-1} \in \A^{k-1}} \varphi   
\begin{pmatrix}
x_{k-1}  \\
 t  \end{pmatrix}\cdot  |t|^{s - (k-1)} \, dx_{k-1}  \,  d^{\times} t
\right) =   \frac{\varphi(0)}{\alpha_{k}} 
 \] 
 where the constant $\alpha_{k}$ is as defined in (\ref{E:alphar}),
 and where the double integral on the LHS is absolutely convergent
 when ${\rm Re}(s) \gg 0$ and has a meromorphic continuation to $\C$.
 \end{Lem}
 \begin{proof}
Let $z(s,\varphi)$ denote the double integral on the LHS. Choose a
function $\varphi^\prime\in \mathcal{S}(M_{k\times (k-1)}(\A))$ that
is left invariant by $K$ with $\varphi^\prime(0)=1$. Construct
$\Phi=\varphi^\prime\otimes \varphi \in \mathcal{S}(M_{k\times
  k}(\A))$ and consider the zeta integral
$Z_k(s,\Phi)=\int_{GL_k(A)}\Phi(A)|\det A|^sdA$. There is a measure
decomposition $dA=\alpha_k |t|^{-(k-1)}dx_{k-1}dB d^\times t d\lambda$
with respect to the Iwasawa decomposition
$A=\smalltwomatrix{B}{}{\leftup{t}{x_{k-1}}}{t}\lambda$, where
$x_{k-1}\in \A^{k-1}$, $B\in \mathrm{GL}_{k-1}(\A)$, $t\in \A^\times$,
and $\lambda \in K$. Let $\varphi_1^\prime$ be the restriction of
$\varphi^\prime$ to $M_{(k-1)\times (k-1)}(\A)\times \{0_{k-1}\}$,
then
\begin{align*}
Z_k(\Phi,s)=&\int_{\mathrm{GL}_k(\A)}\Phi(\leftup{t}{A})|\det A|^sd A\\
=&\alpha_k \int_K \int_{\mathrm{GL}_{k-1}(\A)} \int_{\A^\times }\int_{\A^{k-1}}\Phi\big(\leftup{t}{\lambda}\smalltwomatrix{\leftup{t}{B}}{x_{k-1}}{}{t}\big)|\det B|^s |t|^{s-(k-1)}dx_{k-1}d B d^\times t d\lambda\\
=&\alpha_k \int_{\mathrm{GL}_{k-1}(\A)} \int_{\A^\times } \int_{\A^{k-1}}\Phi\smalltwomatrix{B}{x_{k-1}}{}{t}|\det B|^s |t|^{s-(k-1)}dx_{k-1}d B d^\times t\\
=&\alpha_k Z_{k-1}(\varphi^\prime_1,s)z(s,\varphi).
\end{align*}
Because $Z_k(\Phi,s)$ and $Z_{k-1}(\varphi^\prime_1,s)$ both have a simple pole at $s=0$ and the residues are $-\Phi(0), -\varphi^\prime_1(0)$ respectively (c.f. section ~\ref{zetaintegral}), we have
\[
{\rm Val}_{s=0}z(s,\varphi)=\frac{-\Phi(0)}{-\varphi_1^\prime(0)\alpha_k }=\frac{\varphi(0)}{\alpha_k}.
\]

\end{proof}
 \vskip 10pt
 
 Applying the lemma, we deduce that 
  \[   Z_1(-(r-r'), \Ik^{1,r,r'}(\phi_1))  =
  \frac{1}{\alpha_{r-r'+1}},\]
  so that (\ref{E:comple}) is equal to
\[   \frac{\kappa_{r',r'-1}}{\alpha_{r-r'+1} } \cdot  B^{n,
  r'-1}_0(\Ik^{n,r}(\pi_{K_{H_r}}\phi))  \, \mod {\im A^{n,r}_{-1}}.
\]
 It remains to show that
 \[  \frac{\kappa_{r', r'-1}}{\alpha_{r-r'+1} }
=  \frac{\kappa_{r,r'-1}}{\kappa_{r,r'}}. 
\]
This again follows by ``Iwasawa decomposition in stages". Namely, we
consider the parabolic subgroup $R$ of $H_r$ stabilizing the flag
\[  X_{r-r'} \subset X_{r-r' +1} \subset V_r. \]
Then $R$ is contained in the maximal parabolic subgroups $P(X_{r-r'})$ and $P(X_{r-r'+1})$.
 By the same argument as in Lemma \ref{L:kappa3}, one obtains:
 \[  \kappa_{r,r'} \cdot \kappa_{r', r'-1} =  \kappa_{r, r'-1} \cdot \alpha_{r-r'+1},\]
 which is the desired identity.

Hence (\ref{E:comple}) (which is actually  (\ref{E:induction_step4}))
becomes
\begin{equation}\label{E:induction_step4'}
  \frac{\kappa_{r,r'-1}}{\kappa_{r,r'}} \cdot  B^{n,
  r'-1}_0(\Ik^{n,r}(\pi_{K_{H_r}}\phi))  \, \mod {\im A^{n,r}_{-1}}.
\end{equation}
Hence if one combines the four equations (\ref{E:induction_step1})
(\ref{E:induction_step2}) (\ref{E:induction_step3}) and
(\ref{E:induction_step4'}) with (\ref{E:induction_step}), one can
see that one has the desired second term identity for the pair
$(W_n, V_r)$ with associated $\mathcal{N}_0 = m - d(n)$, which
completes the inductive step.
 \vskip 5pt

Hence the proof of Theorem \ref{T:main} is now complete.
 \vskip 15pt

 \section{\bf Rallis Inner product Formula}
 
 In this section, we use the first and second term identities of the
 regularized Siegel-Weil formula to derive the Rallis inner product
 formula and prove our non-vanishing theorem of global theta lifts.
 \vskip 5pt
 
 \subsection{\bf Dual Pairs.}
 As before, let $V_r$ be an $m$-dimensional $\epsilon$-Hermitian space
 over $E$ with Witt index $r$ and anisotropic kernel $V_0$ of
 dimension $m_0$, so that $m = m_0+2r$.   
 Let $U_n$ be an $n$-dimensional $-\epsilon$-Hermitian space over
 $E$. Then the pair of spaces $(U_n, V_r)$ determines the dual reductive
 pair $G(U_n)\times H(V_r)$ as follows.
 \vskip 5pt

 \begin{center}
 \begin{tabular}{|c|c|c|c|}
 \hline
 $\epsilon_0 $ &   &  $G(U_n)$ & $H(V_r)$ \\
 \hline
 $0$  &  & $\U_n$ & $\U_m$  \\
 \hline
\multirow{2}{*}{$-1$} & $n$ even &   $\OO_n$ & $\Sp_m$  \\
 \cline{2-4}
  &$n$ odd & $\OO_n$ & $\Mp_m$ \\
  \hline
  \multirow{2}{*}{$1$} & $m$ even & $\Sp_n$ & $\OO_m$ \\
  \cline{2-4}
  & $m$ odd & $\Mp_n$ & $\OO_m$ \\
  \hline
\end{tabular}
\end{center}
 \vskip 5pt
 
 Observe that, while we have always taken $H(V_r)$ to be a linear
 group in the earlier sections of the paper, we have now allowed
 $H(V_r)$ to be the metaplectic group.
 
\vskip 5pt
    
 \subsection{\bf Theta lifts.}
 To  define the global theta lifting from $G(U_n)$ to $H(V_r)$, one
 needs to choose some auxiliary data for the splitting of metaplectic
 covers so as to define the Weil representation of $G(U_n) \times
 H(V_r)$. Namely, one fixes a pair of Hecke characters $ (\chi_U,
 \chi_V)$ as follows:
 
 \begin{align*}
\chi_U & = 
 \begin{cases}
  \text{a Hecke character of $\A_{E^{\times}}$ such that
  $\chi_W|_{\A^{\times}} = \chi_E^n$}
  & \text{if $\epsilon_0 = 0$,} \\
  \text{the quadratic Hecke character of $\A^{\times}$
  associated to $\disc\, U_n$}
  & \text{if $\epsilon_0 =-1$,} \\
  \text{the trivial character of $\A^{\times}$}
  & \text{if $\epsilon_0 = 1$;}
 \end{cases}\\
 \chi_V & = 
 \begin{cases}
  \text{a Hecke character of $\A_{E^{\times}}$ such that 
  $\chi_V|_{\A^{\times}} = \chi_E^m$}
  & \text{if $\epsilon_0 = 0$,} \\
  \text{the trivial character of $\A^{\times}$}
  & \text{if $\epsilon_0 = -1$,} \\
  \text{the quadratic Hecke character of $\A^{\times}$
  associated to $\disc\,V_r$}
  & \text{if $\epsilon_0 = 1$.}
 \end{cases}
\end{align*}

Having fixed $(\chi_U, \chi_V)$, one then has the associated  Weil representation
 $\omega_{\psi, \chi_U,\chi_V}$ of $G(U_n) \times H(V_r)$. If   $U_n
 \otimes_E V_r  = \mathbb{X} \oplus \mathbb{Y}$ is a Witt
 decomposition of the symplectic space 
 $U_n \otimes_E V_r $, then  $\omega_{\psi, \chi_U,\chi_V}$ may be
 realized on the Schwartz space $\mathcal{S}(\mathbb{X})(\A)$.
 One then has the usual automorphic realization
 \[  \theta: \omega_{\psi, \chi_U,\chi_V} \longrightarrow
 \{\text{Functions on $[G(U_n)] \times [H(V_r)]$} \}\] 
given by
\[
\theta(\phi)(g,h)=\sum_{x\in\mathbb{X}(F)}\omega_{\psi,\chi_U,\chi_V}(g,h)\phi(x)
\]
for $(g,h)\in G(U_n)(\A)\times H(V_r)(\A)$ and
$\phi\in\mathcal{S}(\mathbb{X})(\A)$.
For a cuspidal representation $\pi$ of $G(U_n)$, we consider its
global theta lift $\Theta_{n,r}(\pi)$ to $H(V_r)$, so that
$\Theta_{n,r}(\pi)$ is the automorphic subrepresentation of $H(V_r)$
spanned by the automorphic forms
\[  \theta_{n,r}(\phi,f)(h) = \int_{[G(U_n)]} \theta(\phi) (g,h) \cdot
\overline{f(g)} \, dg \]
for $\phi \in  \omega_{\psi, \chi_U,\chi_V}$ and $f \in \pi$. 
\vskip 5pt

We assume that $\Theta_{n,j}(\pi) = 0$ for $j < r$. Then a well-known
result of Rallis \cite{R1} says that $\Theta_{n,r}(\pi)$ is contained in the
space of cusp forms on $H(V_r)$.  We are interested in detecting
whether $\Theta_{n,r}(\pi)$ is nonzero.
 \vskip 5pt
 
 \subsection{\bf Doubling See-Saw.}
 To detect if $\Theta_{n,r}(\pi)$ is nonzero, we compute the inner product
 \[  \langle \theta_{n,r}(\phi_1, f_1) , \theta_{n,r}(\phi_2, f_2) \rangle \]
 for $\phi_i \in \omega_{\psi, \chi_U,\chi_V}$ and $f_i \in \pi$. The
 mechanism for this computation is the following doubling see-saw
 diagram
  \[
    \xymatrix{
    G(W_n)\ar@{-}[d]_{i}\ar@{-}[dr]&H(V_r)\times H(V_r)\ar@{-}[d]\\
    G(U_n)\times G(U^-_n)\ar@{-}[ur]&H(V_r)^{\Delta},
    }
\]  
where $U_n^-$ is the space $U_n$ with the form scaled by $-1$, and
\[  W_n = U_n \oplus (U_n^-) \]
  is the split $\epsilon$-Hermitian space of dimension $2n$. Indeed,
  we choose a Witt decomposition of $W_n$ to be
  \[ W_n =  Y_n  \oplus Y_n^* \]
  with
  \[ Y_n  = U_n^{\Delta} = \{  (u,u):  u \in U_n \}  \]
   and
   \[ Y_n^* = U_n^{\nabla} = \{ (u, -u): u \in U_n \}. \]
 \vskip 5pt
 
 The Weil representation $\omega_{\psi, \chi_V, W_n, V_r}$ of
 $G(W_n)\times H(H_r)^{\Delta}$ (which is what we have denoted by
 $\omega_{n,r}$ in the previous sections) can be realized on
 $\mathcal{S}(Y_n^*\otimes V_r)$ such  that $H(V_r)^{\Delta}$ acts by
\[   \omega_{\psi, \chi_V, W_n, V_r}(h) \phi(x)  =  \varphi(h^{-1}\cdot x). \] 
 Thus, the  branch of the see-saw involving the dual pair 
 \[  G(W_n)\times H(V_r)^{\Delta} \]
 gives rise to the (regularized) theta integral which we studied in
 the earlier parts of this paper. By \cite{K2} and \cite{HKS}, one knows that
 \[  \omega_{\psi, \chi_V, W_n, V_r}|_{G(U_n) \times G(U_n)} 
  =  \omega_{\psi, \chi_V, U_n, V_r} \otimes \left(
    \omega_{\psi,\chi_V, U_n, V_r}^{\vee} 
\cdot (\chi_V \circ \det) \right)|_{G(U_n) \times G(U_n)} . \]
  Indeed, there is an isomorphism \cite[Pg. 182]{Li}
  \[  \delta: 
   \omega_{\psi, \chi_V, U_n, V_r} \otimes \left( 
\omega_{\psi,\chi_V, U_n, V_r}^{\vee} \cdot (\chi_V \circ \det) \right)
   \longrightarrow   \omega_{\psi, \chi_V, W_n, V_r} \]
  such that
  \[  \delta(\phi_1\otimes \overline{\phi_2})(0)  = \langle \phi_1, \phi_2 \rangle \]
  for $\phi_i \in \Sw(\mathbb{X})(\A)$.
   \vskip 10pt
  
 \subsection{\bf Inner product.} 
 We remind the reader of the different possible ranges for the triple $(n,m, r)$:
\vskip 5pt

 \begin{center}
 
 \begin{tabular}{|c|c|c|}
 \hline
 \multirow{2}{*}{$m \leq d(n)$} & $r=0$ &  convergent range   \\
 \cline{2-3}
  &$r >0$ & boundary/first term range  \\
  \hline
  \multirow{2}{*}{$d(n) <  m \leq 2 \cdot d(n)$} & $m \leq d(n) +r$  & second term range \\
  \cline{2-3}
  & $m  > d(n) +r$ &  convergent range \\
  \hline
   $ m> 2 \cdot d(n)$ &   & convergent range  \\
  \hline
\end{tabular}
\end{center}
 \vskip 5pt
 
\noindent  Since we are primarily interested in the second term range, we shall assume henceforth that
 \[  d(n)  < m \leq  2 \cdot d(n). \]
  This implies that
  \[  0 < s_{m,n} = \frac{m-d(n)}{2} \leq \frac{d(n)}{2}. \]
  Now the see-saw identity for the doubling see-saw gives:
 \vskip 5pt
  
 \begin{Prop}\label{P:inner_product1}
Suppose that $d(n) < m \leq 2 \cdot d(n)$. Let $\pi$ be an irreducible 
cuspidal automorphic representation of $G(U_n)$ such that
$\Theta_{n,j}(\pi) = 0$ for $j < r$, so that $\Theta_{n,r}(\pi)$ is
cuspidal. 
\vskip 5pt

\noindent (i) If the triple $(n,m, r)$ is in the second term range and  $r \leq n$, 
then one has 
 \begin{align*}
&\langle \theta_{n,r}(\phi_1, f_1) , \theta_{n,r}(\phi_2, f_2) \rangle\\
&= \tau(H_r) \cdot  \int_{[G(U_n) \times G(U_n)]}    B^{n,r}_{-1}( \delta(\phi_1
\otimes \overline{\phi}_2))(g_1, g_2) \cdot \overline{f_1(g_1)} \cdot
f_2(g_2) \cdot \chi_V^{-1}(\det(g_2)) \,dg_1 \, dg_2\\
&=[E:F] \cdot  \int_{[G(U_n) \times G(U_n)]}    A^{n,r}_0(
\delta(\phi_1 \otimes \overline{\phi}_2))(g_1, g_2) \cdot
\overline{f_1(g_1)} \cdot f_2(g_2) \cdot \chi_V^{-1}(\det(g_2))\,
dg_1 \, dg_2
\end{align*}
for $\phi_i \in \omega_{\psi, \chi_U,\chi_V}$ and $f_i \in \pi$.   
 \vskip 5pt
 
 \noindent (ii) If the triple $(n,m,r)$ is in the convergent range, then one has the same identity relating the first term and the third term in (i).
  \end{Prop}
 \vskip 5pt
 \begin{proof}
(i) The first equality is proved in the same way as \cite[Prop. 6.1]{GT}. The second
equality is proved in the same way as \cite[Prop. 6.3]{GT}
by using the second term identity of Theorem \ref{T:main}.
 \vskip 5pt
 
 \noindent (ii) This similarly  follows from the Siegel-Weil formula (\cite{KR1,KR2} and \cite{I3}) in the convergent case.
 \end{proof}
 \vskip 10pt

\subsection{\bf  Doubling zeta integral.}
The last integral in (i) of the above proposition is the so-called doubling
zeta integral.  More precisely, for $f_i \in \pi$ and $\Phi(s)$ a
holomorphic section of  $I^n_n(s, \chi_V) := {\rm
  Ind}^{G(W_n)}_{P(Y_n)} (\chi_V \circ \det) |\det|^s$, one sets
\[  Z(s, \Phi, f_1, f_2) = 
\int_{[G(U_n)] \times [G(U_n)]} E(s,  \Phi)( g_1, g_2) \cdot \overline{f_1(g_1)} \cdot f_2(g_2) \cdot  
 \chi_V^{-1}(\det(g_2))  \\,  dg_1 \, dg_2. \]
By \cite{PS-R87}, this converges for ${\rm Re}(s) \gg 0$ and  extends to
a meromorphic function on $\C$. 
Thus Proposition \ref{P:inner_product1} can be stated as
 \begin{Prop}\label{P:inner_product2}
 Assume that $\pi$ is as in Proposition \ref{P:inner_product1}. Then
 for $\phi_i \in \omega_{\psi, \chi_U,\chi_V}$ and $f_i \in \pi$, one has
  \[    \langle \theta_{n,r}(\phi_1, f_1) , \theta_{n,r}(\phi_2, f_2) \rangle 
 =  [E:F] \cdot   {\rm Val}_{s= s_{m,n}} Z(s,
 \Phi^{n,r}(\delta(\phi_1\otimes \overline{\phi}_2)), f_1, f_2),\]
where recall that $\Phi^{n,r}(\delta(\phi_1\otimes
\overline{\phi}_2))$ is the Siegel-Weil section associated with
$\delta(\phi_1\otimes \overline{\phi}_2)$.
 \end{Prop}
\vskip 5pt

\subsection{\bf Local zeta integrals and $L$-factors}
For ${\rm Re}(s) \gg 0$, if $\Phi = \otimes_v \Phi_v$ and $f_i =
\otimes_v f_{i,v}$ are pure tensors, one has an Euler product
\[ Z(s, \Phi, f_1, f_2)  = \prod_v Z_v(s, \Phi_v, f_{1,v}, f_{2,v}) \]
 where
 \[  Z_v(s, \Phi_v, f_{1,v}, f_{2,v}) = \int_{G(U_n)(F_v)} \Phi_v(
 g_v,1) \cdot \overline{ \langle \pi_v(g_v)  f_1, f_2 \rangle} 
 \, dg_v. \]
 By \cite{PS-R87}, $Z_v(s, \Phi_v, f_{1,v}, f_{2,v})$ extends to a meromorphic function on $\C$. 
 In \cite{LR}, Lapid-Rallis has defined the standard $L$-factors $L(s,
 \pi_v \times \chi_{V,v})$ using this family of local zeta integrals
 (see \cite{G} for the metaplectic case). In \cite{Y4}, Yamana has
 shown that $L(s + \frac{1}{2}, \pi_v \times \chi_{V,v})$ is precisely
 the GCD of the family of doubling zeta integrals associated to ``good"
 sections $\Phi_v(s)$. Moreover, when ${\rm Re}(s) \geq 0$, a section
 $\Phi_v(s)$ is good at $s$ if and only if it is holomorphic there.
  Thus, it is natural to define the normalized
 local zeta integral
 \[  
 Z^*_v(s, \Phi_v, f_{1,v}, f_{2,v}) := \frac{Z_v(s, \Phi_v, f_{1,v},
   f_{2,v})}{L(s+ \frac{1}{2}, \pi_v \times \chi_{V,v})}. \]
Then $ Z^*_v(s, \Phi_v, f_{1,v}, f_{2,v})$ is an entire function of
$s$ and for any $s_0$ such that ${\rm Re}(s_0) \geq 0$, one can find a
standard section $\Phi$ and $f_{i,v} \in \pi_v$ such that
 \[ Z^*_v(s_0, \Phi_v, f_{1,v}, f_{2,v}) \ne 0. \] 
Note that when every data involved is unramified (which is the case
 for almost all $v$), one has
 \[   Z_v(s, \Phi_v, f_{1,v}, f_{2,v}) = \frac{L(s+ \frac{1}{2}, \pi_v
   \times \chi_{V,v})}{d_v(s,\chi_v)} \]
 where $d_v(s,\chi_V)$ is the product of a number of explicit Hecke
 $L$-factors (see \cite[P. 334, Remark 3]{LR}). When $s > 0$, $d_v(s,
 \chi_V)$ has no poles and
  the Euler product $\prod_v d_v(s,\chi_V)$ is absolutely convergent,
  so that 
\begin{equation}\label{E:Z^*}
Z^*(s, \Phi, f_1, f_2) := \prod_v Z^*_v(s, \Phi_v, f_{1,v}, f_{2,v})
\end{equation}
 is absolutely convergent (and thus holomorphic)  when ${\rm Re}(s) > 0$. 
  \vskip 5pt

\subsection{\bf Rallis inner product.} Hence one has  the Rallis inner
product formula as follows.
\vskip 5pt

\begin{Thm}\label{T:Rallis}
  Suppose that $d(n) < m \leq 2\cdot d(n)$ and $r \leq n$. 
   Let $\pi$ be an irreducible cuspidal representation of $G(U_n)$
  and consider its global theta lift $\Theta_{n,r}(\pi)$ on
  $H(V_r)$. Assume that $\Theta_{n,j}(\pi) = 0$ for $j < r$, so that
  $\Theta_{n,r}(\pi)$   
 is cuspidal. 
 \vskip 5pt
 
 \noindent (i) For $\phi_1, \phi_2 \in \omega_{\psi, \chi_V, V_r,
   U_n}$ and $f_1, f_2 \in \pi$, we have
 \[  \langle \theta(\phi_1, f_1) , \theta(\phi_2, f_2) \rangle  
 =  [E:F]  \cdot \Val_{s=s_{m,n}} L(s + \frac{1}{2}, \pi \times \chi_V)  \cdot Z^*(s,
 \phi_1 \otimes \overline{\phi_2}, f_1, f_2), \]
 where 
 \[  s_{m,n} = \frac{m - n-\epsilon_0}{2}  > 0 ,\]
  and $Z^*(s, -)$ denotes the normalized doubling zeta integral as in (\ref{E:Z^*}).  

\vskip 5pt
\noindent (ii) Assume further that for all places $v$ of $F$, the
local theta lift $\Theta_{n,r}(\pi_v)$ is nonzero. Then $L(s+
\frac{1}{2},\pi \times \chi_V)$ is holomorphic at $s = s_{m,n}$, so that in the
context of (i),
 \[  
\langle \theta(\phi_1, f_1) , \theta(\phi_2, f_2) \rangle  
 =  [E:F]  \cdot L(s_{m,n} + \frac{1}{2}, \pi \times \chi_V)  \cdot Z^*(s_{m,n},
 \phi_1 \otimes \overline{\phi_2}, f_1, f_2).
 \]
In particular, the global theta lift $\Theta_{n,r}(\pi)$ is nonzero if and only if
 \begin{enumerate}
\item[(a)] for all places $v$,  $Z_v^*(s_{m,n})$ is nonzero on
  $R(V_{r,v}) \otimes \pi_v^{\vee} \otimes\pi_v$, and
\item[(b)] $L(s_{m,n} + \frac{1}{2}, \pi \times \chi_V)\ne  0$.  
\end{enumerate}

\end{Thm}

\begin{proof}
(i)  This follows immediately  from Proposition \ref{P:inner_product2}
together with the definition of $Z^*(s,-)$.

\vskip 5pt

\noindent (ii) It remains to prove that
$L(s+ \frac{1}{2},\pi \times \chi_V)$ is holomorphic at $s =
s_{m,n}$. This was essentially shown by Yamana in \cite[Lemma
10.1]{Y4}, but we give the proof here for the sake of completeness.
First assume that $\epsilon_0\neq -1$, so $G(U_n)$ is either symplectic or
unitary. By \cite[Theorem 10.1 (2)]{Y4} and the tower property of global theta
lifting, if the $L$-function $L(s +\frac{1}{2}, \pi \times \chi_V)$ is not holomorphic
at $s=s_{m,n}$, then there exists an $\epsilon$-Hermitian space
$V^{\#}_{r'}$ with $\chi_{V^\#_{r'}}=\chi_{V_r}$ and $\dim V_r+\dim V^{\#}_{r'}=2 \cdot d(n)$ such
that the global theta lift $\Theta_{n, V^{\#}_{r'}}(\pi)$ to
$H(V_{r'}^{\#})$ is non-vanishing. 
But since by our assumption
$\Theta_{n,j}(\pi)=0$ for all $j<r$, $V^{\#}_{r'}$ belongs to a
different Witt tower from $V_{r}$. Hence there is a place $v$ so that
$V_{r,v}$ and $V_{r,v}^{\#}$ belong to different Witt towers but
nonetheless the two local theta lifts $\Theta_{n,r}(\pi_v)$ and $\Theta_{n,
  V^{\#}_{r'}}(\pi_v)$ to $H(V_r)(F_v)$ and $H(V^{\#}_{r'})(F_v)$, respectively, are
  non-vanishing. But the conservation relation (see \cite{KR6}, \cite[Thm. 5.4]{GI2} and \cite{SZ}) 
  then implies that
\[ \dim V_{r,v}+\dim V^{\#}_{r',v}\geq 2d(n)+2 \]
which contradicts the fact that $V_{r,v}$ and $V^{\#}_{r',v}$ are of complementary dimension.
\vskip 5pt

Next assume $\epsilon_0=-1$, so that  $G(U_n)$ is orthogonal. Similarly to
the above, there exists an automorphic determinant character $\eta$ on
$G(U_n)(\A)$ such that $\Theta_{n,r'}(\pi\otimes\eta)\neq 0$ where $r'$
is such that $\dim V_r+\dim V_{r'}=2 \cdot d(n)$. By the same reasoning
as above, this would contradict the conservation relation.
 \vskip 5pt

\end{proof}
 
  \vskip 5pt

 \subsection{\bf Local nonvanishing.}
 To obtain the local-global criterion for the nonvanishing of
 $\Theta_{n,r}(\pi)$ by using the Rallis inner product formula we have
 obtained, one needs to understand the nonvanishing of local theta
 lifts in terms of the normalized local zeta integrals. In particular,
 we hope to have
\begin{Conj} \label{C:localnon0}
 The local theta lift $\Theta_{n,r}(\pi_v)$ is nonzero if and only if
 $Z_v^*(s_{m,n})$ is nonzero on $R(V_{r,v}) \otimes \pi_v^{\vee} \otimes
 \pi_v$ (where $0 < s_{m,n} \leq  d(n)/2$).
 \end{Conj}
 Note that the implication $(\Longleftarrow)$ is obvious, so  the
 content of the conjecture is the reverse implication. 
 \vskip 5pt

Unfortunately, we are not able to prove this conjecture in full
generality when the place $v$ is real. The best we know at this
moment is:
\vskip 5pt

\begin{Prop} \label{P:localnon0}
Suppose that $0< s_{m,n} \leq \frac{d(n)}{2}$.
\vskip 5pt

\noindent(i) Assume that one of the following conditions hold:
 \begin{itemize}
 \item $v$ is finite, or
 \item  $\epsilon_0 = -1$, or
 \item  $\epsilon_0 = 0$, $v$ is archimedean   and $E_v = F_v \times F_v$, or
 \item  $\epsilon_0 = 1$ and $F_v = \C$.
 \end{itemize}
 Then the local theta lift
 $\Theta_{n,r}(\pi_v)$ is nonzero if and only if $Z_v^*(s_{m,n})$ is
 nonzero on $R(V_{r,v}) \otimes \pi_v^{\vee} \otimes \pi_v$. 
 \vskip 5pt
 
 \noindent (ii) Assume that we are in a case not covered by (i), i.e.
 \begin{itemize}
 \item $\epsilon_0 = 0$, $F_v = \R$ and $E_v = \C$, or
 \item $\epsilon_0 = 1$, $F_v = \R$.
 \end{itemize}
 Suppose that the signature of $V_{r,v}$ is $(p,q)$ with $p+q = m$. 
 If  the local theta lift
 $\Theta_{n,r}(\pi_v)$ is nonzero, then there is an
 $\epsilon$-Hermitian space $V_v'$  over $E_v$
  such that
  \begin{itemize}
  \item[(a)]  the signature $(p',q')$ of $V_v'$ (with $p'+q' = m$) satisfies
\[   \begin{cases}
 p\equiv p' \mod 4 \text{  if  $\epsilon_0 = 1$} \\
 p \equiv p' \mod 2 \text{  if $\epsilon_0 = 0$.} \end{cases}  \]
 \item[(b)]  $Z_v^*(s_{m,n})$ is  nonzero on $R(V_v') \otimes \pi_v^{\vee} \otimes \pi_v$; 
 \end{itemize}
 When $m=d(n)+1$, one has $V'_v = V_{r,v}$.
  \end{Prop}
  \vskip 5pt

 \begin{proof}
(i) This follows readily from Proposition
 \ref{P:str} and its archimedean analog using the conservation relation \cite{SZ};  see \cite[Lemma
 8.6]{Y4}. 
 \vskip 5pt
 
 \noindent (ii) This follows from a more detailed knowledge of the
 module structure of $I_n^n(s_{m,n}, \chi_{V_r,v})$. 
 For $s_{m,n} \in X_n(\chi)$ and $s_{m,n} >0$, the typical
 module diagram of $I_n^n(s_{m,n}, \chi_{V_r,v})$ is illustrated in
 \cite[Figure 4]{LZ1}  if $\epsilon_0 = 0$ and in \cite[Figure 5]{LZ2} and \cite[\S 6.1, \S 6.2]{LZ4}
 if $\epsilon_0  =1$. Moreover, in these figures, one may read off the
 location of the submodules $R_n(V_v)$ for all $V_v$ such that $\dim V_v =
 \dim V_{r,v}$ and $\chi_{V_v} = \chi_{V_{r,v}}$ (if $\epsilon_0  =1$).
\vskip 5pt 

Consider the submodule 
\[  \Sigma = \bigoplus_{V'_v}  R_n(V'_v) \]
where the sum runs over all   $\epsilon$-Hermitian spaces $V_v'$  over
$E_v$ whose signature $(p',q')$ satisfies the condition (a) in the
Proposition. We need to show that $Z_v^*(s_{m,n})$ is  nonzero when
restricted to $\Sigma \otimes \pi_v^{\vee} \otimes \pi_v$.

 On examining the module diagram of $I_n^n(s_{m,n}, \chi_{V_{r,v}})$
 given in \cite[Figure 4]{LZ1}, \cite[Figure 5]{LZ2} and \cite[\S 6.1, \S 6.2]{LZ4}, one notes
 that 
 \[  I_n^n(s_{m,n}, \chi_{V_{r,v}})/ \Sigma \cong \bigoplus_{U_v}  R_n(U_v) \]
 where the sum $U_v$ runs over all $\epsilon$-Hermitian modules such that
 \begin{itemize}
 \item[-] $\dim U_v + \dim V_{r,v} = 2 \cdot d(n)$;
 \item[-] $\chi_{U_v} = \chi_{V_{r,v}}$ if $\epsilon_0 = 1$;
 \item[-] $U_v$ does not lie in the same Witt tower as the Hermitian spaces $V'_v$ which 
 occur in the definition of $\Sigma$, and in particular, $U_v$ lies in a different Witt tower as $V_{r,v}$.
 \end{itemize}
 Since $\Theta_{n,r}(\pi_v)$ is nonzero by hypothesis, the
 conservation relation shown in \cite{SZ} implies that the local theta lift of
 $\pi_v$ to any such $H(U_v)$ is zero. Thus  
 $Z_v^*(s_{m,n})$ must be  nonzero when restricted to $\Sigma \otimes
 \pi_v^{\vee} \otimes \pi_v$.  
 \vskip 5pt
 
 When $m = d(n) +1$, suppose for the sake of contradiction that $V'_v
 \ne V_{r,v}$. Then $V'_v$ and $V_{r,v}$ belong to two different Witt
 towers. Since the theta lifts of $\pi_v$ to $H(V_{r,v})$ and
 $H(V'_v)$ are both nonzero,  and the two Witt towers in question are
 not ``adjacent" in the language of \cite{SZ}, the conservation
 relation shown in \cite{SZ} implies that
 \[  \dim V_{r,v} + \dim V'_v > 2 d(n) +2. \]
 This contradicts the fact that $\dim V_{r,v} = \dim V'_v = d(n) +1$. 
 The Proposition is proved. 
  \end{proof}
 \vskip 5pt

  \subsection{\bf Nonvanishing of theta lifts.}
 Finally, we can state and prove our non-vanishing theorem:
 \begin{Thm}\label{T:main_nonvanishing}
 
  Assume that $d(n) < m \leq  2 \cdot d(n)$ and $r \leq n$.
Let $\pi$ be  an irreducible  cuspidal representation of $G(U_n)$ and consider its
global theta lift $\Theta_{n,r}(\pi)$ to $H(V_r)$. Assume that
$\Theta_{n,j}(\pi) = 0$ for $j < r$, so that  $\Theta_{n,r}(\pi)$   
 is cuspidal.  
 \vskip 5pt
 
 \noindent (i)  If $\Theta_{n,r}(\pi)$ is nonzero, then
 \begin{enumerate}
\item[(a)] for all places $v$, $\Theta_{n,r}(\pi_v) \ne 0$, and
\item[(b)] $L(s_{m,n} + \frac{1}{2}, \pi \times \chi_V)\ne  0$, i.e. nonzero holomorphic.  
\end{enumerate}
\vskip 10pt

\noindent (ii) The converse to (i) holds when one assumes one of the following conditions:
 \begin{itemize}
 \item  $\epsilon_0 = -1$;
 \item  $\epsilon_0 = 0$  and $E_v = F_v \times F_v$ for all archimedean places $v$ of $F$;
 \item  $\epsilon_0 = 1$ and $F$ is totally complex;
 \item $m = d(n)+1$.
 \end{itemize}
 
 \vskip 5pt
 
 \noindent (iii) In general, under the conditions (a) and (b) in (i),
 there is a $\epsilon$-Hermitian space $V'$ over $E$ such that 
 \begin{itemize}
 \item  $V' \otimes F_v \cong V_r \otimes F_v$  for every finite or complex place of $F$;
 \item  the global theta lift $\Theta_{U_n, V'}(\pi)$ of $\pi$ to $H(V')$ is nonzero.  
 \end{itemize}
 Under any one of the conditions of (ii), one may take $V'$ to be $V_r$.
\end{Thm}
  
\vskip 5pt

\begin{proof}
(i) This follows immediately from Theorem \ref{T:Rallis}(ii).
\vskip 5pt

\noindent (ii) This follows from Theorem \ref{T:Rallis}(ii) and Proposition \ref{P:localnon0}.
\vskip 5pt

\noindent (iii) For each real place $v$ of $F$ as in Proposition
\ref{P:localnon0}(ii), pick an $\epsilon$-Hermitian space $V'_v$ as
supplied by Proposition \ref{P:localnon0}(ii). The condition on the
signature of $V'_v$ ensures that there is an $\epsilon$-Hermitian
space $V'$ over $E$ such that $V' \otimes_F F_v \cong V'_v$ for these
places $v$, and 
$V' \otimes_F F_v \cong V_r \otimes_F F_v$ for all other places.   Now one
applies Theorem \ref{T:Rallis} with $V'$ in place of $V_r$ to
deduce the desired result.  
\end{proof}
 \vskip 5pt

 The reader may recall from \S \ref{SS:kr-reg} that the condition $r \leq n$ in the above theorem holds automatically, except when $\epsilon_ 0 = 1$ and  $m = 2r = 2n+2$, so that 
 \[  H(V_r) \cong \OO_{n+1, n+1} \, \, \text{(split)} \quad \text{and} \quad G(U_n) = \Sp_n \]
 and $n$ is necessarily even. There is, however, no harm in omitting this case from the discussion. Indeed, the dual pair $(G(U_n), H(V_{r}))$ is in the so-called stable range, so that one knows the global and local theta lifts in question are all nonzero. 
 In fact, even the pair $(G(U_n), H(V_{r-1}))$, which is the lower step of the tower, is in the stable range. Thus,  one sees that $\Theta_{n,r-1}(\pi) \ne 0$  and $\Theta_{n,r}(\pi) \ne 0$ is non-cuspidal.

\end{document}